\documentclass[11pt,letterpaper]{scrartcl}
\usepackage[T1]{fontenc}
\usepackage[utf8]{inputenc}
\usepackage{lmodern}
\usepackage[english]{babel}
\usepackage[english]{isodate}
\usepackage{amsthm}
\usepackage{amsmath}
\usepackage{amssymb}
\usepackage{amsfonts}
\usepackage{mathtools}
\usepackage{mathrsfs}
\usepackage{dsfont}
\usepackage{upgreek}
\usepackage{xpatch}
\usepackage{bm}
\usepackage{bbm}
\usepackage{enumitem}
\usepackage[toc,page]{appendix}
\usepackage{footnotebackref}
\usepackage[totalwidth=480pt, totalheight=680pt]{geometry}
\usepackage{hyperref}

\theoremstyle{remark}
\newtheorem{assumptions}{Assumptions}[section]
\newtheorem{remark}[assumptions]{Remark}

\theoremstyle{plain}
\newtheorem{theorem}[assumptions]{Theorem}
\newtheorem{proposition}[assumptions]{Proposition}
\newtheorem{lemma}[assumptions]{Lemma}
\newtheorem{corollary}[assumptions]{Corollary}

\addtokomafont{disposition}{\normalfont\scshape}

\makeatletter
\xpatchcmd{\@sec@pppage}{
\bfseries}{
\normalfont\scshape\Large}{}{}
\makeatother

\setkomafont{section}{\Large}
\setkomafont{subsection}{\large}

\numberwithin{equation}{section}

\begin{document}


\title{
\LARGE Trajectorial Otto calculus\thanks{
We thank Luigi Ambrosio, Mathias Beiglb{\"o}ck, Max Fathi, Ivan Gentil, David Kinderlehrer, Daniel Lacker, Michel Ledoux, Jan Maas, Felix Otto, Chris Rogers, Oleg Szehr, and Lane Chun Yeung for their advice and comments during the preparation of this paper. Special thanks go to Luigi Ambrosio, Michel Ledoux and Jan Maas for their expert guidance, which helped us navigate several difficult narrows successfully. \smallskip\newline
\noindent Preliminary versions of this work, under the titles ``Applying It\^{o} calculus to Otto calculus'' and ``Pathwise Otto calculus'', were posted on \href{https://arxiv.org}{arXiv} on $21$ November $2018$ (\href{https://arxiv.org/pdf/1811.08686v1.pdf}{arXiv:1811.08686v1}) and $6$ February $2019$ (\href{https://arxiv.org/pdf/1811.08686v2.pdf}{arXiv:1811.08686v2}), respectively. \smallskip\newline
\noindent I. Karatzas acknowledges support from the National Science Foundation (NSF) under grant NSF-DMS-14-05210. W. Schachermayer and B. Tschiderer acknowledge support by the Austrian Science Fund (FWF) under grant P28661. W. Schachermayer additionally appreciates support by the Vienna Science and Technology Fund (WWTF) through projects MA14-008 and MA16-021. \smallskip\newline
\noindent Most of this work was done in Fall 2018, when W. Schachermayer was visiting the Department of Mathematics at Columbia University as Minerva Foundation Fellow.
}
}


\author{
\large Ioannis Karatzas\thanks{
Department of Mathematics, Columbia University, 2990 Broadway, New York, NY 10027, USA 
\newline (email: \href{mailto: ikl@columbia.edu}{ikl@columbia.edu}); 
\newline and INTECH Investment Management, One Palmer Square, Suite 441, Princeton, NJ 08542, USA 
\newline (email: \href{mailto: ikaratzas@intechjanus.com}{ikaratzas@intechjanus.com}).
} 
\and \large Walter Schachermayer\thanks{
Faculty of Mathematics, University of Vienna, Oskar-Morgenstern-Platz 1, 1090 Vienna, Austria 
\newline (email: \href{mailto: walter.schachermayer@univie.ac.at}{\mbox{walter}.schachermayer@univie.ac.at}); 
\newline and Department of Mathematics, Columbia University, 2990 Broadway, New York, NY 10027, USA.
} 
\and \large Bertram Tschiderer\thanks{
Faculty of Mathematics, University of Vienna, Oskar-Morgenstern-Platz 1, 1090 Vienna, Austria \newline (email: \href{mailto: bertram.tschiderer@univie.ac.at}{bertram.tschiderer@univie.ac.at}).
}
}


\date{\normalsize 22nd March 2020}


\maketitle


\begin{abstract} \small \noindent \textsc{Abstract.} We revisit the variational characterization of diffusion as entropic gradient flux and provide for it a probabilistic interpretation based on stochastic calculus. It was shown by Jordan, Kinderlehrer, and Otto that, for diffusions of Langevin-Smoluchowski type, the Fokker-Planck probability density flow minimizes the rate of relative entropy dissipation, as measured by the distance traveled in the ambient space of probability measures with finite second moments, in terms of the quadratic Wasserstein metric. We obtain novel, stochastic-process versions of these features, valid along almost every trajectory of the diffusive motion in both the forward and, most transparently, the backward, directions of time, using a very direct perturbation analysis. By averaging our trajectorial results with respect to the underlying measure on path space, we establish the minimum rate of entropy dissipation along the Fokker-Planck flow and measure exactly the deviation from this minimum that corresponds to any given perturbation. As a bonus of our perturbation analysis we derive the so-called HWI inequality relating relative entropy (H), Wasserstein distance (W) and relative Fisher information (I).

\bigskip

\small \noindent \href{https://mathscinet.ams.org/mathscinet/msc/msc2010.html}{\textit{MSC 2010 subject classifications:}} Primary 60H30, 60G44; secondary 82C31, 60J60, 94A17

\bigskip

\small \noindent \textit{Keywords and phrases:} Relative entropy, Wasserstein distance, Fisher information, optimal transport, gradient flux, diffusion processes, time reversal, functional inequalities
\end{abstract}


\newpage


\tableofcontents


\newpage


\section{Introduction} \label{soti}


We provide a trajectorial interpretation of a seminal result by Jordan, Kinderlehrer, and Otto \cite{JKO98}, and present a proof based on stochastic calculus. The basic theme of our approach could be described epigrammatically as ``applying It\^{o} calculus to Otto calculus''. More precisely, we follow a stochastic analysis approach to Otto's characterization of diffusions of Langevin-Smoluchowski type as entropic gradient fluxes in Wasserstein space, and provide stronger, trajectorial versions of these results. For consistency and better readability we adopt the setting and notation of \cite{JKO98}, and even copy some paragraphs of this paper almost verbatim in the remainder of this introductory section. 

\smallskip

Following the lines of \cite{JKO98} we thus consider a Fokker-Planck equation of the form
\begin{equation} \label{fpeq}
\partial_{t} \rho(t,x) = \operatorname{div}\big(\nabla \Psi(x) \, \rho(t,x) \big) + \beta^{-1} \Delta \rho(t,x), \qquad (t,x) \in (0,\infty) \times \mathds{R}^{n},
\end{equation}
with initial condition
\begin{equation} \label{ic}
\rho(0,x) = \rho^{0}(x), \qquad x \in \mathds{R}^{n}.
\end{equation}
Here, $\rho$ is a real-valued function defined for $(t,x) \in [0,\infty) \times \mathds{R}^{n}$, the function $\Psi \colon \mathds{R}^{n} \rightarrow [0,\infty)$ is smooth and plays the role of a potential, $\beta > 0$ is a real constant, and $\rho^{0}$ is a probability density on $\mathds{R}^{n}$. The solution $\rho(t,x)$ of \hyperref[fpeq]{(\ref*{fpeq})} with initial condition \hyperref[ic]{(\ref*{ic})} stays non-negative and conserves its mass, which means that the spatial integral
\begin{equation}
\int_{\mathds{R}^{n}} \rho(t,x) \, \textnormal{d}x
\end{equation}
is independent of the time parameter $t \geqslant 0$ and is thus equal to $\int \rho^{0} \, \textnormal{d}x = 1$. Therefore, $\rho(t, \, \cdot \,)$ must be a probability density on $\mathds{R}^{n}$ for every fixed time $t \geqslant 0$.

\smallskip

As in \cite{JKO98} we note that the Fokker-Planck equation \hyperref[fpeq]{(\ref*{fpeq})} with initial condition \hyperref[ic]{(\ref*{ic})} is inherently related to the stochastic differential equation of Langevin-Smoluchowski type \cite{Fri75,Gar09,Ris96,Sch80} 
\begin{equation} \label{sdeid}
\textnormal{d}X(t) = - \nabla \Psi\big(X(t)\big) \, \textnormal{d}t + \sqrt{2 \beta^{-1}} \, \textnormal{d}W(t), \qquad X(0) = X^{0}.
\end{equation}
In the equation above, $(W(t))_{t \geqslant 0}$ is an $n$-dimensional Brownian motion started from $0$, and the $\mathds{R}^{n}$-valued random variable $X^{0}$ is independent of the process $(W(t))_{t \geqslant 0}$. The probability distribution of $X^0$ has density $\rho^{0}$ and, unless specified otherwise, the reference measure will always be Lebesgue measure on $\mathds{R}^{n}$. Then $\rho(t, \, \cdot \,)$, the solution of \hyperref[fpeq]{(\ref*{fpeq})} with initial condition \hyperref[ic]{(\ref*{ic})}, gives at any given time $t \geqslant 0$ the probability density function of the random variable $X(t)$ from \hyperref[sdeid]{(\ref*{sdeid})}. 

\smallskip

If the potential $\Psi$ grows rapidly enough so that $\mathrm{e}^{-\beta \Psi} \in L^{1}(\mathds{R}^{n})$, then the partition function 
\begin{equation} \label{czpf}
Z (\beta) = \int_{\mathds{R}^{n}} \mathrm{e}^{- \beta \Psi(x)} \, \textnormal{d}x
\end{equation}
is finite and there exists a unique stationary solution of the Fokker-Planck equation \hyperref[fpeq]{(\ref*{fpeq})}; namely, the probability density $\rho_{s}$ of the Gibbs distribution given by \cite{Gar09,JK96,Ris96}
\begin{equation} \label{usdotgdm}
\rho_{s}(x) = \big(Z(\beta)\big)^{-1} \, \mathrm{e}^{ - \beta \Psi(x)}
\end{equation}
for $x \in \mathds{R}^{n}$. When it exists, the probability measure on $\mathds{R}^{n}$ with density function $\rho_{s}$ is called \textit{Gibbs distribution}, and is the unique invariant measure for the Markov process $(X(t))_{t \geqslant 0}$ defined by the stochastic differential equation \hyperref[sdeid]{(\ref*{sdeid})}; see, e.g., \cite[Exercise 5.6.18, p.\ 361]{KS98}.

\smallskip

In \cite{JK96} it is shown that the stationary probability density $\rho_{s}$ satisfies the following variational principle: it minimizes the \textit{free energy functional}
\begin{equation} \label{fef}
F(\rho) = E(\rho) + \beta^{-1} \, S(\rho)
\end{equation}
over all probability densities $\rho$ on $\mathds{R}^{n}$. Here, the functional
\begin{equation} \label{penotgbef}
E(\rho) \vcentcolon = \int_{\mathds{R}^{n}} \Psi \rho \, \textnormal{d}x
\end{equation}
models the potential energy, whereas the internal energy is given by the negative of the Gibbs-Boltzmann entropy functional 
\begin{equation} \label{ie}
S(\rho) \vcentcolon = \int_{\mathds{R}^{n}} \rho \log \rho \, \textnormal{d}x.
\end{equation}

\smallskip

Similarly as in \cite[Theorem 5.1]{JKO98} we consider the following regularity assumptions.
\begin{assumptions} \label{osaojko} \
\begin{enumerate}[label=(\roman*)] 
\item \label{faosaojko} The potential $\Psi \colon \mathds{R}^{n} \rightarrow [0,\infty)$ is of class $\mathcal{C}^{\infty}(\mathds{R}^{n};[0,\infty))$.
\item \label{saosaojko} The distribution of $X(0)$ in \hyperref[sdeid]{(\ref*{sdeid})} has a probability density function $\rho^{0}(x)$ with respect to Lebesgue measure on $\mathds{R}^{n}$, which has finite second moment as well as finite free energy, i.e.,
\begin{equation} \label{ffecaoo}
\int_{\mathds{R}^{n}} \rho^{0}(x) \, \vert x \vert^{2} \, \textnormal{d}x < \infty  
\qquad \textnormal{ and } \qquad
F(\rho^{0}) \in \mathds{R}. 
\end{equation}
\end{enumerate}
\end{assumptions}

In \cite{JKO98} it is also assumed that the potential $\Psi$ satisfies, for some real constant $C > 0$, the bound $\vert \nabla \Psi \vert \leqslant C \, (\Psi + 1)$, which we do not need here. However, we shall impose the following, additional assumptions.

\begin{assumptions}[\textsf{Regularity assumptions for the trajectorial results of the present paper}] \label{sosaojkoia} In addition to conditions \hyperref[faosaojko]{\ref*{faosaojko}} and \hyperref[saosaojko]{\ref*{saosaojko}} of \hyperref[osaojko]{Assumptions \ref*{osaojko}}, we also impose that:
\begin{enumerate}[label=(\roman*)] 
\setcounter{enumi}{2}
\item \label{naltsaosaojko} The potential $\Psi$ satisfies, for some real constants $c \geqslant 0$ and $R \geqslant 0$, the drift (or coercivity) condition
\begin{equation} \label{tppstdc}
\big\langle x \, , \nabla \Psi(x) \big\rangle_{\mathds{R}^{n}} \geqslant - c \, \vert x \vert^{2}
\end{equation}
for all $x \in \mathds{R}^{n}$ with $\vert x \vert \geqslant R$.
\item \label{tsaosaojko} The potential $\Psi$ is sufficiently well-behaved to guarantee that the solution of \hyperref[sdeid]{(\ref*{sdeid})} is well-defined for all $t \geqslant 0$, and that the solution $(t,x) \mapsto \rho(t,x)$ of \hyperref[fpeq]{(\ref*{fpeq})} with initial condition \hyperref[ic]{(\ref*{ic})} is continuous and strictly positive on $(0,\infty) \times \mathds{R}^{n}$, differentiable with respect to the time variable $t$ for each $x \in \mathds{R}^{n}$, and smooth in the space variable $x$ for each $t > 0$. We also assume that the logarithmic derivative $(t,x) \mapsto \nabla \log \rho(t,x)$ is continuous on $(0,\infty) \times \mathds{R}^{n}$. For example, by requiring that all derivatives of $\Psi$ grow at most exponentially as $\vert x \vert$ tends to infinity, one may adapt the arguments from \cite{Rog85} showing that this is indeed the case.
\end{enumerate}
For the formulation of \hyperref[thetthre]{Theorem \ref*{thetthre}} we will need a vector field $\beta \colon \mathds{R}^{n} \rightarrow \mathds{R}^{n}$ which is the gradient of a potential $B \colon \mathds{R}^{n} \rightarrow \mathds{R}$ satisfying the following regularity assumption:
\begin{enumerate}[label=(\roman*)] 
\setcounter{enumi}{4} 
\item \label{naltsaosaojkos} The potential $B \colon \mathds{R}^{n} \rightarrow \mathds{R}$ is of class $\mathcal{C}^{\infty}(\mathds{R}^{n};\mathds{R})$ and has compact support. Consequently, its gradient $\beta \vcentcolon = \nabla B \colon \mathds{R}^{n} \rightarrow \mathds{R}^{n}$ is of class $\mathcal{C}^{\infty}(\mathds{R}^{n};\mathds{R}^{n})$ and again compactly supported. We also assume that the perturbed potential $\Psi + B$ satisfies condition \hyperref[tsaosaojko]{\ref*{tsaosaojko}}.
\end{enumerate}
\end{assumptions}

The \hyperref[sosaojkoia]{Assumptions \ref*{sosaojkoia}} are satisfied by typical convex potentials $\Psi$. They also accommodate examples such as double-well potentials of the form $\Psi(x) = (x^{2}-\alpha^{2})^{2}$ on the real line, for real constants $\alpha > 0$. Furthermore, they guarantee that the second-moment condition in \hyperref[ffecaoo]{(\ref*{ffecaoo})} propagates in time, i.e.,
\begin{equation} \label{1.12}
\int_{\mathds{R}^{n}} \rho(t,x) \, \vert x \vert^{2} \, \textnormal{d}x < \infty, \qquad t \geqslant 0;
\end{equation}
see \hyperref[fosmolds]{Lemma \ref*{fosmolds}} below. It is important to point out, that these assumptions do not rule out the case when the constant $Z(\beta)$ in \hyperref[czpf]{(\ref*{czpf})} is infinite; thus, they allow for cases (such as $\Psi = 0$) in which the stationary probability density function $\rho_{s}$ does not exist. In fact, in \cite{JKO98} the authors point out explicitly that, even when the stationary probability density $\rho_{s}$ is not defined, the free energy \textnormal{\hyperref[fef]{(\ref*{fef})}} of a density $\rho(t,x)$ satisfying the Fokker-Planck equation \hyperref[fpeq]{(\ref*{fpeq})} with initial condition \hyperref[ic]{(\ref*{ic})} can be defined, provided that the free energy $F(\rho^{0})$ is finite. 

\begin{assumptions}[\textsf{Regularity assumptions regarding the Wasserstein distance}] \label{sosaojkoianoew} In addition to conditions \hyperref[faosaojko]{\ref*{faosaojko}} -- \hyperref[naltsaosaojkos]{\ref*{naltsaosaojkos}} of \hyperref[sosaojkoia]{Assumptions \ref*{sosaojkoia}}, and in order to compute explicitly the metric derivative of the quadratic Wasserstein distance along the Fokker-Planck probability density flow, which is the purpose of \hyperref[stwt]{Section \ref*{stwt}}, we require that:
\begin{enumerate}[label=(\roman*)] 
\setcounter{enumi}{5} 
\item \label{nalwstasas} For every $t \geqslant 0$, there exists a sequence of functions $\big(\varphi_{m}(t, \cdot \,)\big)_{m \geqslant 1} \subseteq \mathcal{C}_{c}^{\infty}(\mathds{R}^{n};\mathds{R})$, whose gradients $\big( \nabla \varphi_{m}(t, \cdot \,)\big)_{m \geqslant 1}$ converge in $L^{2}(P(t))$ to the time-dependent velocity field $v(t, \, \cdot \,) = \nabla \varphi(t, \, \cdot \, )$ of gradient type as in \hyperref[tdvfvtx]{(\ref*{tdvfvtx})} with $\varphi(t,x) = - \Psi(x) - \frac{1}{2} \log \rho(t,x)$, as $m \rightarrow \infty$. Here, $P(t)$ denotes the probability measure on the Borel sets of $\mathds{R}^{n}$ with density $\rho(t, \, \cdot \,)$.
\end{enumerate}
\end{assumptions}

\begin{remark} The last-mentioned requirement guarantees, for every $t \geqslant 0$, that the time-dependent velocity field $v(t, \, \cdot \,)$ is an element of the tangent space of $\mathscr{P}_{2}(\mathds{R}^{n})$ at the point $P(t) \in \mathscr{P}_{2}(\mathds{R}^{n})$ in the sense of \cite[Definition 8.4.1]{AGS08}. For the details we refer to our \hyperref[stwt]{Section \ref*{stwt}}, in particular, the display \hyperref[tanpcvecup]{(\ref*{tanpcvecup})}. However, we do not know whether this condition \hyperref[nalwstasas]{\ref*{nalwstasas}} in \hyperref[sosaojkoianoew]{Assumptions \ref*{sosaojkoianoew}} is actually an additional requirement, or whether it is automatically satisfied in our setting. But as this issue only affects the Wasserstein distance, and has no relevance for our novel trajectorial results which constitute the main point of this work, we will not pursue this question here any further.

\smallskip

The condition \hyperref[nalwstasas]{\ref*{nalwstasas}} in \hyperref[sosaojkoianoew]{Assumptions \ref*{sosaojkoianoew}} is satisfied by simple potentials such as for example $\Psi \equiv 0$ or $\Psi(x) = \vert x \vert^{2}/4$. More generally, potentials with a curvature lower bound $\textnormal{Hess}(\Psi) \geqslant \kappa \, I_{n}$, for some $\kappa \in \mathds{R}$ (as in \hyperref[sndcbe]{(\ref*{sndcbe})} below), for instance the double-well potential $\Psi(x) = (x^{2}-\alpha^{2})^{2}$ on the real line, satisfy this condition; this follows from \cite[Theorem 10.4.13]{AGS08}. The above condition \hyperref[nalwstasas]{\ref*{nalwstasas}} is also satisfied, whenever $\int_{\mathds{R}^{n}} \varphi^{2}(t,x) \, \rho(t,x) \, \textnormal{d}x < \infty$ holds for all $t \geqslant 0$.
\end{remark}

\subsection{Preview}

We set up in \hyperref[snaas]{Section \ref*{snaas}} our model for the Langevin-Smoluchowski diffusion and introduce its fundamental quantities, such as the current and the invariant distribution of particles, the resulting likelihood ratio process, as well as the associated concepts of free energy, relative entropy, and relative Fisher information.

\smallskip

\hyperref[thetheoseccth]{Section \ref*{thetheoseccth}} presents our basic results. These include \hyperref[thetone]{Theorem \ref*{thetone}}, which computes in terms of the relative Fisher information the rate of relative entropy decay in the ambient Wasserstein space of probability density functions with finite second moment; as well as its ``perturbed'' counterpart, \hyperref[thetthre]{Theorem \ref*{thetthre}}. We compute explicitly the difference between these perturbed and unperturbed rates and show that it is always non-negative, in fact strictly positive unless the perturbation and the gradient of the log-likelihood ratio are collinear. This way, the Langevin-Smoluchowski diffusion emerges as the \textit{steepest descent} (or ``gradient flux'') of the relative entropy functional with respect to the Wasserstein metric.

\smallskip

We also show that both \hyperref[thetone]{Theorems \ref*{thetone}} and \hyperref[thetthre]{\ref*{thetthre}} follow as very simple consequences of their stronger, trajectorial versions, \hyperref[thetsix]{Theorems \ref*{thetsix}} and \hyperref[thetthretv]{\ref*{thetthretv}}, respectively. These latter are the main results of this work; they provide very detailed descriptions of the semimartingale dynamics for the relative entropy process, in both its ``pure'' and ``perturbed'' forms. Such descriptions are most transparent when time is \textit{reversed}, so we choose to present them primarily in this context. Several important consequences and ramifications of \hyperref[thetsix]{Theorems \ref*{thetsix}}, \hyperref[thetthretv]{\ref*{thetthretv}} are developed in \hyperref[subimpconq]{Subsections \ref*{subimpconq}} and \hyperref[ramifications]{\ref*{ramifications}}, including a derivation of the famous HWI inequality of Otto and Villani \cite{OV00, Vil03, Vil09} (see also Cordero-Erausquin \cite{CE02}) that relates relative entropy (H) to Wasserstein distance (W) and to relative Fisher information (I).  

\smallskip

Most of the detailed arguments and proofs are collected in \hyperref[ssgrodwudm]{Section \ref*{ssgrodwudm}} and in the appendices. In particular, \hyperref[atrod]{Appendix \ref*{atrod}} presents a completely self-contained account of time reversal for It\^{o} diffusion processes. The necessary background on optimal Wasserstein transport is recalled in \hyperref[stwt]{Section \ref*{stwt}}. 


\section{The stochastic approach} \label{snaas}


Thus far, we have been mostly quoting from \cite{JKO98}. We adopt now a more probabilistic point of view, and translate our setting into the language of stochastic processes and probability measures. For notational convenience, and without loss of generality, we fix the constant $\beta > 0$ to equal $2$, so that the stochastic differential equation \hyperref[sdeid]{(\ref*{sdeid})} becomes 
\begin{equation} \label{sdeids}
\textnormal{d}X(t) = - \nabla \Psi\big(X(t)\big) \, \textnormal{d}t + \textnormal{d}W(t), \qquad t \geqslant 0.
\end{equation}

\smallskip

Let $P(0)$ be a probability measure on the Borel sets of $\mathds{R}^{n}$ with density function $p^{0}(x) \vcentcolon = \rho^{0}(x)$. We shall study the stochastic differential equation \hyperref[sdeids]{(\ref*{sdeids})} with initial probability distribution $P(0)$.

\smallskip

While we do make an effort to follow the setting and notation of \cite{JKO98} as closely as possible, our notation here differs slightly from \cite{JKO98}. To conform with our probabilistic approach, we shall use from now onward the familiar letters $p^{0}$ and $p(0, \, \cdot \,)$ rather than $\rho^{0}$ and $\rho(0, \, \cdot \, )$.

\medskip

The initial probability measure $P(0)$ on $\mathds{R}^{n}$ with density function $p(0, \, \cdot \,)$, induces a probability measure $\mathds{P}$ on the path space $\Omega = \mathcal{C}(\mathds{R}_{+};\mathds{R}^{n})$ of $\mathds{R}^{n}$-valued continuous functions on $\mathds{R}_{+}=[0,\infty)$, under which the canonical coordinate process $(X(t,\omega))_{t \geqslant 0} = (\omega(t))_{t \geqslant 0}$ satisfies the stochastic differential equation \hyperref[sdeids]{(\ref*{sdeids})} with initial probability distribution $P(0)$. We shall denote by $P(t)$ the probability distribution of the random vector $X(t)$ under $\mathds{P}$, and by $p(t) \equiv p (t, \, \cdot \, )$ the corresponding probability density function, at each time $t \geqslant 0$. This function solves the equation \hyperref[fpeq]{(\ref*{fpeq})} with initial condition \hyperref[ic]{(\ref*{ic})}.

\smallskip

We shall see in \hyperref[fosmoldspola]{Appendix \ref*{fosmoldspola}} that, in conjunction with the second-moment condition in \hyperref[ffecaoo]{(\ref*{ffecaoo})}, the drift condition \hyperref[tppstdc]{(\ref*{tppstdc})} guarantees finite second moments of the probability density functions $p(t)$ at all times $t \geqslant 0$; equivalently, membership of the probability distribution $P(t)$ in the space $\mathscr{P}_{2}(\mathds{R}^{n})$ of definition \hyperref[ptrn]{(\ref*{ptrn})} in \hyperref[stwt]{Section \ref*{stwt}}, for all $t \geqslant 0$. This property also holds when the potential $\Psi$ is replaced by $\Psi + B$ as in condition \hyperref[naltsaosaojkos]{\ref*{naltsaosaojkos}} of \hyperref[sosaojkoia]{Assumptions \ref*{sosaojkoia}}; see \hyperref[fosmoldspv]{Lemma \ref*{fosmoldspv}}.

\begin{lemma} \label{fosmolds} Under the \textnormal{\hyperref[sosaojkoia]{Assumptions \ref*{sosaojkoia}}}, the Langevin-Smoluchowski diffusion equation \textnormal{\hyperref[sdeids]{(\ref*{sdeids})}} with initial distribution $P(0)$ admits a pathwise unique, strong solution, which satisfies $P(t) \in \mathscr{P}_{2}(\mathds{R}^{n})$ for all $t \geqslant 0$.
\end{lemma}

\smallskip

An important role will be played by the \textit{Radon-Nikod\'{y}m derivative}, or \textit{likelihood ratio process}, 
\begin{equation} \label{rndlr}
\frac{\textnormal{d}P(t)}{\textnormal{d}\mathrm{Q}}\big(X(t)\big) = \ell\big(t,X(t)\big), 
\qquad \textnormal{ where } \quad 
\ell(t,x) \vcentcolon = \frac{p(t,x)}{q(x)} = p(t,x) \, \mathrm{e}^{2 \Psi(x)}
\end{equation}
for $t \geqslant 0$ and $x \in \mathds{R}^{n}$. Here and throughout, we denote by $\mathrm{Q}$ the $\sigma$-finite measure on the Borel sets of $\mathds{R}^{n}$, whose density with respect to Lebesgue measure is
\begin{equation}
q(x) \vcentcolon = \mathrm{e}^{-2\Psi(x)}, \qquad x \in \mathds{R}^{n}.
\end{equation}
The \textit{relative entropy} and the \textit{relative Fisher information} (see, e.g., \cite{CT06}) of $P(t)$ with respect to this measure $\mathrm{Q}$, are defined respectively as
\begin{equation} \label{doref}
H\big( P(t) \, \vert \, \mathrm{Q} \big) \vcentcolon = \mathds{E}_{\mathds{P}}\big[ \log \ell\big(t,X(t)\big) \big] = \int_{\mathds{R}^{n}} \log \bigg(\frac{p(t,x)}{q(x)}\bigg) \, p(t,x) \, \textnormal{d}x, \qquad t \geqslant 0,
\end{equation}

\smallskip

\begin{equation} \label{rfi}
I\big( P(t) \, \vert \, \mathrm{Q}\big) \vcentcolon = \mathds{E}_{\mathds{P}}\Big[ \, \big\vert \nabla \log \ell\big(t,X(t)\big) \big\vert^{2} \, \Big] = \int_{\mathds{R}^{n}} \big\vert \nabla \log \ell(t,x) \big\vert^{2} \, p(t,x) \, \textnormal{d}x, \qquad t \geqslant 0.
\end{equation}

\smallskip

\begin{remark} Following the approach of \cite[Section 2]{Leo14}, we show in \hyperref[wdadotrewrttmq]{Appendix \ref*{wdadotrewrttmq}} that the relative entropy $H( P \, \vert \, \mathrm{Q})$ is well-defined for every probability measure $P$ in $\mathscr{P}_{2}(\mathds{R}^{n})$ and takes values in $(-\infty,\infty]$.
\end{remark}

The following well-known identity \hyperref[reeef]{(\ref*{reeef})} states that the relative entropy $H(P(t) \, \vert \, \mathrm{Q})$ is equal to the free energy $F(p(t, \, \cdot \,))$, up to a multiplicative factor of $2$, for all $t \geqslant 0$. In light of condition \hyperref[saosaojko]{\ref*{saosaojko}} in \hyperref[osaojko]{Assumptions \ref*{osaojko}}, this identity implies $H( P(0) \, \vert \, \mathrm{Q} ) \in \mathds{R}$, so the quantity in \hyperref[doref]{(\ref*{doref})} is finite for $t = 0$; thus, on account of \hyperref[stewrtpefitre]{(\ref*{stewrtpefitre})} below, finite also for $t > 0$.

\begin{lemma} \label{llreeef} Under the \textnormal{\hyperref[sosaojkoia]{Assumptions \ref*{sosaojkoia}}}, and along the curve of probability measures $(P(t))_{t \geqslant 0}$, the free energy functional in \textnormal{\hyperref[fef]{(\ref*{fef})}} and the relative entropy in \textnormal{\hyperref[doref]{(\ref*{doref})}} are related for each $t \geqslant 0$ through the equation
\begin{equation} \label{reeef} 
2 \, F\big(p(t, \, \cdot \,)\big) = H\big( P(t) \, \vert \, \mathrm{Q} \big).
\end{equation}
\begin{proof} Indeed,
\begingroup
\addtolength{\jot}{0.7em}
\begin{align}
\mathds{E}_{\mathds{P}}\big[ \log \ell\big(t,X(t)\big) \big] 
&= \mathds{E}_{\mathds{P}}\Big[ \log\Big( \mathrm{e}^{2 \Psi(X(t))} \, p\big(t,X(t)\big) \Big) \Big] = \mathds{E}_{\mathds{P}}\big[ 2 \, \Psi\big(X(t)\big) \big] + \mathds{E}_{\mathds{P}}\big[ \log p\big(t,X(t)\big) \big] \\
&= 2 \int_{\mathds{R}^{n}} \Psi(x) \, p(t,x) \, \textnormal{d}x + \int_{\mathds{R}^{n}} p(t,x) \, \log p(t,x) \, \textnormal{d}x,
\end{align}
\endgroup
which equals $2 \, F(p(t, \, \cdot \,))$.
\end{proof}
\end{lemma}

The identity \hyperref[reeef]{(\ref*{reeef})} shows that studying the decay of the free energy $F(p(t, \, \cdot \,))$, is equivalent to studying the decay of the relative entropy $H( P(t) \, \vert \, \mathrm{Q})$, a key aspect of thermodynamics.

\begin{remark} In conjunction with \hyperref[reeef]{(\ref*{reeef})}, the condition $F(p(0, \, \cdot \,))  \in \mathds{R}$ in \hyperref[ffecaoo]{(\ref*{ffecaoo})}, and \hyperref[fef]{(\ref*{fef})} -- \hyperref[ie]{(\ref*{ie})}, the decrease of the relative entropy established in \hyperref[stewrtpefitre]{(\ref*{stewrtpefitre})} shows that $\mathds{E}_{\mathds{P}}\big[ \Psi(X(t)) \big]$ is finite for all $t \geqslant 0$. Thus, if the potential $\Psi$ dominates a quadratic, we deduce that $\mathds{E}_{\mathds{P}}\big[ \vert X(t) \vert^{2} \big] < \infty$, i.e., $P(t) \in \mathscr{P}_{2}(\mathds{R}^{n})$, also holds for all $t \geqslant 0$,  without invoking the coercivity condition \hyperref[tppstdc]{(\ref*{tppstdc})}. But of course, \hyperref[tppstdc]{(\ref*{tppstdc})} accommodates functions, such as $\Psi \equiv 0$, that fail to dominate a quadratic.
\end{remark}


\section{The theorems} \label{thetheoseccth}


As already indicated in \hyperref[fpeq]{(\ref*{fpeq})} and \hyperref[sdeid]{(\ref*{sdeid})}, the probability density function $p \colon [0,\infty) \times \mathds{R}^{n} \rightarrow [0,\infty)$ solves the Fokker-Planck or forward Kolmogorov \cite{Kol31} equation \cite{Fri75,Gar09,Ris96,Sch80}
\begin{equation} \label{fpeqnwfp}
\partial_{t} p(t,x) = \operatorname{div}\big(\nabla \Psi(x) \, p(t,x) \big) + \tfrac{1}{2} \Delta p(t,x), \qquad (t,x) \in (0,\infty) \times \mathds{R}^{n},
\end{equation}
with initial condition
\begin{equation} \label{icnwfp}
p(0,x) = p^{0}(x), \qquad x \in \mathds{R}^{n}.
\end{equation}
By contrast, the function $q(\, \cdot \,)$ does not depend on the temporal variable, and solves the stationary version of the forward Kolmogorov equation \hyperref[fpeqnwfp]{(\ref*{fpeqnwfp})}, namely
\begin{equation} \label{svfpeqnwfp}
0 = \operatorname{div}\big(\nabla \Psi(x) \, q(x) \big) + \tfrac{1}{2} \Delta q(x), \qquad x \in \mathds{R}^{n}.
\end{equation}

\bigskip
 
In light of \hyperref[llreeef]{Lemma \ref*{llreeef}}, the object of interest in \cite{JKO98} is to relate the decay of the \textit{relative entropy functional}
\begin{equation} \label{dotref}
\mathscr{P}_{2}(\mathds{R}^{n}) \ni P \longmapsto H(P \, \vert \, \mathrm{Q}) \in (-\infty,\infty]
\end{equation}
along the curve $(P(t))_{t \geqslant 0}$, to the quadratic Wasserstein distance $W_{2}$ defined in \hyperref[doqwd]{(\ref*{doqwd})} of \hyperref[stwt]{Section \ref*{stwt}}. We resume the remarkable relation between these two quantities in the following two theorems; these provide a way to quantify the relationship between displacement in the ambient space (the denominator of the expression in \hyperref[ttofeosl]{(\ref*{ttofeosl})}) and fluctuations of the free energy, or equivalently of the relative entropy (the numerator in the expression \hyperref[ttofeosl]{(\ref*{ttofeosl})}). The proofs will be given in \hyperref[subimpconq]{Subsection \ref*{subimpconq}} below.

\medskip

\begin{theorem} \label{thetone} Under the \textnormal{\hyperref[sosaojkoianoew]{Assumptions \ref*{sosaojkoianoew}}}, the relative Fisher information $I( P(t_{0}) \, \vert \, \mathrm{Q})$ is finite for Lebesgue-almost every $t_{0} \geqslant 0$, and we have the \textnormal{\textsf{generalized de Bruijn identity}}
\begin{equation} \label{flffflnv}
\lim_{t \rightarrow t_{0}} \, \frac{H\big( P(t) \, \vert \, \mathrm{Q} \big) - H\big( P(t_{0}) \, \vert \, \mathrm{Q}\big)}{t-t_{0}} 
= - \tfrac{1}{2} \, I\big( P(t_{0}) \, \vert \, \mathrm{Q}\big),
\end{equation}
as well as the local behavior of the quadratic Wasserstein distance
\begin{equation} \label{agswtuffasihnv}
\lim_{t \rightarrow t_{0}} \, \frac{W_{2}\big(P(t),P(t_{0})\big)}{\vert t - t_{0} \vert} 
= \tfrac{1}{2} \, \sqrt{I\big( P(t_{0}) \, \vert \, \mathrm{Q}\big)},
\end{equation}
so that
\begin{equation} \label{ttofeosl}
\lim_{t \rightarrow t_{0}} \, \Bigg( \operatorname{sgn}(t-t_{0}) \cdot \frac{H\big( P(t) \, \vert \, \mathrm{Q} \big) - H\big( P(t_{0}) \, \vert \, \mathrm{Q}\big) }{W_{2}\big( P(t),P(t_{0})\big)} \Bigg)
= - \sqrt{I\big( P(t_{0}) \, \vert \, \mathrm{Q}\big)}.
\end{equation}
Furthermore, if $t_{0} \geqslant 0$ is chosen so that the generalized de Bruijn identity \textnormal{\hyperref[flffflnv]{(\ref*{flffflnv})}} does hold, then the limiting assertions \textnormal{\hyperref[agswtuffasihnv]{(\ref*{agswtuffasihnv})}} and \textnormal{\hyperref[ttofeosl]{(\ref*{ttofeosl})}} are also valid.
\end{theorem}

\medskip

The ratio on the left-hand side of \hyperref[ttofeosl]{(\ref*{ttofeosl})} can be interpreted as the slope of the relative entropy functional \hyperref[dotref]{(\ref*{dotref})} at $P = P(t_{0})$ along the curve $(P(t))_{t \geqslant 0}$, if we measure distances in the ambient space $\mathscr{P}_{2}(\mathds{R}^{n})$ of probability measures by the quadratic Wasserstein distance $W_{2}$ of \hyperref[doqwd]{(\ref*{doqwd})}. The quantity appearing on the right-hand side of \hyperref[ttofeosl]{(\ref*{ttofeosl})} is the square root of the relative Fisher information in \hyperref[rfi]{(\ref*{rfi})}, written more explicitly in terms of the ``score function'' $\nabla \ell(t, \, \cdot \,) / \ell(t, \, \cdot \,)$ as
\begin{equation} \label{merfi}
I\big( P(t_{0}) \, \vert \, \mathrm{Q}\big) = \mathds{E}_{\mathds{P}}\Bigg[ \ \frac{\big\vert \nabla \ell\big(t_{0},X(t_{0})\big) \big\vert^{2}}{\ell\big(t_{0},X(t_{0})\big)^{2}} \ \Bigg] 
= \int_{\mathds{R}^{n}} \bigg\vert \frac{\nabla p(t_{0},x)}{p(t_{0},x)} + 2 \, \nabla \Psi(x) \bigg\vert^{2} \, p(t_{0},x) \, \textnormal{d}x.
\end{equation}

\begin{remark} \label{rethpoitwooslda} Under the \hyperref[sosaojkoia]{Assumptions \ref*{sosaojkoia}} it is perfectly possible for the relative Fisher information $I(P(t_{0}) \, \vert \, \mathrm{Q})$ to be infinite at $t_{0} = 0$. For instance, think of $p(0, \, \cdot \,)$ as the indicator function of a subset of $\mathds{R}^{n}$ with Lebesgue measure equal to $1$. 

\smallskip 

For future reference, we denote by $N$ the set of exceptional points $t_{0} \geqslant 0$ for which the right-sided limiting assertion 
\begin{equation} \label{rgtsflffflnv}
\lim_{t \downarrow t_{0}} \, \frac{H\big( P(t) \, \vert \, \mathrm{Q} \big) - H\big( P(t_{0}) \, \vert \, \mathrm{Q}\big)}{t-t_{0}} 
= - \tfrac{1}{2} \, I\big( P(t_{0}) \, \vert \, \mathrm{Q}\big)
\end{equation}
fails. According to \hyperref[thetone]{Theorem \ref*{thetone}}, this exceptional set $N$ has zero Lebesgue measure.
\end{remark}

\bigskip

The remarkable insight of \cite{JKO98} states that the slope in \hyperref[ttofeosl]{(\ref*{ttofeosl})} in the direction of the curve $(P(t))_{t \geqslant 0}$ is, in fact, the slope of \textit{steepest descent} for the relative entropy functional \hyperref[dotref]{(\ref*{dotref})} at the point $P = P(t_{0})$. To formalize this assertion, we fix a time $t_{0} \geqslant 0$ and let the vector field $\beta = \nabla B \colon \mathds{R}^{n} \rightarrow \mathds{R}^{n}$ be the gradient of a potential $B$ of class $\mathcal{C}^{\infty}(\mathds{R}^{n};\mathds{R})$ with compact support, as in condition \hyperref[naltsaosaojkos]{\ref*{naltsaosaojkos}} of \hyperref[sosaojkoia]{Assumptions \ref*{sosaojkoia}}. This gradient vector field $\beta$ will serve as a perturbation, and we consider the thus perturbed Fokker-Planck equation 
\begin{equation} \label{pfpeq}
\partial_{t} p^{\beta}(t,x) = \operatorname{div}\Big(\big(\nabla \Psi(x) + \beta(x) \big) \, p^{\beta}(t,x) \Big) + \tfrac{1}{2} \Delta p^{\beta}(t,x), \qquad (t,x) \in (t_{0},\infty) \times \mathds{R}^{n}
\end{equation}
with initial condition
\begin{equation} \label{pic}
p^{\beta}(t_{0},x) = p(t_{0},x), \qquad x \in \mathds{R}^{n}.
\end{equation}
We denote by $\mathds{P}^{\beta}$ the probability measure on the path space $\Omega = \mathcal{C}([t_{0},\infty);\mathds{R}^{n})$, under which the canonical coordinate process $(X(t))_{t \geqslant t_{0}}$ satisfies the stochastic differential equation 
\begin{equation} \label{wpsdeids}
\textnormal{d}X(t) = - \Big( \nabla \Psi\big(X(t)\big) + \beta\big(X(t)\big) \Big) \, \textnormal{d}t + \textnormal{d}W^{\beta}(t), \qquad t \geqslant t_{0}
\end{equation}
with initial probability distribution $P(t_{0})$. Here, the process $(W^{\beta}(t))_{t \geqslant t_{0}}$ is Brownian motion under $\mathds{P}^{\beta}$. The probability distribution of $X(t)$ under $\mathds{P}^{\beta}$ on $\mathds{R}^{n}$ will be denoted by $P^{\beta}(t)$, for $t \geqslant t_{0}$; once again, the corresponding probability density function $p^{\beta}(t) \equiv p^{\beta}(t, \, \cdot \, )$ solves the equation \hyperref[pfpeq]{(\ref*{pfpeq})} subject to the initial condition \hyperref[pic]{(\ref*{pic})}.

\smallskip

In the following analogue of \hyperref[fosmolds]{Lemma \ref*{fosmolds}}, we state that the perturbed probability density functions $p^{\beta}(t, \, \cdot \,)$ of \hyperref[pfpeq]{(\ref*{pfpeq})}, \hyperref[pic]{(\ref*{pic})} also admit finite second moments at all times $t \geqslant t_{0}$. For the proof we refer again to \hyperref[fosmoldspola]{Appendix \ref*{fosmoldspola}}.

\begin{lemma} \label{fosmoldspv} Under the \textnormal{\hyperref[sosaojkoia]{Assumptions \ref*{sosaojkoia}}}, let $t_{0} \geqslant 0$. Then the perturbed diffusion equation \textnormal{\hyperref[wpsdeids]{(\ref*{wpsdeids})}} with initial distribution $P^{\beta}(t_{0}) = P(t_{0})$ admits a pathwise unique, strong solution, which satisfies $P^{\beta}(t) \in \mathscr{P}_{2}(\mathds{R}^{n})$ for all $t \geqslant t_{0}$.
\end{lemma}

\smallskip

After these preparations we can state the result formalizing the gradient flux, or \textit{steepest descent}, property of the flow $(P(t))_{t \geqslant 0}$ generated by the Langevin-Smoluchowski diffusion \textnormal{\hyperref[sdeids]{(\ref*{sdeids})}} in the ambient space of probability measures $\mathscr{P}_{2}(\mathds{R}^{n})$ endowed with the Wasserstein metric.

\begin{theorem} \label{thetthre} Under the \textnormal{\hyperref[sosaojkoianoew]{Assumptions \ref*{sosaojkoianoew}}}, the following assertions hold for every point $t_{0} \geqslant 0$ at which the right-sided limiting identity \textnormal{\hyperref[rgtsflffflnv]{(\ref*{rgtsflffflnv})}} is valid \textnormal{(}i.e., every $t_{0} \in \mathds{R}_{+} \setminus N$\textnormal{):}

\smallskip

\noindent The $\mathds{R}^{n}$-valued random vectors 
\begin{equation} \label{ttrvzo}
a \vcentcolon = \nabla \log \ell\big(t_{0},X(t_{0})\big) = \nabla \log p\big(t_{0},X(t_{0})\big) + 2 \, \nabla \Psi\big( X(t_{0}) \big) \, , \, \qquad b \vcentcolon = \beta\big(X(t_{0})\big)    
\end{equation}
are elements of the Hilbert space $L^{2}(\mathds{P})$, and the perturbed version of the generalized de Bruijn identity \textnormal{\hyperref[flffflnv]{(\ref*{flffflnv})}} reads  
\begin{equation} \label{tatpvotgdbi}
\lim_{t \downarrow t_{0}} \, \frac{H\big( P^{\beta}(t) \, \vert \, \mathrm{Q} \big) - H\big( P^{\beta}(t_{0}) \, \vert \, \mathrm{Q}\big)}{t-t_{0}} 
= - \tfrac{1}{2} \, I\big( P(t_{0}) \, \vert \, \mathrm{Q}\big) - \langle a,b \rangle_{L^{2}(\mathds{P})} = -  \tfrac{1}{2} \, \big\langle a , a + 2b \big\rangle_{L^{2}(\mathds{P})}.
\end{equation}
Furthermore, the local behavior of the quadratic Wasserstein distance \textnormal{\hyperref[agswtuffasihnv]{(\ref*{agswtuffasihnv})}} in this perturbed context is given by
\begin{equation} \label{ftlbotwditpc}
\lim_{t \downarrow t_{0}} \, \frac{W_{2}\big( P^{\beta}(t),P^{\beta}(t_{0})\big)}{t-t_{0}}
=  \tfrac{1}{2}  \, \| a + 2 b\|_{L^{2}(\mathds{P})}.
\end{equation}
Combining \textnormal{\hyperref[tatpvotgdbi]{(\ref*{tatpvotgdbi})}} with \textnormal{\hyperref[ftlbotwditpc]{(\ref*{ftlbotwditpc})}}, and assuming $\| a + 2b\|_{L^{2}(\mathds{P})} > 0$, we have
\begin{equation}
\lim_{t \downarrow t_{0}} \, \frac{H\big( P^{\beta}(t) \, \vert \, \mathrm{Q} \big) - H\big( P^{\beta}(t_{0}) \, \vert \, \mathrm{Q}\big)}{W_{2}\big( P^{\beta}(t),P^{\beta}(t_{0})\big)} 
= - \Bigg\langle a \, , \, \frac{a+2b}{\| a + 2b\|_{L^{2}(\mathds{P})}} \Bigg\rangle_{L^{2}(\mathds{P})} \, ,
\end{equation}
and therefore
\begingroup
\addtolength{\jot}{0.7em}
\begin{align} 
\lim_{t \downarrow t_{0}} \, \Bigg( \, \frac{H\big( P^{\beta}(t) \, \vert \, \mathrm{Q} \big) - H\big( P^{\beta}(t_{0}) \, \vert \, \mathrm{Q}\big)}{W_{2}\big( P^{\beta}(t),P^{\beta}(t_{0})\big)} \, &- \, 
\frac{H\big( P(t) \, \vert \, \mathrm{Q} \big) - H\big( P(t_{0}) \, \vert \, \mathrm{Q}\big)}{W_{2}\big( P(t),P(t_{0})\big)} \, \Bigg) \label{nlwpthtt} \\
&= \|a\|_{L^{2}(\mathds{P})}  - \Bigg\langle a \, , \, \frac{a+2b}{\| a + 2b\|_{L^{2}(\mathds{P})}} \Bigg\rangle_{L^{2}(\mathds{P})} \, . \label{slnlwpthtt}
\end{align}
\endgroup
\end{theorem}

\smallskip

\begin{remark} On the strength of the Cauchy-Schwarz inequality, the expression in \hyperref[slnlwpthtt]{(\ref*{slnlwpthtt})} is non-negative, and vanishes if and only if $a$ and $b$ are collinear. Consequently, when the vector field $\beta$ is not a scalar multiple of $\nabla \log \ell(t_{0},\, \cdot \,)$, the difference of the two slopes in \hyperref[nlwpthtt]{(\ref*{nlwpthtt})} is strictly positive. In other words, the slope quantified by the first term of the difference \hyperref[nlwpthtt]{(\ref*{nlwpthtt})}, is then strictly bigger than the (negative) slope expressed by the second term of \hyperref[nlwpthtt]{(\ref*{nlwpthtt})}.
\end{remark}

\medskip

These two theorems are essentially well known. They build upon a vast amount of previous work. In the quadratic case $\Psi(x) = \vert x \vert^{2}/4$, i.e., when the process $(X(t))_{t \geqslant 0}$ in \hyperref[sdeids]{(\ref*{sdeids})} is Ornstein-Uhlenbeck with invariant measure in \hyperref[usdotgdm]{(\ref*{usdotgdm})} standard Gaussian, the relation
\begin{equation} \label{dbirbtfimsf}
\tfrac{\textnormal{d}}{\textnormal{d}t} \, H\big( P(t) \, \vert \, \mathrm{Q} \big) = - \tfrac{1}{2} \, I\big( P(t) \, \vert \, \mathrm{Q}\big)
\end{equation}
has been known since \cite{Sta59} as \textit{de Bruijn's identity}. This relationship between the two fundamental information measures, due to Shannon and Fisher, respectively, is a dominant theme in many aspects of information theory and probability. We refer to the book \cite{CT06} by Cover and Thomas for an account of the results by Barron, Blachman, Brown, Linnik, R\'{e}nyi, Shannon, Stam and many others in this vein, as well as to the paper \cite{MV00} by Markowich and Villani, and the book \cite{Vil03} by Villani. See also the paper by Carlen and Soffer \cite{CS91} and the book by Johnson \cite{Joh04} on the relation of \hyperref[dbirbtfimsf]{(\ref*{dbirbtfimsf})} to the central limit theorem. For the connections with large deviations we refer to \cite{ADPZ13} and \cite{Fat16}.

\smallskip

In \hyperref[flffflnv]{(\ref*{flffflnv})}, the de Bruijn identity \hyperref[dbirbtfimsf]{(\ref*{dbirbtfimsf})} is established for more general measures $\mathrm{Q}$, those that satisfy \hyperref[sosaojkoia]{Assumptions \ref*{sosaojkoia}}; in a similar vein, see also the seminal work \cite{BE85} by Bakry and {\'E}mery. 

\smallskip

The paper \cite{JKO98} broke new ground in this respect, as it considered a general potential $\Psi$ and established the relation to the quadratic Wasserstein distance, culminating with the characterization of $(P(t))_{t \geqslant 0}$ as a gradient flux. This relation was further investigated by Otto in the paper \cite{Ott01}, where the theory now known as \textit{``Otto calculus''} was developed. For a recent application of Otto calculus to the Schr\"odinger problem, see \cite{GLR20}.

\smallskip

The statements of our \hyperref[thetone]{Theorems \ref*{thetone}}, \hyperref[thetthre]{\ref*{thetthre}} complement the existing results in some important details, e.g., the precise form \hyperref[slnlwpthtt]{(\ref*{slnlwpthtt})}, measuring the difference of the two slopes appearing in \hyperref[nlwpthtt]{(\ref*{nlwpthtt})}. The main novelty of our approach, however, will only become apparent with the formulation of \hyperref[thetsix]{Theorems \ref*{thetsix}}, \hyperref[thetthretv]{\ref*{thetthretv}} below, the trajectorial versions of \hyperref[thetone]{Theorems \ref*{thetone}} and \hyperref[thetthre]{\ref*{thetthre}}. 

\bigskip

We shall thus investigate \hyperref[thetone]{Theorems \ref*{thetone}} and \hyperref[thetthre]{\ref*{thetthre}} in a trajectorial fashion, by considering the \textit{relative entropy process} 
\begin{equation} \label{stlrpd}
\log \ell\big(t,X(t)\big) = \log \Bigg( \frac{p\big(t,X(t)\big)}{q\big(X(t)\big)} \Bigg) = \log p\big(t,X(t)\big) + 2 \, \Psi\big(X(t)\big) \, , \qquad t \geqslant 0
\end{equation}
along each trajectory of the canonical coordinate process $(X(t))_{t \geqslant 0}$, and calculating its dynamics (stochastic differential) under the probability measure $\mathds{P}$. The expectation with respect to $\mathds{P}$ of this quantity is, of course, the relative entropy in \hyperref[doref]{(\ref*{doref})}.

\medskip

A decisive tool in the analysis of the relative entropy process \hyperref[stlrpd]{(\ref*{stlrpd})} is to reverse time, and use a remarkable insight due to Fontbona and Jourdain \cite{FJ16}. These authors consider the canonical coordinate process $(X(t))_{0 \leqslant t \leqslant T}$ on the path space $\Omega = \mathcal{C}([0,T];\mathds{R}^{n})$ in the \textit{reverse direction} of time, i.e., they work with the time-reversed process $(X(T-s))_{0 \leqslant s \leqslant T}$; it is then notationally convenient to consider a finite time interval $[0,T]$, rather than $\mathds{R}_{+}$. Of course, this does not restrict the generality of the arguments.

\medskip

At this stage it becomes important to specify the relevant filtrations: We denote by $(\mathcal{F}(t))_{t \geqslant 0}$ the smallest continuous filtration to which the canonical coordinate process $(X(t))_{t \geqslant 0}$ is adapted. That is, modulo $\mathds{P}$-augmentation, we have
\begin{equation} \label{fftbfmgtmt}
\mathcal{F}(t) = \sigma \big( X(u) \colon \, 0 \leqslant u \leqslant t \big), \qquad t \geqslant 0;
\end{equation} 
and we call $(\mathcal{F}(t))_{t \geqslant 0}$ the ``filtration generated by $(X(t))_{t \geqslant 0}$''. Likewise, we let $(\mathcal{G}(T-s))_{0 \leqslant s \leqslant T}$ be the ``filtration generated by the time-reversed canonical coordinate process $(X(T-s))_{0 \leqslant s \leqslant T}$'' in the same sense as before. In particular,
\begin{equation} \label{tbfmgtmt}
\mathcal{G}(T-s) = \sigma \big( X(T-u) \colon \, 0 \leqslant u \leqslant s \big), \qquad 0 \leqslant s \leqslant T,
\end{equation}
modulo $\mathds{P}$-augmentation. For the necessary measure-theoretic operations that ensure the continuity (from both left and right) of filtrations associated with continuous processes, consult Section 2.7 in \cite{KS98}; in particular, Problems 7.1 -- 7.6 and Proposition 7.7.

\subsection{Main results}

The following two \hyperref[thetsix]{Theorems \ref*{thetsix}} and \hyperref[thetthretv]{\ref*{thetthretv}} are the main new results of this paper. They can be regarded as trajectorial versions of \hyperref[thetone]{Theorems \ref*{thetone}} and \hyperref[thetthre]{\ref*{thetthre}}, whose proofs will follow from \hyperref[thetsix]{Theorems \ref*{thetsix}} and \hyperref[thetthretv]{\ref*{thetthretv}} simply by taking expectations. Similar trajectorial approaches have already been applied successfully to the theory of optimal stopping in \cite{DK94}, to Doob's martingale inequalities in \cite{ABPST13}, and to the Burkholder-Davis-Gundy inequality in \cite{BS15}. 

\bigskip 

The significance of \hyperref[thetsix]{Theorem \ref*{thetsix}} right below, is that the trade-off between the decay of relative entropy and the ``Wasserstein transportation cost'', both of which are characterized in terms of the cumulative relative Fisher information process, is valid not only in expectation, but also along (almost) each trajectory, provided we run time in the reverse direction.\footnote{As David Kinderlehrer kindly pointed out to the second named author, the implicit Euler scheme used in \cite{JKO98} also reflects the idea of going back in time at each step of the discretization.}

\begin{theorem} \label{thetsix} Under the \textnormal{\hyperref[sosaojkoia]{Assumptions \ref*{sosaojkoia}}}, we let $T > 0$ and define the cumulative relative Fisher information process, accumulated from the right, as
\begingroup
\addtolength{\jot}{0.7em}
\begin{equation} \label{fispdaftr}
\begin{aligned} 
F(T-s) \vcentcolon =& \int_{0}^{s} \frac{1}{2} \frac{\big\vert \nabla \ell\big(T-u,X(T-u)\big) \big\vert^{2}}{\ell\big(T-u,X(T-u)\big)^{2}} \, \textnormal{d}u \\
=& \int_{0}^{s} \frac{1}{2} \bigg\vert \frac{\nabla p\big(T-u,X(T-u)\big)}{p\big(T-u,X(T-u)\big)} + 2 \, \nabla \Psi\big(X(T-u)\big) \bigg\vert^{2} \, \textnormal{d}u 
\end{aligned}
\end{equation}
\endgroup
for $0 \leqslant s \leqslant T$. Then $\mathds{E}_{\mathds{P}}\big[F(0)\big] = \frac{1}{2} \int_{0}^{T} I\big(P(t) \, \vert \, \mathrm{Q}\big) \, \textnormal{d} t < \infty$, and the process
\begin{equation} \label{dollafipim}
M(T-s) \vcentcolon = \Big( \log \ell\big(T-s,X(T-s)\big) - \log \ell\big(T,X(T)\big) \Big) - F(T-s)  
\end{equation}
for $0 \leqslant s \leqslant T$, is a square-integrable martingale of the backwards filtration $(\mathcal{G}(T-s))_{0 \leqslant s \leqslant T}$ under the probability measure $\mathds{P}$. More explicitly, the martingale of \textnormal{\hyperref[dollafipim]{(\ref*{dollafipim})}} can be represented as 
\begin{equation} \label{erdollafipim}
M(T-s) = \int_{0}^{s} \Bigg\langle \frac{\nabla \ell\big(T-u,X(T-u)\big)}{\ell\big(T-u,X(T-u)\big)} \, , \, \textnormal{d}\overline{W}^\mathds{P}(T-u) \Bigg\rangle_{\mathds{R}^{n}} \, , \qquad 0 \leqslant s \leqslant T,
\end{equation}
where the stochastic process $\big(\overline{W}^\mathds{P}(T-s)\big)_{0 \leqslant s \leqslant T}$ is a $\mathds{P}$-Brownian motion of the backwards filtration $(\mathcal{G}(T-s))_{0 \leqslant s \leqslant T}$. In particular, the quadratic variation of the martingale of \textnormal{\hyperref[dollafipim]{(\ref*{dollafipim})}} is given by the non-decreasing process in \textnormal{\hyperref[fispdaftr]{(\ref*{fispdaftr})}}, up to a multiplicative factor of $1/2$.
\end{theorem}

\begin{remark} The finiteness of the expression $\mathds{E}_{\mathds{P}}\big[F(0)\big] = \frac{1}{2} \int_{0}^{T} I\big(P(t) \, \vert \, \mathrm{Q}\big) \, \textnormal{d}t$, in conjunction with the representation \textnormal{\hyperref[erdollafipim]{(\ref*{erdollafipim})}}, shows that the martingale of \textnormal{\hyperref[dollafipim]{(\ref*{dollafipim})}} is bounded in $L^{2}(\mathds{P})$. 
\end{remark}

\smallskip

Next, we state the trajectorial version of \hyperref[thetthre]{Theorem \ref*{thetthre}} --- or equivalently, the ``perturbed'' analogue of \hyperref[thetsix]{Theorem \ref*{thetsix}}. As we did in \hyperref[thetthre]{Theorem \ref*{thetthre}}, in particular in the preceding equations \hyperref[pfpeq]{(\ref*{pfpeq})} -- \hyperref[wpsdeids]{(\ref*{wpsdeids})}, we consider the perturbation $\beta \colon \mathds{R}^{n} \rightarrow \mathds{R}^{n}$ and denote the \textit{perturbed likelihood ratio function} by
\begin{equation} \label{dotplrfb}
\ell^{\beta}(t,x) \vcentcolon = \frac{p^{\beta}(t,x)}{q(x)} = p^{\beta}(t,x) \, \mathrm{e}^{2 \Psi(x)} \, , \qquad (t,x) \in [t_{0},\infty) \times \mathds{R}^{n}.
\end{equation}
The stochastic analogue of this quantity is the \textit{perturbed likelihood ratio process} 
\begin{equation} \label{pvdotplrfbs}
\ell^{\beta}\big(t,X(t)\big) = \frac{p^{\beta}\big(t,X(t)\big)}{q\big(X(t)\big)} = p^{\beta}\big(t,X(t)\big) \, \mathrm{e}^{2 \Psi(X(t))} \, , \qquad t \geqslant t_{0}.
\end{equation}
The logarithm of this process is the \textit{perturbed relative entropy process}
\begin{equation} \label{prerepitufdllpitufd} 
\log \ell^{\beta}\big(t,X(t)\big) = \log \Bigg( \frac{p^{\beta}\big(t,X(t)\big)}{q\big(X(t)\big)} \Bigg) = \log p^{\beta}\big(t,X(t)\big) + 2 \, \Psi\big(X(t)\big) \, , \qquad  t \geqslant t_{0}.
\end{equation}

\begin{theorem} \label{thetthretv} Under the \textnormal{\hyperref[sosaojkoia]{Assumptions \ref*{sosaojkoia}}}, we let $t_{0} \geqslant 0$ and $T > t_{0}$. We define the perturbed cumulative relative Fisher information process, accumulated from the right, as
\begin{equation} \label{fispdaftrps}
F^{\beta}(T-s) \vcentcolon = \int_{0}^{s} \Bigg( \, \frac{1}{2}\frac{\big\vert \nabla \ell^{\beta}\big(T-u,X(T-u)\big) \big\vert^{2}}{\ell^{\beta}\big(T-u,X(T-u)\big)^{2}}   + \Big( \big\langle \beta \, , \, 2 \, \nabla \Psi \big\rangle_{\mathds{R}^{n}}- \operatorname{div} \beta \Big) \big(X(T-u)\big) \Bigg) \, \textnormal{d}u
\end{equation}
for $0 \leqslant s \leqslant T - t_{0}$. Then $\mathds{E}_{\mathds{P}^{\beta}}\big[F^{\beta}(t_{0})\big]< \infty$, and the process
\begin{equation} \label{dollafipimps}
M^{\beta}(T-s) \vcentcolon = \Big( \log \ell^{\beta}\big(T-s,X(T-s)\big) - \log \ell^{\beta}\big(T,X(T)\big) \Big) - F^{\beta}(T-s)  
\end{equation}
for $0 \leqslant s \leqslant T - t_{0}$, is a square-integrable martingale of the backwards filtration $(\mathcal{G}(T-s))_{0 \leqslant s \leqslant T - t_{0}}$ under the probability measure $\mathds{P}^{\beta}$. More explicitly, the martingale of \textnormal{\hyperref[dollafipimps]{(\ref*{dollafipimps})}} can be represented as 
\begin{equation} \label{erdollafipimps}
M^{\beta}(T-s) = \int_{0}^{s} \Bigg\langle \frac{\nabla \ell^{\beta}\big(T-u,X(T-u)\big)}{\ell^{\beta}\big(T-u,X(T-u)\big)} \, , \, \textnormal{d}\overline{W}^{\mathds{P}^{\beta}}(T-u) \Bigg\rangle_{\mathds{R}^{n}} \, , \qquad 0 \leqslant s \leqslant T-t_{0},
\end{equation}
where the stochastic process $\big(\overline{W}^{\mathds{P}^{\beta}}(T-s)\big)_{0 \leqslant s \leqslant T-t_{0}}$ is a $\mathds{P}^{\beta}$-Brownian motion of the backwards filtration $(\mathcal{G}(T-s))_{0 \leqslant s \leqslant T-t_{0}}$. 
\end{theorem}

\begin{remark} The representation \textnormal{\hyperref[erdollafipimps]{(\ref*{erdollafipimps})}}, in conjunction with the finiteness of $\mathds{E}_{\mathds{P}^{\beta}}\big[F^{\beta}(0)\big]$, shows that the martingale of \textnormal{\hyperref[dollafipimps]{(\ref*{dollafipimps})}} is bounded in $L^{2}(\mathds{P}^{\beta})$. 
\end{remark}

\subsection{Important consequences} \label{subimpconq}

We state now several important consequences of these two basic results, \hyperref[thetsix]{Theorems \ref*{thetsix}} and \hyperref[thetthretv]{\ref*{thetthretv}}. In particular, we indicate how the corresponding assertions in the earlier \hyperref[thetone]{Theorems \ref*{thetone}}, \hyperref[thetthre]{\ref*{thetthre}} follow directly from these results by taking expectations.

\begin{corollary}[\textsf{Dissipation of relative entropy}] \label{thetsixcorao} Under the \textnormal{\hyperref[sosaojkoianoew]{Assumptions \ref*{sosaojkoianoew}}}, we have for all $t, t_{0} \geqslant 0$ the relative entropy identity
\begin{equation} \label{stewrtpefitre}
H\big( P(t) \, \vert \, \mathrm{Q} \big) - H\big( P(t_{0}) \, \vert \, \mathrm{Q}\big) = \mathds{E}_{\mathds{P}}\Bigg[ \log \Bigg( \frac{\ell\big(t,X(t)\big)}{\ell\big(t_{0},X(t_{0})\big)} \Bigg) \Bigg]  
= \mathds{E}_{\mathds{P}}\Bigg[\int_{t_{0}}^{t} \Bigg(- \frac{1}{2}\frac{\big\vert \nabla \ell\big(u,X(u)\big) \big\vert^{2}}{\ell\big(u,X(u)\big)^{2}} \, \Bigg) \, \textnormal{d}u\Bigg].
\end{equation}
Furthermore, we have for Lebesgue-almost every $t_{0} \geqslant 0$ the \textnormal{\textsf{generalized de Bruijn identity}}
\begin{equation} \label{flfffl}
\lim_{t \rightarrow t_{0}} \, \frac{H\big( P(t) \, \vert \, \mathrm{Q} \big) - H\big( P(t_{0}) \, \vert \, \mathrm{Q}\big)}{t-t_{0}} 
= - \tfrac{1}{2} \, \mathds{E}_{\mathds{P}}\Bigg[ \ \frac{\big\vert \nabla \ell\big(t_{0},X(t_{0})\big) \big\vert^{2}}{\ell\big(t_{0},X(t_{0})\big)^{2}} \ \Bigg],
\end{equation}
as well as the local behavior of the quadratic Wasserstein distance
\begin{equation} \label{agswtuffasih}
\lim_{t \rightarrow t_{0}} \, \frac{W_{2}\big(P(t),P(t_{0})\big)}{\vert t - t_{0} \vert} 
= \tfrac{1}{2} \, \Bigg( \, \mathds{E}_{\mathds{P}}\Bigg[ \ \frac{\big\vert \nabla \ell\big(t_{0},X(t_{0})\big) \big\vert^{2}}{\ell\big(t_{0},X(t_{0})\big)^{2}} \ \Bigg] \, \Bigg)^{1/2}.
\end{equation}
If $t_{0} \geqslant 0$ is chosen so that the generalized de Bruijn identity \textnormal{\hyperref[flfffl]{(\ref*{flfffl})}} does hold, then the limiting assertion \textnormal{\hyperref[agswtuffasih]{(\ref*{agswtuffasih})}} pertaining to the Wasserstein distance is also valid.
\begin{proof}[Proof of \texorpdfstring{\hyperref[thetsixcorao]{Corollary \ref*{thetsixcorao}}}{} from \texorpdfstring{\hyperref[thetsix]{Theorem \ref*{thetsix}}}{}:] The identity \hyperref[stewrtpefitre]{(\ref*{stewrtpefitre})} follows by taking expectations with respect to the probability measure $\mathds{P}$, and invoking the martingale property of the process in \hyperref[dollafipim]{(\ref*{dollafipim})} for $T \geqslant \max\{t_{0},t\}$. In particular, \hyperref[stewrtpefitre]{(\ref*{stewrtpefitre})} shows that the relative entropy function $t \mapsto H( P(t) \, \vert \, \mathrm{Q})$ from \hyperref[doref]{(\ref*{doref})}, thus also the free energy function $t \mapsto F(p(t, \, \cdot \, ))$ from \hyperref[reeef]{(\ref*{reeef})}, are strictly decreasing provided $\ell(t, \, \cdot \,)$ is not constant. 

\smallskip

According to the Lebesgue differentiation theorem, the monotone function $t \mapsto H(P(t) \, \vert \, \mathrm{Q})$ is differentiable for Lebesgue-almost every $t_{0} \geqslant 0$, in which case \hyperref[stewrtpefitre]{(\ref*{stewrtpefitre})} leads to the identity \hyperref[flfffl]{(\ref*{flfffl})}. 

\smallskip

The limiting behavior of the Wasserstein distance \hyperref[agswtuffasih]{(\ref*{agswtuffasih})}, for Lebesgue-almost every $t_{0} \geqslant 0$, is well known and carefully worked out in \cite{AGS08}; see \hyperref[stwt]{Section \ref*{stwt}} below for the details. In \hyperref[agswt]{Theorem \ref*{agswt}} we will prove the last-mentioned assertion of \hyperref[thetsixcorao]{Corollary \ref*{thetsixcorao}}, claiming that the validity of \hyperref[flfffl]{(\ref*{flfffl})} for some $t_{0} \geqslant 0$ implies that the limiting assertion \hyperref[agswtuffasih]{(\ref*{agswtuffasih})} also holds for the same point $t_{0}$.
\end{proof}
\end{corollary}

\begin{proof}[Proof of \texorpdfstring{\hyperref[thetone]{Theorem \ref*{thetone}}}{} from \texorpdfstring{\hyperref[thetsix]{Theorem \ref*{thetsix}}}{}:] Recalling the definition of the relative Fisher information \hyperref[rfi]{(\ref*{rfi})} as well as \hyperref[merfi]{(\ref*{merfi})}, we realize that the limiting assertions \hyperref[flffflnv]{(\ref*{flffflnv})} and \hyperref[agswtuffasihnv]{(\ref*{agswtuffasihnv})} in \hyperref[thetone]{Theorem \ref*{thetone}} correspond to the limits \hyperref[flfffl]{(\ref*{flfffl})} and \hyperref[agswtuffasih]{(\ref*{agswtuffasih})} in the just proved \hyperref[thetsixcorao]{Corollary \ref*{thetsixcorao}}. If $t_{0} \geqslant 0$ is chosen so that the limit \hyperref[flffflnv]{(\ref*{flffflnv})} exists, the last part of \hyperref[thetsixcorao]{Corollary \ref*{thetsixcorao}} tells us that then the limit \hyperref[agswtuffasihnv]{(\ref*{agswtuffasihnv})} exists as well. Therefore, we can divide the first of these limits by the second, in order to obtain the limiting identity \hyperref[ttofeosl]{(\ref*{ttofeosl})} of \hyperref[thetone]{Theorem \ref*{thetone}} for Lebesgue-almost every $t_{0} \geqslant 0$.
\end{proof}

In a manner similar to the derivation of the above \hyperref[thetsixcorao]{Corollary \ref*{thetsixcorao}} from \hyperref[thetsix]{Theorem \ref*{thetsix}}, we deduce now from \hyperref[thetthretv]{Theorem \ref*{thetthretv}} the following \hyperref[thetsixcoraopv]{Corollary \ref*{thetsixcoraopv}}. Its first identity \hyperref[stewrtpefitrepv]{(\ref*{stewrtpefitrepv})} shows, in particular, that the relative entropy $H( P^{\beta}(t) \, \vert \, \mathrm{Q})$ is real-valued for all $t \geqslant t_{0}$.  

\begin{corollary}[\textsf{Dissipation of relative entropy under perturbations}] \label{thetsixcoraopv} Under the \textnormal{\hyperref[sosaojkoianoew]{Assumptions \ref*{sosaojkoianoew}}}, we have for all $t \geqslant t_{0} \geqslant 0$ the relative entropy identity
\begingroup
\addtolength{\jot}{0.7em}
\begin{equation} \label{stewrtpefitrepv}
\begin{aligned}
& H\big( P^{\beta}(t) \, \vert \, \mathrm{Q} \big) - H\big( P^{\beta}(t_{0}) \, \vert \, \mathrm{Q}\big) = \mathds{E}_{\mathds{P}^{\beta}}\Bigg[ \log \Bigg( \frac{\ell^{\beta}\big(t,X(t)\big)}{\ell^{\beta}\big(t_{0},X(t_{0})\big)}\Bigg)\Bigg]   \\
& \qquad = \mathds{E}_{\mathds{P}^{\beta}}\Bigg[\int_{t_{0}}^{t} \Bigg( - \frac{1}{2}\frac{\big\vert \nabla \ell^{\beta}\big(u,X(u)\big) \big\vert^{2}}{\ell^{\beta}\big(u,X(u)\big)^{2}} + \Big(\operatorname{div} \beta -  \big\langle \beta \, , \, 2 \, \nabla \Psi \big\rangle_{\mathds{R}^{n}}\Big)\big(X(u)\big) \Bigg) \textnormal{d}u \Bigg].
\end{aligned}
\end{equation}
\endgroup
Furthermore, for every point $t_{0} \geqslant 0$ at which the right-sided limiting assertion \textnormal{\hyperref[rgtsflffflnv]{(\ref*{rgtsflffflnv})}} is valid \textnormal{(}i.e., every $t_{0} \in \mathds{R}_{+} \setminus N$\textnormal{)}, we have also the limiting identities
\begin{equation} \label{flffflpv}
\lim_{t \downarrow t_{0}} \, \frac{H\big( P^{\beta}(t) \, \vert \, \mathrm{Q} \big) - H\big( P^{\beta}(t_{0}) \, \vert \, \mathrm{Q}\big)}{t-t_{0}} = \mathds{E}_{\mathds{P}}\Bigg[ - \frac{1}{2} \frac{\big\vert \nabla \ell\big(t_{0},X(t_{0})\big) \big\vert^{2}}{\ell\big(t_{0},X(t_{0})\big)^{2}} + \Big( \operatorname{div} \beta -  \big\langle \beta \, , \, 2 \, \nabla \Psi \big\rangle_{\mathds{R}^{n}}\Big)\big(X(t_{0})\big)  \Bigg],
\end{equation}
as well as
\begin{equation} \label{svpvompvv}
\lim_{t \downarrow t_{0}} \, \frac{W_{2}\big( P^{\beta}(t),P^{\beta}(t_{0})\big)}{t-t_{0}} 
= \tfrac{1}{2} \, \Bigg( \, \mathds{E}_{\mathds{P}}\Bigg[ \ \bigg\vert \frac{\nabla \ell\big(t_{0},X(t_{0})\big)}{\ell\big(t_{0},X(t_{0})\big)} + 2 \, \beta\big(X(t_{0})\big) \bigg\vert^{2} \ \Bigg] \, \Bigg)^{1/2}.
\end{equation}
\begin{proof}[Proof of \texorpdfstring{\hyperref[thetsixcoraopv]{Corollary \ref*{thetsixcoraopv}}}{} from \texorpdfstring{\hyperref[thetthretv]{Theorem \ref*{thetthretv}}}{}:] Taking expectations under the probability measure $\mathds{P}^{\beta}$ and using the martingale property of the process in \hyperref[dollafipimps]{(\ref*{dollafipimps})} for $T \geqslant t \geqslant t_{0}$, leads to the identity \hyperref[stewrtpefitrepv]{(\ref*{stewrtpefitrepv})}.

\medskip

In order to derive the limiting identity \hyperref[flffflpv]{(\ref*{flffflpv})} from \hyperref[stewrtpefitrepv]{(\ref*{stewrtpefitrepv})}, some care is needed to show that \hyperref[flffflpv]{(\ref*{flffflpv})} is valid for every time $t_{0} \geqslant 0$ which is not an exceptional point excluded by \hyperref[thetone]{Theorem \ref*{thetone}}, or equivalently by \hyperref[thetsixcorao]{Corollary \ref*{thetsixcorao}}. More precisely, if $t_{0} \geqslant 0$ is chosen so that the right-sided limit \hyperref[rgtsflffflnv]{(\ref*{rgtsflffflnv})} can be derived from \hyperref[stewrtpefitre]{(\ref*{stewrtpefitre})} in \hyperref[thetsixcorao]{Corollary \ref*{thetsixcorao}} (i.e., if $t_{0} \in \mathds{R}_{+} \setminus N$), we have to show that for the same point $t_{0}$ the perturbed equation \hyperref[stewrtpefitrepv]{(\ref*{stewrtpefitrepv})} leads to the identity \hyperref[flffflpv]{(\ref*{flffflpv})}. Colloquially speaking, we want to show that the generalized de Bruijn identity \hyperref[rgtsflffflnv]{(\ref*{rgtsflffflnv})} is stable under perturbations; see in this context also \hyperref[rkresolmz]{Remark \ref*{rkresolmz}} below.

\smallskip

We shall verify in \hyperref[hctclwittmeitpocasot]{Lemma \ref*{hctclwittmeitpocasot}} of \hyperref[subsomusefullem]{Subsection \ref*{subsomusefullem}} below the following estimates on the ratio between the probability density function $p(t, \, \cdot \, )$ and its perturbed version $p^{\beta}(t, \, \cdot \, )$: For every $t_{0} \geqslant 0$ and $T > t_{0}$ there is a constant $C > 0$ such that 
\begin{equation} \label{ilpdepvhfv}
\bigg\vert \frac{\ell^{\beta}(t,x)}{\ell(t,x)} - 1 \bigg\vert = \bigg\vert \frac{p^{\beta}(t,x)}{p(t,x)} - 1 \bigg\vert \leqslant C \, (t-t_{0}) \, , \qquad (t,x) \in [t_{0},T] \times \mathds{R}^{n}
\end{equation}
as well as
\begin{equation} \label{ilpdepvhfvsv}
\mathds{E}_{\mathds{P}}\Bigg[\int_{t_{0}}^{t} \ \Bigg\vert \nabla \log \Bigg( \frac{\ell^{\beta}\big(u,X(u)\big)}{\ell\big(u,X(u)\big)} \Bigg) \Bigg\vert^{2} \, \textnormal{d}u \Bigg] 
\leqslant C \, (t-t_{0})^{2} \, , \qquad t_{0} \leqslant t \leqslant T.
\end{equation}

\smallskip

We turn now to the derivation of \hyperref[flffflpv]{(\ref*{flffflpv})} from \hyperref[stewrtpefitrepv]{(\ref*{stewrtpefitrepv})}. First, as the perturbation $\beta$ is smooth and compactly supported, and the paths of the canonical coordinate process $(X(t))_{t \geqslant 0}$ are continuous, we have clearly
\begin{equation} \label{llethar}
\lim_{t \downarrow t_{0}} \, \frac{1}{t-t_{0}} \, \mathds{E}_{\mathds{P}^{\beta}}\Bigg[\int_{t_{0}}^{t} \Big(  \operatorname{div} \beta -  \big\langle \beta \, , \, 2 \, \nabla \Psi \big\rangle_{\mathds{R}^{n}}\Big)\big(X(u)\big) \, \textnormal{d}u \Bigg] 
= \mathds{E}_{\mathds{P}^{\beta}}\Big[ \Big(\operatorname{div} \beta -  \big\langle \beta \, , \, 2 \, \nabla \Psi \big\rangle_{\mathds{R}^{n}}\Big)\big(X(t_{0})\big) \Big] 
\end{equation}
\textit{for every} $t_{0} \geqslant 0$. Secondly, the random variable $X(t_{0})$ has the same distribution under $\mathds{P}$, as it does under $\mathds{P}^{\beta}$, so it is immaterial whether we express the expectation on the right-hand side of \hyperref[llethar]{(\ref*{llethar})} with respect to the probability measure $\mathds{P}$ or $\mathds{P}^{\beta}$. Hence this expression equals the corresponding term on the right-hand side of \hyperref[flffflpv]{(\ref*{flffflpv})}, as required.

\smallskip

Regarding the remaining term on the right-hand side of \hyperref[flffflpv]{(\ref*{flffflpv})}, it can be seen by applying \hyperref[ilpdepvhfv]{(\ref*{ilpdepvhfv})} and \hyperref[ilpdepvhfvsv]{(\ref*{ilpdepvhfvsv})}, that the equality
\begin{equation} \label{tciwsolafophvs}
\lim_{t \downarrow t_{0}} \, \frac{1}{t-t_{0}} \, \mathds{E}_{\mathds{P}^{\beta}}\Bigg[\int_{t_{0}}^{t} \Bigg( -\frac{1}{2}\frac{\big\vert \nabla \ell^{\beta}\big(u,X(u)\big) \big\vert^{2}}{\ell^{\beta}\big(u,X(u)\big)^{2}} \, \Bigg) \, \textnormal{d}u \Bigg] 
= \lim_{t \downarrow t_{0}} \, \frac{1}{t-t_{0}} \, \mathds{E}_{\mathds{P}}\Bigg[\int_{t_{0}}^{t} \Bigg( -\frac{1}{2}\frac{\big\vert \nabla \ell\big(u,X(u)\big) \big\vert^{2}}{\ell\big(u,X(u)\big)^{2}} \, \Bigg) \, \textnormal{d}u \Bigg]
\end{equation}
holds as long as $t_{0} \geqslant 0$ is chosen so that one of the limits exists; for the details we refer to \cite[Section 3.1]{Tsc19}. In other words, the existence and equality of the limits in \hyperref[tciwsolafophvs]{(\ref*{tciwsolafophvs})} is guaranteed if and only if $t_{0} \in \mathds{R}_{+} \setminus N$. It develops that both limits in \hyperref[tciwsolafophvs]{(\ref*{tciwsolafophvs})} exist if $t_{0} \geqslant 0$ is not contained in the exceptional set $N$ of zero Lebesgue measure, and their common value is 
\begin{equation} \label{tcigwwsolafophvs}
- \tfrac{1}{2} \, I\big( P(t_{0}) \, \vert \, \mathrm{Q}\big) = - \tfrac{1}{2} \, \mathds{E}_{\mathds{P}}\Bigg[ \ \frac{\big\vert \nabla \ell\big(t_{0},X(t_{0})\big) \big\vert^{2}}{\ell\big(t_{0},X(t_{0})\big)^{2}} \ \Bigg];
\end{equation}
in conjunction with \hyperref[llethar]{(\ref*{llethar})}, which is valid for every $t_{0} \geqslant 0$, this establishes the limiting identity \hyperref[flffflpv]{(\ref*{flffflpv})} for every $t_{0} \in \mathds{R}_{+} \setminus N$. Therefore, the right-sided limiting assertion \hyperref[rgtsflffflnv]{(\ref*{rgtsflffflnv})} and the similar perturbed limiting assertion in \hyperref[flffflpv]{(\ref*{flffflpv})} fail on precisely the same set of exceptional points $N$. 

\medskip

As regards the final assertion \hyperref[svpvompvv]{(\ref*{svpvompvv})}, we note that, by analogy with \hyperref[agswtuffasih]{(\ref*{agswtuffasih})}, the limiting behavior of the Wasserstein distance \hyperref[svpvompvv]{(\ref*{svpvompvv})}, for Lebesgue-almost every $t_{0} \geqslant 0$, is well known \cite{AGS08}; for the details we refer to \hyperref[stwt]{Section \ref*{stwt}} below. More precisely, it will follow from \hyperref[bvagswt]{Theorem \ref*{bvagswt}} that the limiting assertion
\begin{equation} \label{osagswtuffasihnv}
\lim_{t \downarrow t_{0}} \, \frac{W_{2}\big(P(t),P(t_{0})\big)}{t - t_{0}} 
= \tfrac{1}{2} \, \sqrt{I\big( P(t_{0}) \, \vert \, \mathrm{Q}\big)}
\end{equation}
is valid for every $t_{0} \in \mathds{R}_{+} \setminus N$. Once again, concerning the relation between the limits in \hyperref[osagswtuffasihnv]{(\ref*{osagswtuffasihnv})} and \hyperref[svpvompvv]{(\ref*{svpvompvv})} pertaining to the Wasserstein distance, we discern a similar pattern as in the case of the generalized de Bruijn identity. In fact, \hyperref[bvagswt]{Theorem \ref*{bvagswt}} will tell us that the perturbed Wasserstein limit \hyperref[svpvompvv]{(\ref*{svpvompvv})} also holds for every $t_{0} \in \mathds{R}_{+} \setminus N$. In other words, the local behavior of the quadratic Wasserstein distance is stable under perturbations as well; we shall come back to this point in \hyperref[wvotarrkresolmz]{Remark \ref*{wvotarrkresolmz}} below. 

\medskip

Summing up, if $t_{0} \in \mathds{R}_{+} \setminus N$, i.e., whenever the limiting identity \hyperref[rgtsflffflnv]{(\ref*{rgtsflffflnv})} holds, the limiting assertions \hyperref[flffflpv]{(\ref*{flffflpv})} and \hyperref[svpvompvv]{(\ref*{svpvompvv})} are valid as well. Hence, except for the set $N$ of zero Lebesgue measure in \hyperref[rethpoitwooslda]{Remark \ref*{rethpoitwooslda}}, we have shown the validity of \hyperref[flffflpv]{(\ref*{flffflpv})} and \hyperref[svpvompvv]{(\ref*{svpvompvv})}, thus completing the proof of \hyperref[thetsixcoraopv]{Corollary \ref*{thetsixcoraopv}}.
\end{proof}
\end{corollary}

\begin{proof}[Proof of \texorpdfstring{\hyperref[thetthre]{Theorem \ref*{thetthre}}}{} from \texorpdfstring{\hyperref[thetsix]{Theorems \ref*{thetsix}}}{}, \texorpdfstring{\hyperref[thetthretv]{\ref*{thetthretv}}}{}:] Let $t_{0} \in \mathds{R}_{+} \setminus N$, i.e., such that the limiting assertion \hyperref[rgtsflffflnv]{(\ref*{rgtsflffflnv})}, and as a consequence also \hyperref[osagswtuffasihnv]{(\ref*{osagswtuffasihnv})}, are valid (these are the right-sided limits corresponding to \hyperref[flfffl]{(\ref*{flfffl})}, \hyperref[agswtuffasih]{(\ref*{agswtuffasih})} in \hyperref[thetsixcorao]{Corollary \ref*{thetsixcorao}} of \hyperref[thetthre]{Theorem \ref*{thetthre}}). Then the limiting identities \hyperref[flffflpv]{(\ref*{flffflpv})}, \hyperref[svpvompvv]{(\ref*{svpvompvv})} from \hyperref[thetsixcoraopv]{Corollary \ref*{thetsixcoraopv}} of \hyperref[thetthretv]{Theorem \ref*{thetthretv}} are valid as well. Recalling the abbreviations $a = \nabla \log \ell(t_{0},X(t_{0}))$ and $b = \beta(X(t_{0}))$ in \hyperref[ttrvzo]{(\ref*{ttrvzo})}, we summarize now the identities just mentioned as
\begingroup
\addtolength{\jot}{1em}
\begin{alignat}{3}
&\lim_{t \downarrow t_{0}} \, \frac{H\big( P(t) \, \vert \, \mathrm{Q} \big) - H\big( P(t_{0}) \, \vert \, \mathrm{Q}\big)}{t-t_{0}} 
&&= - && \tfrac{1}{2} \, \| a \|_{L^{2}(\mathds{P})}^{2}, \label{nlwpthtt1} \\
&\lim_{t \downarrow t_{0}} \, \frac{W_{2}\big( P(t),P(t_{0})\big)}{t-t_{0}}
&&= && \tfrac{1}{2} \, \| a \|_{L^{2}(\mathds{P})}, \label{nlwpthtt2} \\
&\lim_{t \downarrow t_{0}} \, \frac{H\big( P^{\beta}(t) \, \vert \, \mathrm{Q} \big) - H\big( P^{\beta}(t_{0}) \, \vert \, \mathrm{Q}\big)}{t-t_{0}} 
&&= - && \tfrac{1}{2}  \, \big\langle a , a + 2b \big\rangle_{L^{2}(\mathds{P})}, \label{nlwpthtt3} \\
&\lim_{t \downarrow t_{0}} \, \frac{W_{2}\big( P^{\beta}(t),P^{\beta}(t_{0})\big)}{t-t_{0}} 
&&= && \tfrac{1}{2}  \, \| a + 2 b\|_{L^{2}(\mathds{P})}. \label{nlwpthtt4}
\end{alignat}
\endgroup
Indeed, the equations \hyperref[nlwpthtt1]{(\ref*{nlwpthtt1})}, \hyperref[nlwpthtt2]{(\ref*{nlwpthtt2})}, and \hyperref[nlwpthtt4]{(\ref*{nlwpthtt4})} correspond precisely to \hyperref[rgtsflffflnv]{(\ref*{rgtsflffflnv})}, \hyperref[osagswtuffasihnv]{(\ref*{osagswtuffasihnv})}, and \hyperref[svpvompvv]{(\ref*{svpvompvv})}, respectively. As for \hyperref[nlwpthtt3]{(\ref*{nlwpthtt3})}, we note that, according to equation \hyperref[flffflpv]{(\ref*{flffflpv})} of \hyperref[thetsixcoraopv]{Corollary \ref*{thetsixcoraopv}}, the limit in \hyperref[nlwpthtt3]{(\ref*{nlwpthtt3})} equals 
\begin{equation}
- \tfrac{1}{2} \, \| a \|_{L^{2}(\mathds{P})}^{2} + \mathds{E}_{\mathds{P}}\Big[ \Big(\operatorname{div} \beta - 2 \, \big\langle \beta , \nabla \Psi \big\rangle_{\mathds{R}^{n}}\Big)\big(X(t_{0})\big) \Big].    
\end{equation}
Therefore, in view of the right-hand side of \hyperref[nlwpthtt3]{(\ref*{nlwpthtt3})}, we have to show the identity
\begin{equation} \label{telimitthetthretv}
\mathds{E}_{\mathds{P}}\Big[ \Big(\operatorname{div} \beta -  \big\langle \beta \, , \, 2 \, \nabla \Psi \big\rangle_{\mathds{R}^{n}}\Big)\big(X(t_{0})\big) \Big] = - \langle a,b \rangle_{L^{2}(\mathds{P})}.
\end{equation}
In order to do this, we write the left-hand side of \hyperref[telimitthetthretv]{(\ref*{telimitthetthretv})} as
\begin{equation} \label{dftelimitthetthretv}
\int_{\mathds{R}^{n}} \Big( \operatorname{div} \beta(x) - \big\langle \beta(x) \, , \, 2 \, \nabla \Psi(x) \big\rangle_{\mathds{R}^{n}} \Big) \, p(t_{0},x) \, \textnormal{d}x.
\end{equation}
Using --- for the first time, and only in order to show the identity \hyperref[telimitthetthretv]{(\ref*{telimitthetthretv})} --- integration by parts, and the fact that the perturbation $\beta$ is assumed to be smooth and compactly supported, we see that the expression \hyperref[dftelimitthetthretv]{(\ref*{dftelimitthetthretv})} becomes
\begin{equation}
- \int_{\mathds{R}^{n}} \Big\langle \beta(x) \, , \, \nabla \log p(t_{0},x) + 2 \, \nabla \Psi(x) \Big\rangle_{\mathds{R}^{n}} \, p(t_{0},x) \, \textnormal{d}x,
\end{equation}
which is the same as $- \big\langle \beta(X(t_{0})), \nabla \log \ell(t_{0},X(t_{0})) \big\rangle_{L^{2}(\mathds{P})} = - \langle b, a\rangle_{L^{2}(\mathds{P})}$. 

\smallskip

The limiting identities \hyperref[nlwpthtt1]{(\ref*{nlwpthtt1})} -- \hyperref[nlwpthtt4]{(\ref*{nlwpthtt4})} now clearly imply the assertions of \hyperref[thetthre]{Theorem \ref*{thetthre}}. 
\end{proof}

\medskip

The following two results, \hyperref[thetsixcor]{Propositions \ref*{thetsixcor}} and \hyperref[thetthretvcor]{\ref*{thetthretvcor}}, are trajectorial versions of \hyperref[thetsixcorao]{Corollaries \ref*{thetsixcorao}} and \hyperref[thetsixcoraopv]{\ref*{thetsixcoraopv}}, respectively. They compute the rate of temporal change of relative entropy for the equation \hyperref[sdeids]{(\ref*{sdeids})} and for its perturbed version \hyperref[wpsdeids]{(\ref*{wpsdeids})}, respectively, in the more precise trajectorial manner of \hyperref[thetsix]{Theorems \ref*{thetsix}}, \hyperref[thetthretv]{\ref*{thetthretv}}. 

\begin{proposition}[\textsf{Trajectorial rate of relative entropy dissipation}] \label{thetsixcor} Under the \textnormal{\hyperref[sosaojkoia]{Assumptions \ref*{sosaojkoia}}}, let $t_{0} \geqslant 0$ be such that the generalized de Bruijn identity \textnormal{\hyperref[flfffl]{(\ref*{flfffl})}} does hold. Then the relative entropy process \textnormal{\hyperref[stlrpd]{(\ref*{stlrpd})}} satisfies, with $T > t_{0}$, the following trajectorial relations:
\begingroup
\addtolength{\jot}{0.7em}
\begin{align}
& \lim_{s \uparrow T-t_{0}} \, \frac{\mathds{E}_{\mathds{P}}\Big[ \log \ell\big(t_{0},X(t_{0})\big) \ \big\vert \ \mathcal{G}(T-s) \Big] - \log \ell \big( T-s,X(T-s)\big)}{T-t_{0}-s}  \label{thetsixcorfe} \\
& \qquad = \lim_{s \downarrow T-t_{0}} \, \frac{\mathds{E}_{\mathds{P}}\Big[ \log \ell\big(T-s,X(T-s)\big) \ \big\vert \ \mathcal{G}(t_{0}) \Big] - \log \ell \big( t_{0},X(t_{0})\big)}{s-(T-t_{0})}   \label{thetsixcorse} \\
& \qquad =  \frac{1}{2} \frac{\big\vert \nabla \ell\big(t_{0},X(t_{0})\big) \big\vert^{2}}{\ell\big(t_{0},X(t_{0})\big)^{2}} = \frac{1}{2} \bigg\vert \frac{\nabla p\big(t_{0},X(t_{0})\big)}{p\big(t_{0},X(t_{0})\big)} + 2 \, \nabla \Psi\big(X(t_{0})\big)  \bigg\vert^{2}, \qquad \qquad \label{thetsixcorthe}
\end{align}
\endgroup
where the limits \textnormal{\hyperref[thetsixcorfe]{(\ref*{thetsixcorfe})}} and \textnormal{\hyperref[thetsixcorse]{(\ref*{thetsixcorse})}} exist in $L^{1}(\mathds{P})$.
\end{proposition}

\begin{remark} The limiting assertions \hyperref[thetsixcorfe]{(\ref*{thetsixcorfe})} -- \hyperref[thetsixcorthe]{(\ref*{thetsixcorthe})} of \hyperref[thetsixcor]{Proposition \ref*{thetsixcor}} are the conditional trajectorial versions of the generalized de Bruijn identity \hyperref[flfffl]{(\ref*{flfffl})}.
\end{remark}

\begin{proof}[Proof of \texorpdfstring{\hyperref[thetsixcor]{Proposition \ref*{thetsixcor}}}{} from \texorpdfstring{\hyperref[thetsix]{Theorem \ref*{thetsix}}}{}:] Let $t_{0} \geqslant 0$ be such that the generalized de Bruijn identity \textnormal{\hyperref[flfffl]{(\ref*{flfffl})}} from \hyperref[thetsixcorao]{Corollary \ref*{thetsixcorao}} of \hyperref[thetsix]{Theorem \ref*{thetsix}} is valid, and select $T > t_{0}$. The martingale property of the process in \hyperref[dollafipim]{(\ref*{dollafipim})} allows us to write the numerator in \hyperref[thetsixcorfe]{(\ref*{thetsixcorfe})} as
\begin{equation}
\mathds{E}_{\mathds{P}}\Big[ F(t_{0})-F(T-s) \ \big\vert \ \mathcal{G}(T-s) \Big], \qquad 0 \leqslant s \leqslant T-t_{0},
\end{equation}
in the notation of \hyperref[fispdaftr]{(\ref*{fispdaftr})}. Similarly, the numerator in \hyperref[thetsixcorse]{(\ref*{thetsixcorse})} equals $\mathds{E}_{\mathds{P}}\big[ F(T-s) - F(t_{0}) \ \vert \ \mathcal{G}(t_{0})\big]$, $T-t_{0} \leqslant s \leqslant T$. By analogy with the derivation of \hyperref[flfffl]{(\ref*{flfffl})} from \hyperref[stewrtpefitre]{(\ref*{stewrtpefitre})}, where we calculated real-valued expectations, we rely on the Lebesgue differentiation theorem to obtain the corresponding results \hyperref[thetsixcorfe]{(\ref*{thetsixcorfe})} -- \hyperref[thetsixcorthe]{(\ref*{thetsixcorthe})} for conditional expectations. Using the left-continuity of the backwards filtration $(\mathcal{G}(T-s))_{0 \leqslant s \leqslant T}$, we can invoke the measure-theoretic result in \hyperref[probamtr]{Proposition \ref*{probamtr}} of \hyperref[apsecamtr]{Appendix \ref*{apsecamtr}}, which establishes the claims \hyperref[thetsixcorfe]{(\ref*{thetsixcorfe})} -- \hyperref[thetsixcorthe]{(\ref*{thetsixcorthe})} pertaining to conditional expectations. 
\end{proof}

\begin{proposition}[\textsf{Trajectorial rate of relative entropy dissipation under perturbations}] \label{thetthretvcor} Under the \textnormal{\hyperref[sosaojkoia]{Assumptions \ref*{sosaojkoia}}}, let $t_{0} \in \mathds{R}_{+} \setminus N$. Then the relative entropy process \textnormal{\hyperref[stlrpd]{(\ref*{stlrpd})}} and its perturbed version \textnormal{\hyperref[prerepitufdllpitufd]{(\ref*{prerepitufdllpitufd})}} satisfy, with $T > t_{0}$, the following trajectorial relations:
\begingroup
\addtolength{\jot}{0.7em}
\begin{equation} \label{thetsixcorfes}
\begin{aligned}
& \lim_{s \uparrow T-t_{0}} \, \frac{\mathds{E}_{\mathds{P}^{\beta}}\Big[ \log \ell^{\beta}\big(t_{0},X(t_{0})\big) \ \big\vert \ \mathcal{G}(T-s) \Big] - \log \ell^{\beta} \big( T-s,X(T-s)\big)}{T-t_{0}-s}  \\
& \qquad = \frac{1}{2}\frac{\big\vert \nabla \ell\big(t_{0},X(t_{0})\big) \big\vert^{2}}{\ell\big(t_{0},X(t_{0})\big)^{2}} - \operatorname{div} \beta\big(X(t_{0})\big)  
+ \Big\langle \beta\big(X(t_{0})\big) \, , \, 2 \, \nabla \Psi\big(X(t_{0})\big) \Big\rangle_{\mathds{R}^{n}}, 
\end{aligned}
\end{equation}
\endgroup
as well as
\begingroup
\addtolength{\jot}{0.7em}
\begin{equation} \label{thetsixcorfessl} 
\begin{aligned}
& \lim_{s \uparrow T-t_{0}} \, \frac{\mathds{E}_{\mathds{P}}\Big[ \log \ell^{\beta}\big(t_{0},X(t_{0})\big) \ \big\vert \ \mathcal{G}(T-s) \Big] - \log \ell^{\beta} \big( T-s,X(T-s)\big)}{T-t_{0}-s} \\
& \qquad = \frac{1}{2}\frac{\big\vert \nabla \ell\big(t_{0},X(t_{0})\big) \big\vert^{2}}{\ell\big(t_{0},X(t_{0})\big)^{2}} - \bigg( \operatorname{div} \beta\big(X(t_{0})\big)  
+ \Big\langle \beta\big(X(t_{0})\big) \, , \, \nabla  \log p \big(t_{0},X(t_{0})\big) \Big\rangle_{\mathds{R}^{n}} \bigg), 
\end{aligned}
\end{equation}
\endgroup
and
\begingroup
\addtolength{\jot}{0.7em}
\begin{equation} \label{titmlwhtcfs}
\begin{aligned}
& \lim_{s \uparrow T-t_{0}} \, \frac{\log \ell^{\beta}\big(T-s,X(T-s)\big) - \log \ell\big(T-s,X(T-s)\big)}{T-t_{0}-s}  \\
& \qquad = \operatorname{div} \beta\big(X(t_{0})\big) + \Big\langle \beta\big(X(t_{0})\big) \, , \, \nabla  \log p\big(t_{0},X(t_{0})\big) \Big\rangle_{\mathds{R}^{n}}, 
\end{aligned}
\end{equation}
\endgroup
where the limits in \textnormal{\hyperref[thetsixcorfes]{(\ref*{thetsixcorfes})}} -- \textnormal{\hyperref[titmlwhtcfs]{(\ref*{titmlwhtcfs})}} exist in both $L^{1}(\mathds{P})$ and $L^{1}(\mathds{P}^{\beta})$.
\end{proposition}

\begin{remark} It is perhaps noteworthy that the three limiting expressions in \hyperref[thetsixcorfes]{(\ref*{thetsixcorfes})}, \hyperref[thetsixcorfessl]{(\ref*{thetsixcorfessl})} and \hyperref[titmlwhtcfs]{(\ref*{titmlwhtcfs})} are quite different from each other. The first limiting assertion \hyperref[thetsixcorfes]{(\ref*{thetsixcorfes})} of \hyperref[thetthretvcor]{Proposition \ref*{thetthretvcor}} is the conditional trajectorial version of the perturbed de Bruijn identity \hyperref[flffflpv]{(\ref*{flffflpv})}. We also note that in fact the third limiting assertion \hyperref[titmlwhtcfs]{(\ref*{titmlwhtcfs})} is valid for all $t_{0} > 0$.
\end{remark}

\begin{proof}[Proof of the assertion \texorpdfstring{\hyperref[thetsixcorfes]{\textnormal{(\ref*{thetsixcorfes})}}}{} in \texorpdfstring{\hyperref[thetthretvcor]{Proposition \ref*{thetthretvcor}}}{}, from \texorpdfstring{\hyperref[thetthretv]{Theorem \ref*{thetthretv}}}{}:] Let $t_{0} \in \mathds{R}_{+} \setminus N$, i.e., so that the right-sided limiting assertion \hyperref[rgtsflffflnv]{(\ref*{rgtsflffflnv})} is valid, and select $T > t_{0}$. In \hyperref[tciwsolafophvs]{(\ref*{tciwsolafophvs})} from \hyperref[thetsixcoraopv]{Corollary \ref*{thetsixcoraopv}} of \hyperref[thetthretv]{Theorem \ref*{thetthretv}} we have seen that the limits in \hyperref[rgtsflffflnv]{(\ref*{rgtsflffflnv})} and \hyperref[flffflpv]{(\ref*{flffflpv})} have the same exceptional sets, hence also the limiting identity \hyperref[flffflpv]{(\ref*{flffflpv})} holds. Now, for such $t_{0}$, we show the limiting assertion \hyperref[thetsixcorfes]{(\ref*{thetsixcorfes})} in the same way as the assertion \hyperref[thetsixcorfe]{(\ref*{thetsixcorfe})} in the proof of \hyperref[thetsixcor]{Proposition \ref*{thetsixcor}} above. Indeed, this time we invoke the $\mathds{P}^{\beta}$-martingale property of the process in \hyperref[dollafipimps]{(\ref*{dollafipimps})}, and write the numerator in the first line of \hyperref[thetsixcorfes]{(\ref*{thetsixcorfes})} as $\mathds{E}_{\mathds{P}^{\beta}}\big[ F^{\beta}(t_{0})-F^{\beta}(T-s) \ \big\vert \ \mathcal{G}(T-s) \big]$, $0 \leqslant s \leqslant T-t_{0}$, in the notation of \hyperref[fispdaftrps]{(\ref*{fispdaftrps})}. 

\smallskip

Applying \hyperref[probamtr]{Proposition \ref*{probamtr}} of \hyperref[apsecamtr]{Appendix \ref*{apsecamtr}} in this situation proves the limiting identity \hyperref[thetsixcorfes]{(\ref*{thetsixcorfes})} in $L^{1}(\mathds{P}^{\beta})$. As we shall see in \hyperref[hctclwittmeitpocasotpv]{Lemma \ref*{hctclwittmeitpocasotpv}} of \hyperref[subsomusefullem]{Subsection \ref*{subsomusefullem}} below, the probability measures $\mathds{P}$ and $\mathds{P}^{\beta}$ are equivalent, and the mutual Radon-Nikod\'{y}m derivatives $\frac{\textnormal{d}\mathds{P}^{\beta}}{\textnormal{d}\mathds{P}}$ and $\frac{\textnormal{d}\mathds{P}}{\textnormal{d}\mathds{P}^{\beta}}$ are bounded on the $\sigma$-algebra $\mathcal{F}(T) = \mathcal{G}(0)$ (recall, in this vein, the claims of \hyperref[ilpdepvhfv]{(\ref*{ilpdepvhfv})}). Hence, convergence in $L^{1}(\mathds{P})$ is equivalent to convergence in $L^{1}(\mathds{P}^{\beta})$. This establishes the $L^{1}(\mathds{P})$-convergence of \hyperref[thetsixcorfes]{(\ref*{thetsixcorfes})}, which completes the proof of the limiting assertion \hyperref[thetsixcorfes]{(\ref*{thetsixcorfes})}.

\smallskip

The proofs of the limiting assertions \hyperref[thetsixcorfessl]{(\ref*{thetsixcorfessl})} and \hyperref[titmlwhtcfs]{(\ref*{titmlwhtcfs})} are postponed to \hyperref[nspofthetthretvcor]{Subsection \ref*{nspofthetthretvcor}}.
\end{proof}


\subsection{Ramifications} \label{ramifications}

\hyperref[thetthre]{Theorem \ref*{thetthre}} and, in particular, its equation \hyperref[nlwpthtt3]{(\ref*{nlwpthtt3})} above, show --- at least on a formal level --- that the functional
\begin{equation}
\mathscr{P}_{2}(\mathds{R}^{n}) \ni P \longmapsto H( P \, \vert \, \mathrm{Q} ) - H ( P(0) \, \vert \, \mathrm{Q} )
\end{equation}
can be approximated linearly in the neighborhood of $P(0)$ by the functional
\begin{equation} \label{cblaitnopbtf}
\mathscr{P}_{2}(\mathds{R}^{n}) \ni P \longmapsto \langle a , c \rangle_{L^{2}(\mathds{P})}, \qquad \textnormal{ where } \quad c = -\tfrac{a}{2}-b
\end{equation}
as in \hyperref[nlwpthtt3]{(\ref*{nlwpthtt3})} with $t_{0} = 0$ and $a = \nabla \log \ell(0,X(0))$, $b = \beta(X(0))$. As it turns out, formula \hyperref[cblaitnopbtf]{(\ref*{cblaitnopbtf})} is closely related to the sharpened form of the \textit{HWI inequality} due to Otto and Villani \cite{OV00} (see also Cordero-Erausquin \cite{CE02} and \cite[p.\ 650]{Vil09}); we explain presently how.

\smallskip

Consider the starting time $t_{0} = 0$ and the curve $(P^{\beta}(t))_{t \geqslant 0}$ as in \hyperref[thetthre]{Theorems \ref*{thetthre}} and \hyperref[thetthretv]{\ref*{thetthretv}}, for a fixed perturbation $\beta$ as above, and suppose that this $t_{0}$ is not an exceptional point in the preceding limiting assertions. Let us fix $P_{1} \in \mathscr{P}_{2}(\mathds{R}^{n})$ and study the ``tangent'' $(P_{t})_{0 \leqslant t \leqslant 1}$ to the curve $(P^{\beta}(t))_{t \geqslant 0}$ at the point $P^{\beta}(0) = P(0) = P_{0}$ in the quadratic Wasserstein space $\mathscr{P}_{2}(\mathds{R}^{n})$, and analyze the behavior of the relative entropy functional \hyperref[dotref]{(\ref*{dotref})} along the curve $(P_{t})_{0 \leqslant t \leqslant 1}$. Here $(P_{t})_{0 \leqslant t \leqslant 1}$ is understood to be a ``straight line'' in $\mathscr{P}_{2}(\mathds{R}^{n})$; i.e., using the terminology of McCann \cite{McC97}, as the \textit{``displacement interpolation''} or \textit{``constant speed geodesic''} between the elements $P_{0}$ and $P_{1}$ in $\mathscr{P}_{2}(\mathds{R}^{n})$.

\smallskip

Once we have identified this tangent $(P_{t})_{0 \leqslant t \leqslant 1}$, it is geometrically obvious --- at least on an intuitive level --- that the slope of the relative entropy functional \hyperref[dotref]{(\ref*{dotref})} along the straight line $(P_{t})_{0 \leqslant t \leqslant 1}$ should be equal to the slope along $(P^{\beta}(t))_{t \geqslant 0}$ at the touching point $P^{\beta}(0) = P_{0}$. This slope is given by \hyperref[cblaitnopbtf]{(\ref*{cblaitnopbtf})}, and we shall verify in the following \hyperref[hlotl]{Lemma \ref*{hlotl}} and in \hyperref[hwiai]{Proposition \ref*{hwiai}} that --- under suitable regularity assumptions --- it coincides with the slope along $(P_{t})_{0 \leqslant t \leqslant 1}$ as identified by Otto and Villani in \cite{OV00} and Cordero-Erausquin in \cite{CE02}.

\smallskip

To work out the connection between \hyperref[cblaitnopbtf]{(\ref*{cblaitnopbtf})} and \cite{OV00}, \cite{CE02} we shall turn things upside down; i.e., we \textit{first} define the tangent $(P_{t})_{0 \leqslant t \leqslant 1}$, and \textit{then} find the corresponding perturbation $\beta$ so that the curve $(P^{\beta}(t))_{t \geqslant 0}$ indeed has $(P_{t})_{0 \leqslant t \leqslant 1}$ as tangent at the point $P^{\beta}(0) = P_{0}$. In this manner, we shall treat $\beta$ more as an element of ``control'', than as a perturbation.

\smallskip

Fix an element $P \in \mathscr{P}_{2}(\mathds{R}^{n})$, and let $\gamma \colon \mathds{R}^{n} \rightarrow \mathds{R}^{n}$ be such that $T(x) \vcentcolon = x + \gamma(x)$ transports $P_{0} = P(0)$ to $P_{1} = P$ optimally with respect to the quadratic Wasserstein distance, i.e., $T_{\#}(P_{0}) = P_{1}$ and $\| \gamma \|_{L^{2}(P_{0})} = W_{2}(P_{0},P_{1})$; see also \hyperref[hlotlse]{(\ref*{hlotlse})} and \hyperref[otmitwsnp]{(\ref*{otmitwsnp})} below. Here and throughout, $T_{\#}(P_{0})$ denotes the pushforward measure of $P_{0}$ by the map $T$ and is given by $(T_{\#}(P_{0}))(B) = P_{0}(T^{-1}(B))$ for every Borel set $B \subseteq \mathds{R}^{n}$. Then \hyperref[nlwpthtt3]{(\ref*{nlwpthtt3})}, \hyperref[cblaitnopbtf]{(\ref*{cblaitnopbtf})} suggest that the displacement interpolation $(P_{t})_{0 \leqslant t \leqslant 1}$ between the two probability measures $P_{0} = P(0)$ and $P_{1} = P$, to be defined in \hyperref[hlotlse]{(\ref*{hlotlse})} below, is tangent to the curve $(P^{\beta}(t))_{t \geqslant 0}$ as in \hyperref[thetthre]{Theorems \ref*{thetthre}} and \hyperref[thetthretv]{\ref*{thetthretv}}, if $\gamma$ and $\beta$ are related via  
\begin{equation}
\gamma(x) = - \tfrac{1}{2} \, \nabla \log \ell(0,x) - \beta(x), \qquad x \in \mathds{R}^{n};
\end{equation}
whereas the random variable $c$ of \hyperref[cblaitnopbtf]{(\ref*{cblaitnopbtf})} becomes $c = \gamma(X(0))$.

\smallskip

We formalize these intuitive geometric insights in the subsequent \hyperref[hlotl]{Lemma \ref*{hlotl}}, which provides the analogue of \hyperref[nlwpthtt3]{(\ref*{nlwpthtt3})} for the displacement interpolation flow $(P_{t})_{0 \leqslant t \leqslant 1}$ of \hyperref[hlotlse]{(\ref*{hlotlse})}. To this end, we impose temporarily the following strong regularity conditions. As it will turn out in the proof of \hyperref[hwiai]{Proposition \ref*{hwiai}}, these will not restrict, eventually, the generality of the argument.

\begin{assumptions}[\textsf{Regularity assumptions of \texorpdfstring{\hyperref[hlotl]{Lemma \ref*{hlotl}}}{}}] \label{hwisosaojkoia} In addition to the conditions \hyperref[faosaojko]{\ref*{faosaojko}} -- \hyperref[nalwstasas]{\ref*{nalwstasas}} of \hyperref[sosaojkoia]{Assumptions \ref*{sosaojkoia}}, we also impose that:
\begin{enumerate}[label=(\roman*)] 
\setcounter{enumi}{6}
\item \label{hwitsaosaojko} $P_{0}$ and $P_{1}$ are probability measures in $\mathscr{P}_{2}(\mathds{R}^{n})$ with smooth densities, which are compactly supported and strictly positive on the interior of their respective supports. Hence there is a map $\gamma \colon \mathds{R}^{n} \rightarrow \mathds{R}^{n}$ of the form $\gamma(x) = \nabla(G(x) - \vert x \vert^{2}/2)$ for some convex function $G \colon \mathds{R}^{n} \rightarrow \mathds{R}$, uniquely defined on and supported by the support of $P_{0}$, and smooth in the interior of this set, such that $\gamma$ induces the optimal quadratic Wasserstein transport from $P_{0}$ to $P_{1}$ via
\begin{equation} \label{hlotlse}
T_{t}^{\gamma}(x) \vcentcolon = x + t \cdot \gamma(x) = (1-t) \cdot x + t \cdot \nabla G(x) \qquad \textnormal{and} \qquad P_{t} \vcentcolon = (T_{t}^{\gamma})_{\#}(P_{0}) = P_{0} \circ (T_{t}^{\gamma})^{-1}
\end{equation}
for $0 \leqslant t \leqslant 1$, and $T_{1}^{\gamma} = \nabla G $; to wit, the curve $(P_{t})_{0 \leqslant t \leqslant 1}$ is the displacement interpolation (constant speed geodesic) between $P_{0}$ and $P_{1}$, and we have along it the linear growth of the quadratic Wasserstein distance
\begin{equation} \label{otmitwsnp}
W_{2}(P_{0},P_{t}) = t \, \sqrt{\int_{\mathds{R}^{n}} \vert x - \nabla G(x) \vert^{2} \, \textnormal{d} P_{0}(x)} = t \, \| \gamma \|_{L^{2}(P_{0})}, \qquad 0 \leqslant t \leqslant 1.
\end{equation}
The ($P_{0}$-almost everywhere) unique gradient $T \vcentcolon = \nabla G$ of a convex function pushing $P_{0}$ forward to $P_{1}$, i.e., $T_{\#}(P_{0}) = P_{1}$, and having the optimality property \hyperref[otmitwsnp]{(\ref*{otmitwsnp})} with respect to the quadratic Wasserstein distance, is called the \textit{Brenier map}; see \cite[Theorem 2.12]{Vil03}.
\end{enumerate}
\end{assumptions}

\begin{remark} \label{saeog} 
For the existence and uniqueness of the optimal transport map $\gamma \colon \mathds{R}^{n} \rightarrow \mathds{R}^{n}$ we refer to \cite[Theorem 2.12]{Vil03}, and for its smoothness to \cite[Theorem 4.14]{Vil03} as well as \cite[Remarks 4.15]{Vil03}. These results are known collectively under the rubric of \textit{Brenier's theorem} \cite{Bre91}.
\end{remark}

\begin{remark} We remark at this point, that we have chosen the subscript notation for $P_{t}$ in order to avoid confusion with the probability measure $P(t)$ from our \hyperref[snaas]{Section \ref*{snaas}} here. While $P_{0} = P(0)$, the flow $(P_{t})_{0 \leqslant t \leqslant 1}$ from $P_{0}$ to $P_{1}$ will have otherwise very little to do with the flow $(P(t))_{t \geqslant 0}$ from $P(0)$ to $\mathrm{Q}$ appearing in \hyperref[thetone]{Theorems \ref*{thetone}} and \hyperref[thetthre]{\ref*{thetthre}} (except for the tangential relation at $P_{0} = P(0)$). Similarly, the likelihood ratio function 
\begin{equation} \label{nlrfitramfcc}
\ell_{t}(x) = \frac{p_{t}(x)}{q(x)}, \qquad (t,x) \in [0,1] \times \mathds{R}^{n},
\end{equation}
is different from $\ell(t, \, \cdot \,)$, as now $p_{t}(\, \cdot \,)$ is the density function of the probability measure $P_{t}$. 
\end{remark}

Let us now return to our general theme, where we consider the potential $\Psi$ and the (possibly only $\sigma$-finite) measure $\mathrm{Q}$ on the Borel sets of $\mathds{R}^{n}$ with density $q(x) = \mathrm{e}^{-2 \Psi(x)}$ for $x \in \mathds{R}^{n}$. 

\begin{lemma} \label{hlotl} Under the \textnormal{\hyperref[hwisosaojkoia]{Assumptions \ref*{hwisosaojkoia}}}, let $X_{0}$ be a random variable with probability distribution $P_{0} = P(0)$, defined on some probability space $(S,\mathcal{S},\nu)$. Then we have
\begin{equation} \label{hlotlte}
\lim_{t \downarrow 0} \frac{H(P_{t} \, \vert \, \mathrm{Q}) - H(P_{0} \, \vert \, \mathrm{Q})}{t} = \big\langle \nabla \log \ell_{0}(X_{0}) \, , \, \gamma(X_{0}) \big\rangle_{L^{2}(\nu)}.
\end{equation}
\end{lemma}

\begin{remark} \label{rrewdfcspgpst} The relative entropy $H( P \, \vert \, \mathrm{Q})$ is well-defined for every $P \in \mathscr{P}_{2}(\mathds{R}^{n})$, and takes values in $(-\infty,\infty]$; see \hyperref[wdadotrewrttmq]{Appendix \ref*{wdadotrewrttmq}}. As the displacement interpolation $(P_{t})_{0 \leqslant t \leqslant 1}$ is the constant-speed geodesic joining the probability measures $P_{0}$ and $P_{1}$ in $\mathscr{P}_{2}(\mathds{R}^{n})$, we see that the relative entropy $H( P_{t} \, \vert \, \mathrm{Q})$ is well-defined for every $t \in [0,1]$.
\end{remark} 

We relegate the proof of \hyperref[hlotl]{Lemma \ref*{hlotl}}, which follows a similar (but considerably simpler) line of reasoning as the proof of \hyperref[thetthre]{Theorem \ref*{thetthre}}, to \hyperref[polhlotl]{Appendix \ref*{polhlotl}}. Combining \hyperref[hlotl]{Lemma \ref*{hlotl}} with well-known arguments, in particular, with a fundamental result on displacement convexity due to McCann \cite{McC97}, we derive now the HWI inequality of Otto and Villani \cite{OV00} and Cordero-Erausquin \cite{CE02}; see also \cite[p.\ 650]{Vil09}. This result relates the fundamental quantities of relative entropy (H), Wasserstein distance (W) and relative Fisher information (I). 

\begin{proposition}[\textsf{HWI inequality \cite{OV00}}] \label{hwiai} Under the \textnormal{\hyperref[sosaojkoia]{Assumptions \ref*{sosaojkoia}}}, we set $P_{0} = P(0)$, fix $P_{1} \in \mathscr{P}_{2}(\mathds{R}^{n})$ with finite relative entropy $H( P_{1} \, \vert \, \mathrm{Q})$, and let $\gamma$ be as in \textnormal{\hyperref[hlotlse]{(\ref*{hlotlse})}}. We suppose in addition that the potential $\Psi \colon \mathds{R}^{n} \rightarrow [0,\infty)$ satisfies a curvature lower bound 
\begin{equation} \label{sndcbe}
\textnormal{Hess}(\Psi) \geqslant \kappa \, I_{n},
\end{equation}
for some $\kappa \in \mathds{R}$. Then we have
\begin{equation} \label{sivothwii}
H( P_{0} \, \vert \, \mathrm{Q} ) - H( P_{1} \, \vert \, \mathrm{Q} ) \leqslant - \big\langle \nabla \log \ell_{0}(X_{0}) \, , \, \gamma(X_{0}) \big\rangle_{L^{2}(\nu)} - \tfrac{\kappa}{2} \, W_{2}^{2}(P_{0},P_{1}),
\end{equation}
where the random variable $X_{0}$, the likelihood ratio function $\ell_{0}$, and the probability measure $\nu$, are as in \textnormal{\hyperref[hlotl]{Lemma \ref*{hlotl}}}.
\end{proposition}

\begin{remark} Let us stress that \hyperref[hwiai]{Proposition \ref*{hwiai}} does not require $\mathrm{Q}$ to be a probability measure in the formulation of the HWI inequality \hyperref[sivothwii]{(\ref*{sivothwii})}.
\end{remark}

On the strength of the Cauchy-Schwarz inequality, we have
\begin{equation}
- \big\langle \nabla \log \ell_{0}(X_{0}) \, , \, \gamma(X_{0}) \big\rangle_{L^{2}(\nu)} \leqslant \| \nabla \log \ell_{0}(X_{0}) \|_{L^{2}(\nu)} \ \| \gamma(X_{0}) \|_{L^{2}(\nu)},
\end{equation}
with equality if and only if the functions $\nabla \log \ell_{0}(\, \cdot \,)$ and $\gamma(\, \cdot \,)$ are negatively collinear. Now the relative Fisher information of $P_{0}$ with respect to $\mathrm{Q}$ equals 
\begin{equation} 
I( P_{0} \, \vert \, \mathrm{Q}) = \mathds{E}_{\nu}\Big[ \vert \nabla \log \ell_{0}(X_{0}) \vert^{2} \Big] = \| \nabla \log \ell_{0}(X_{0}) \|_{L^{2}(\nu)}^{2},
\end{equation}
and by \textit{Brenier's theorem} \cite[Theorem 2.12]{Vil03} we deduce
\begin{equation}
\| \gamma(X_{0}) \|_{L^{2}(\nu)} = W_{2}(P_{0},P_{1})
\end{equation}
as in \hyperref[otmitwsnp]{(\ref*{otmitwsnp})}, along with the inequality
\begin{equation} \label{csfhwi} 
- \big\langle \nabla \log \ell_{0}(X_{0}) \, , \, \gamma(X_{0}) \big\rangle_{L^{2}(\nu)} \leqslant \sqrt{I( P_{0} \,  \vert \, \mathrm{Q} )} \ W_{2}(P_{0},P_{1}).
\end{equation}
Inserting \hyperref[csfhwi]{(\ref*{csfhwi})} into \hyperref[sivothwii]{(\ref*{sivothwii})} we obtain the usual form of the HWI inequality 
\begin{equation} \label{hwiast}
H( P_{0} \, \vert \, \mathrm{Q} ) - H( P_{1} \, \vert \, \mathrm{Q} ) \leqslant W_{2}(P_{0},P_{1}) \ \sqrt{I( P_{0} \,  \vert \, \mathrm{Q} )} - \tfrac{\kappa}{2} \, W_{2}^{2}(P_{0},P_{1}).
\end{equation}
When there is a non-trivial angle between $\nabla \log \ell_{0}(X_{0})$ and $\gamma(X_{0})$ in $L^{2}(\nu)$, the inequality \hyperref[sivothwii]{(\ref*{sivothwii})} gives a sharper bound than \hyperref[hwiast]{(\ref*{hwiast})}. We refer to the original paper \cite{OV00}, as well as to \cite{CE02}, \cite[Chapter 5]{Vil03}, \cite[p.\ 650]{Vil09} and the recent paper \cite{GLRT20}, for a detailed discussion of the HWI inequality. For a good survey on transport inequalities, see \cite{GL10}. 

\begin{remark} \label{logsobaoieq} Let us suppose now that the strong non-degeneracy condition \hyperref[sndcbe]{(\ref*{sndcbe})} holds with $\kappa > 0$, and that $\mathrm{Q}$ is a probability measure in $\mathscr{P}_{2}(\mathds{R}^{n})$. Then the inequality \hyperref[hwiast]{(\ref*{hwiast})} contains as special cases the \textit{Talagrand} \cite{Tal96} and \textit{logarithmic Sobolev} \cite{Fed69,Gro75} \textit{inequalities}, namely
\begin{equation} \label{lsiloggro}
W_{2}^{2}(P,\mathrm{Q}) \leqslant \tfrac{2}{\kappa} \, H(P \, \vert \, \mathrm{Q}), \qquad H(P \, \vert \, \mathrm{Q}) \leqslant \tfrac{1}{2\kappa} \, I(P \, \vert \, \mathrm{Q}),
\end{equation}
respectively; just by reading \hyperref[hwiast]{(\ref*{hwiast})} first with $(P_{0} , P_{1}) = (\mathrm{Q}, P)$, then with $(P_{0} , P_{1}) = (P,\mathrm{Q})$ and applying Young's inequality $xy \leqslant x^{2}/2 + y^{2}/2$, which is valid for all $x,y \in \mathds{R}$. On the other hand, and now in the context of \hyperref[snaas]{Section \ref*{snaas}}, the second inequality in \hyperref[lsiloggro]{(\ref*{lsiloggro})} leads, in conjunction with the generalized de Bruijn identity \hyperref[flfffl]{(\ref*{flfffl})} and \hyperref[rfi]{(\ref*{rfi})}, to
\begin{equation}
\tfrac{\textnormal{d}}{\textnormal{d}t} \, H\big( P(t) \, \vert \, \mathrm{Q} \big) \leqslant - \kappa  \, H\big( P(t) \, \vert \, \mathrm{Q} \big),
\end{equation}
and thence to the Bakry-{\'E}mery \cite{BE85} \textit{exponential temporal dissipation of the relative entropy}
\begin{equation}
H\big( P(t) \, \vert \, \mathrm{Q} \big) \leqslant H\big( P(t_{0}) \, \vert \, \mathrm{Q} \big) \, \mathrm{e}^{-\kappa(t-t_{0})}, \qquad t \geqslant t_{0}
\end{equation}
as well as of the Wasserstein distance $W_{2}(P(t),\mathrm{Q})$ on account of \hyperref[lsiloggro]{(\ref*{lsiloggro})}. For an exposition of the Bakry-{\'E}mery theory, which derives also the \textit{exponential temporal dissipation of the relative Fisher information} in the context of \hyperref[snaas]{Section \ref*{snaas}}, we refer to \cite{BGL14} and \cite{Gen14}.
\end{remark}

The inequality \hyperref[hwiast]{(\ref*{hwiast})} is yet another manifestation of the interplay between displacement in the ambient space of probability measures (the quantity $W_{2}(P_{0},P_{1})$) and fluctuations of the relative entropy (the quantity $H( P_{0} \, \vert \, \mathrm{Q} ) - H( P_{1} \, \vert \, \mathrm{Q} )$) as governed by the square root of the Fisher information $\sqrt{I( P_{0} \,  \vert \, \mathrm{Q} )}$, very much in the spirit of \hyperref[ttofeosl]{(\ref*{ttofeosl})}.

\begin{proof}[Proof of \texorpdfstring{\hyperref[hwiai]{Proposition \ref*{hwiai}}}] As elaborated in \cite[Section 9.4]{Vil03} we may assume without loss of generality that $P_{0}$ and $P_{1}$ satisfy the strong regularity \hyperref[hwisosaojkoia]{Assumptions \ref*{hwisosaojkoia}}. For the existence and smoothness of the optimal transport map $\gamma$ we refer to \hyperref[saeog]{Remark \ref*{saeog}}.

We consider now the relative entropy with respect to $\mathrm{Q}$ along the constant-speed geodesic $(P_{t})_{0 \leqslant t \leqslant 1}$, namely, the function $f(t) \vcentcolon = H( P_{t} \, \vert \, \mathrm{Q})$, for $0 \leqslant t \leqslant 1$. We show that the displacement convexity results of McCann \cite{McC97} imply 
\begin{equation} \label{dcromcc}
f''(t) \geqslant \kappa \, W_{2}^{2}(P_{0},P_{1}), \qquad 0 \leqslant t \leqslant 1.
\end{equation}

Indeed, under the condition \hyperref[sndcbe]{(\ref*{sndcbe})}, the potential $\Psi$ is $\kappa$-uniformly convex. Consequently, by items (i) and (ii) of \cite[Theorem 5.15]{Vil03}, the internal and potential energies
\begin{equation}
g(t) \vcentcolon = \int_{\mathds{R}^{n}} p_{t}(x) \log p_{t}(x) \, \textnormal{d}x, \qquad h(t) \vcentcolon = 2 \int_{\mathds{R}^{n}} \Psi(x) \, p_{t}(x) \, \textnormal{d}x, \qquad 0 \leqslant t \leqslant 1,
\end{equation}
are displacement convex and $\kappa$-uniformly displacement convex, respectively; i.e.,
\begin{equation} 
g''(t) \geqslant 0, \qquad h''(t) \geqslant \kappa \, W_{2}^{2}(P_{0},P_{1}), \qquad 0 \leqslant t \leqslant 1.
\end{equation}
By analogy with \hyperref[llreeef]{Lemma \ref*{llreeef}} we have $f = g + h$, and conclude that the relative entropy function $f$ is $\kappa$-uniformly displacement convex, i.e., its second derivative satisfies \hyperref[dcromcc]{(\ref*{dcromcc})}.

We appeal now to \hyperref[hlotl]{Lemma \ref*{hlotl}}, according to which we have
\begin{equation} \label{irf}
f'(0+) = \lim_{t \downarrow 0} \, \frac{f(t)-f(0)}{t} = \big\langle \nabla \log \ell_{0}(X_{0}) \, , \, \gamma(X_{0}) \big\rangle_{L^{2}(\nu)}.
\end{equation}
In conjunction with \hyperref[dcromcc]{(\ref*{dcromcc})} and \hyperref[irf]{(\ref*{irf})}, the Taylor formula $f(1) = f(0) + f'(0+) + \int_{0}^{1} (1-t) f''(t) \, \textnormal{d}t$ now yields \hyperref[sivothwii]{(\ref*{sivothwii})}.
\end{proof}


\section{Details and proofs} \label{ssgrodwudm}


In this section we complete the proofs of \hyperref[thetsixcoraopv]{Corollary \ref*{thetsixcoraopv}} and \hyperref[thetthretvcor]{Proposition \ref*{thetthretvcor}}, and provide the proofs of our main results, \hyperref[thetsix]{Theorems \ref*{thetsix}} and \hyperref[thetthretv]{\ref*{thetthretv}}. In fact, all we have to do in order to prove these latter theorems is to apply It\^{o}'s formula so as to calculate the dynamics, i.e., the stochastic differentials, of the ``pure'' and ``perturbed'' relative entropy processes of \hyperref[stlrpd]{(\ref*{stlrpd})} and \hyperref[prerepitufdllpitufd]{(\ref*{prerepitufdllpitufd})} under the measures $\mathds{P}$ and $\mathds{P}^{\beta}$, respectively. We may (and shall) do this in both the forward and, most importantly, the backward, directions of time.

\smallskip

However, such a brute-force approach does not provide any hint as to why we obtain the remarkable form of the drift term of the time-reversed relative entropy process
\begin{equation} \label{ttrrep}
\log \ell\big(T-s,X(T-s)\big) = \log \Bigg( \frac{p\big(T-s,X(T-s)\big)}{q\big(X(T-s)\big)} \Bigg) \, , \qquad 0 \leqslant s \leqslant T,
\end{equation}
as stated in \hyperref[thetsix]{Theorem \ref*{thetsix}}, namely
\begin{equation} \label{sdotpttrrep}
\textnormal{d}\log \ell\big(T-s,X(T-s)\big)
= \Bigg\langle \frac{\nabla \ell\big(T-s,X(T-s)\big)}{\ell\big(T-s,X(T-s)\big)} \, , \, \textnormal{d}\overline{W}^\mathds{P}(T-s) \Bigg\rangle_{\mathds{R}^{n}} \mkern-9mu + \frac{1}{2} \frac{\big\vert \nabla \ell\big(T-s,X(T-s)\big) \big\vert^{2}}{\ell\big(T-s,X(T-s)\big)^{2}} \, \textnormal{d}s,
\end{equation}
for $0 \leqslant s \leqslant T$, with respect to the backwards filtration $(\mathcal{G}(T-s))_{0 \leqslant s \leqslant T}$. Therefore, in order to motivate and illustrate the derivation of the dynamics \hyperref[sdotpttrrep]{(\ref*{sdotpttrrep})}, we first impose the additional assumption $\mathrm{Q}(\mathds{R}^{n}) < \infty$ (which precludes the case $\Psi \equiv 0$), so as to conform to the setting of \cite{FJ16}. This is done in \hyperref[tcqrnliwqis]{Subsection \ref*{tcqrnliwqis}}, which serves purely as motivation; in the remainder of this paper we do not rely on the assumption $\mathrm{Q}(\mathds{R}^{n}) < \infty$.


\subsection{Some preliminaries} 


Our first task is to calculate the dynamics of the time-reversed relative entropy process \hyperref[ttrrep]{(\ref*{ttrrep})} under the probability measure $\mathds{P}$. In order to do this, we start by calculating the stochastic differential of the time-reversed canonical coordinate process $(X(T-s))_{0 \leqslant s \leqslant T}$ under $\mathds{P}$, which is a well-known and classical theme; see e.g.\ \cite{Foe85,Foe86}, \cite{HP86}, \cite{Mey94}, \cite{Nel01}, and \cite{Par86}. For the convenience of the reader we present the theory of time reversal for diffusion processes in \hyperref[atrod]{Appendix \ref*{atrod}}. The idea of time reversal goes back to the ideas of Boltzmann \cite{Bol96,Bol98a,Bol98b} and Schr\"odinger \cite{Sch31,Sch32}, as well as Kolmogorov \cite{Kol37}. In fact, as we shall recall in \hyperref[ahnbwrbmtthe]{Appendix \ref*{ahnbwrbmtthe}}, the relation between time reversal of a Brownian motion and the quadratic Wasserstein distance may \textit{in nuce} be traced back to an insight of Bachelier in his thesis \cite{Bac00,Bac06} from 1900; at least, when admitting a good portion of wisdom of hindsight.

\smallskip

Recall that the probability measure $\mathds{P}$ was defined on the path space $\Omega = \mathcal{C}(\mathds{R}_{+};\mathds{R}^{n})$ so that the canonical coordinate process $(X(t,\omega))_{t \geqslant 0} = (\omega(t))_{t \geqslant 0}$ satisfies the stochastic differential equation \hyperref[sdeids]{(\ref*{sdeids})} with initial probability distribution $P(0)$ for $X(0)$ under $\mathds{P}$. In other words, the process
\begin{equation} \label{dobmwopsrss}
W(t) = X(t) - X(0) + \int_{0}^{t} \nabla \Psi\big(X(u)\big) \, \textnormal{d}u, \qquad t \geqslant 0
\end{equation}
defines a Brownian motion of the forward filtration $(\mathcal{F}(t))_{t \geqslant 0}$ under the probability measure $\mathds{P}$, where the integral in \hyperref[dobmwopsrss]{(\ref*{dobmwopsrss})} is to be understood in a pathwise Riemann-Stieltjes sense. Passing to the reverse direction of time, the following classical result is well known to hold under the \hyperref[sosaojkoia]{Assumptions \ref*{sosaojkoia}}.

\begin{proposition} \label{ptra} Under the \textnormal{\hyperref[sosaojkoia]{Assumptions \ref*{sosaojkoia}}}, we let $T > 0$. The process 
\begin{equation} \label{dotowptmtpbm}
\overline{W}^\mathds{P}(T-s) \vcentcolon = W(T-s) - W(T) - \int_{0}^{s} \nabla \log p\big(T-u,X(T-u)\big) \, \textnormal{d}u
\end{equation}
for $0 \leqslant s \leqslant T$, is a Brownian motion of the backwards filtration $(\mathcal{G}(T-s))_{0 \leqslant s \leqslant T}$ under the probability measure $\mathds{P}$. Moreover, the time-reversed canonical coordinate process $(X(T-s))_{0 \leqslant s \leqslant T}$ satisfies the stochastic differential equation
\begingroup
\addtolength{\jot}{0.7em}
\begin{align}
\textnormal{d} X(T-s) &= \Big( \nabla \log p\big(T-s,X(T-s)\big) + \nabla \Psi\big(X(T-s)\big) \Big) \, \textnormal{d}s + \textnormal{d}\overline{W}^{\mathds{P}}(T-s) \label{poortdfttrpp} \\    
 &= \Big( \nabla \log \ell\big(T-s,X(T-s)\big) - \nabla \Psi\big(X(T-s)\big) \Big) \, \textnormal{d}s + \textnormal{d}\overline{W}^{\mathds{P}}(T-s), \label{oortdfttrpp}
\end{align}
\endgroup
for $0 \leqslant s \leqslant T$, with respect to the backwards filtration $(\mathcal{G}(T-s))_{0 \leqslant s \leqslant T}$. 
\end{proposition}

We provide proofs and references for this result in \hyperref[Thm1]{Theorems \ref*{Thm1}} and \hyperref[Thm2]{\ref*{Thm2}} of \hyperref[atrod]{Appendix \ref*{atrod}}.

\smallskip

Before proving \hyperref[thetsix]{Theorem \ref*{thetsix}} in \hyperref[cotpottpf]{Subsection \ref*{cotpottpf}} --- as already announced --- we digress now to present the following didactic, illuminating and important special case.


\subsection{The case \texorpdfstring{$\mathrm{Q}(\mathds{R}^{n}) < \infty$}{where Q is a probability measure}} \label{tcqrnliwqis}


We shall impose, for the purposes of the present subsection only, the additional assumption
\begin{equation} \label{wsiftpotpsotaa} 
\mathrm{Q}(\mathds{R}^{n}) = \int_{\mathds{R}^{n}} \mathrm{e}^{-2\Psi(x)} \, \textnormal{d}x < \infty.
\end{equation}
Under this assumption, the measure $\mathrm{Q}$ on the Borel sets of $\mathds{R}^{n}$, introduced in \hyperref[snaas]{Section \ref*{snaas}}, is finite and can thus be re-normalized, so as to become a probability measure. In this manner, it induces a probability measure $\mathds{Q}$ on the path space $\Omega = \mathcal{C}(\mathds{R}_{+};\mathds{R}^{n})$, under which the canonical coordinate process $(X(t,\omega))_{t \geqslant 0} = (\omega(t))_{t \geqslant 0}$ satisfies the stochastic equation \hyperref[sdeids]{(\ref*{sdeids})} with initial probability distribution $\mathrm{Q}$ for $X(0)$. And because this distribution is invariant, it is also the distribution of $X(t)$ under $\mathds{Q}$ for every $t \geqslant 0$.

\medskip

For the present authors, the eye-opener leading to \hyperref[sdotpttrrep]{(\ref*{sdotpttrrep})} was the subsequent remarkable insight due to Fontbona and Jourdain \cite{FJ16}. This provided us with much of the original motivation to start this line of research. The result right below holds in much greater generality (essentially one only needs the Markovian structure of the process $(X(t))_{t \geqslant 0}$) but we only state it in the present setting given by \hyperref[sdeids]{(\ref*{sdeids})} under the \hyperref[sosaojkoia]{Assumptions \ref*{sosaojkoia}} and $\mathrm{Q}(\mathds{R}^{n}) = 1$ in \hyperref[wsiftpotpsotaa]{(\ref*{wsiftpotpsotaa})}. For another application of time reversal in a similar context, see \cite{Leo17}.

\begin{theorem}[\textsf{Fontbona-Jourdain theorem \cite{FJ16}}] \label{ovtfjofmidpwcd} Under the \textnormal{\hyperref[sosaojkoia]{Assumptions \ref*{sosaojkoia}}} and $\mathrm{Q}(\mathds{R}^{n}) = 1$, we fix $T \in (0,\infty)$. The time-reversed likelihood ratio process of \textnormal{\hyperref[ttrrep]{(\ref*{ttrrep})}} is a martingale of the backwards filtration $(\mathcal{G}(T-s))_{0 \leqslant s \leqslant T}$ under the probability measure $\mathds{Q}$.
\end{theorem}

\begin{corollary} \label{cortrlrrepzrp} Under the \textnormal{\hyperref[sosaojkoia]{Assumptions \ref*{sosaojkoia}}} and $\mathrm{Q}(\mathds{R}^{n}) = 1$, we fix $T \in (0,\infty)$. The time-reversed process 
\begin{equation} \label{trlrrepzrp}
\ell\big(T-s,X(T-s)\big) \cdot \log \ell\big(T-s,X(T-s)\big) \, , \qquad 0 \leqslant s \leqslant T
\end{equation}
is a submartingale of the backwards filtration $(\mathcal{G}(T-s))_{0 \leqslant s \leqslant T}$ under the probability measure $\mathds{Q}$. In particular, we have
\begin{equation} \label{dotrefuc}
H\big( P(t) \, \vert \, \mathrm{Q} \big) \leqslant H\big( P(0) \, \vert \, \mathrm{Q} \big), \qquad 0 \leqslant t \leqslant T.
\end{equation}
\begin{proof} This is an immediate consequence of \hyperref[ovtfjofmidpwcd]{Theorem \ref*{ovtfjofmidpwcd}}, Jensen's inequality for conditional expectations, and the convexity of the function $f(x) = x \log x$, $x > 0$.
\end{proof}
\end{corollary}

For the convenience of the reader we recall in \hyperref[atfjo]{Appendix \ref*{atfjo}} the surprisingly straightforward proof of \hyperref[ovtfjofmidpwcd]{Theorem \ref*{ovtfjofmidpwcd}}. Since this result states that the time-reversed likelihood ratio process \hyperref[ttrrep]{(\ref*{ttrrep})} is a $\mathds{Q}$-martingale with respect to the backwards filtration $(\mathcal{G}(T-s))_{0 \leqslant s \leqslant T}$, we will first state the analogue of \hyperref[ptra]{Proposition \ref*{ptra}} in terms of the probability measure $\mathds{Q}$ on the path space $\Omega = \mathcal{C}(\mathds{R}_{+};\mathds{R}^{n})$, which is induced by the invariant probability distribution $\mathrm{Q}$ on $\mathds{R}^{n}$.

\begin{proposition} \label{qptra} Under the \textnormal{\hyperref[sosaojkoia]{Assumptions \ref*{sosaojkoia}}} and $\mathrm{Q}(\mathds{R}^{n}) = 1$, we fix $T \in (0,\infty)$. The process 
\begin{equation} \label{qdotowptmtpbm}
\overline{W}^\mathds{Q}(T-s) \vcentcolon = W(T-s) - W(T) + 2  \int_{0}^{s} \nabla \Psi\big( X(T-u)\big) \, \textnormal{d}u
\end{equation}
for $0 \leqslant s \leqslant T$, is a Brownian motion of the backwards filtration $(\mathcal{G}(T-s))_{0 \leqslant s \leqslant T}$ under the probability measure $\mathds{Q}$. Moreover, the time-reversed canonical coordinate process $(X(T-s))_{0 \leqslant s \leqslant T}$ satisfies the stochastic differential equation
\begin{equation} \label{qoortdfttrpp}
\textnormal{d} X(T-s) = - \nabla \Psi\big(X(T-s)\big) \, \textnormal{d}s + \textnormal{d}\overline{W}^{\mathds{Q}}(T-s),
\end{equation}
for $0 \leqslant s \leqslant T$, with respect to the backwards filtration $(\mathcal{G}(T-s))_{0 \leqslant s \leqslant T}$. 
\end{proposition}

Again, for the proof of this result, we refer to \hyperref[Thm1]{Theorems \ref*{Thm1}} and \hyperref[Thm2]{\ref*{Thm2}} of \hyperref[atrod]{Appendix \ref*{atrod}}. 

\smallskip

In the following lemma we determine the drift term that allows one to pass from the $\mathds{Q}$-Brownian motion $\big(\overline{W}^\mathds{Q}(T-s)\big)_{0 \leqslant s \leqslant T}$ to the $\mathds{P}$-Brownian motion $\big(\overline{W}^\mathds{P}(T-s)\big)_{0 \leqslant s \leqslant T}$, and vice versa.

\begin{lemma} Under the \textnormal{\hyperref[sosaojkoia]{Assumptions \ref*{sosaojkoia}}} and $\mathrm{Q}(\mathds{R}^{n}) = 1$, we fix $T \in (0,\infty)$. For $0 \leqslant s \leqslant T$, we have
\begin{equation} \label{dtwrtpaqbmbm}
\textnormal{d}\overline{W}^\mathds{Q}(T-s) = \nabla \log \ell\big(T-s,X(T-s)\big) \, \textnormal{d}s + \textnormal{d}\overline{W}^\mathds{P}(T-s).
\end{equation}
\begin{proof} One just has to compare the equations \hyperref[oortdfttrpp]{(\ref*{oortdfttrpp})} and \hyperref[qoortdfttrpp]{(\ref*{qoortdfttrpp})}.
\end{proof}
\end{lemma}

The next corollary is a direct consequence of \hyperref[ovtfjofmidpwcd]{Theorem \ref*{ovtfjofmidpwcd}}, \hyperref[qptra]{Proposition \ref*{qptra}} and It\^{o}'s formula.

\begin{corollary} \label{fcotfjtabmgs} Under the \textnormal{\hyperref[sosaojkoia]{Assumptions \ref*{sosaojkoia}}} and $\mathrm{Q}(\mathds{R}^{n}) = 1$, we fix $T \in (0,\infty)$. The time-reversed likelihood ratio process \textnormal{\hyperref[ttrrep]{(\ref*{ttrrep})}} and its logarithm satisfy the stochastic differential equations
\begin{equation} \label{sdefcotfjtabmgs}
\textnormal{d}\ell \big(T-s,X(T-s)\big) = \Big\langle \nabla \ell\big(T-s,X(T-s)\big) \, , \, \textnormal{d}\overline{W}^\mathds{Q}(T-s) \Big\rangle_{\mathds{R}^{n}}
\end{equation}
and
\begin{equation} \label{sdeftpmpaei}
\textnormal{d}\log \ell\big(T-s,X(T-s)\big) = \Bigg\langle \frac{\nabla \ell\big(T-s,X(T-s)\big)}{\ell\big(T-s,X(T-s)\big)} \, , \, \textnormal{d}\overline{W}^\mathds{Q}(T-s)\Bigg\rangle_{\mathds{R}^{n}} \mkern-9mu - \frac{1}{2} \frac{\big\vert \nabla \ell\big(T-s,X(T-s)\big) \big\vert^{2}}{\ell\big(T-s,X(T-s)\big)^{2}} \, \textnormal{d}s,
\end{equation}
respectively, for $0 \leqslant s \leqslant T$, with respect to the backwards filtration $(\mathcal{G}(T-s))_{0 \leqslant s \leqslant T}$. 
\begin{proof} To prove \hyperref[sdefcotfjtabmgs]{(\ref*{sdefcotfjtabmgs})}, the decisive insight is provided by \hyperref[ovtfjofmidpwcd]{Theorem \ref*{ovtfjofmidpwcd}} due to Fontbona and Jourdain \cite{FJ16}. This implies that the drift term in \hyperref[sdefcotfjtabmgs]{(\ref*{sdefcotfjtabmgs})} must vanish, so that it suffices to calculate the diffusion term in front of $\textnormal{d}\overline{W}^\mathds{Q}(T-s)$ in \hyperref[sdefcotfjtabmgs]{(\ref*{sdefcotfjtabmgs})}; using \hyperref[qoortdfttrpp]{(\ref*{qoortdfttrpp})}, this is a straightforward task.

\smallskip

We note that the vanishing of the drift term in \hyperref[sdefcotfjtabmgs]{(\ref*{sdefcotfjtabmgs})} can also be obtained from applying It\^{o}'s formula to the process \hyperref[ttrrep]{(\ref*{ttrrep})}, using \hyperref[qoortdfttrpp]{(\ref*{qoortdfttrpp})} as well as the backwards Kolmogorov equation \hyperref[fpdefflrf]{(\ref*{fpdefflrf})} for the likelihood ratio function $\ell(t,x)$ and observing that all these terms cancel out. But such a procedure does not provide a hint as to why this miracle happens.

\smallskip

Having said this, we apply It\^{o}'s formula to the process \hyperref[ttrrep]{(\ref*{ttrrep})} and use \hyperref[ovtfjofmidpwcd]{Theorem \ref*{ovtfjofmidpwcd}} to obtain \hyperref[sdefcotfjtabmgs]{(\ref*{sdefcotfjtabmgs})}. Assertion \hyperref[sdeftpmpaei]{(\ref*{sdeftpmpaei})} follows once again by applying It\^{o}'s formula to the logarithm of the process \hyperref[ttrrep]{(\ref*{ttrrep})}, and using the dynamics \hyperref[sdefcotfjtabmgs]{(\ref*{sdefcotfjtabmgs})}.
\end{proof}
\end{corollary} 

We have now all the ingredients in order to compute, under the additional assumption $\mathrm{Q}(\mathds{R}^{n}) = 1$, the dynamics of the time-reversed relative entropy process \hyperref[ttrrep]{(\ref*{ttrrep})} under the probability measure $\mathds{P}$. Indeed, substituting \hyperref[dtwrtpaqbmbm]{(\ref*{dtwrtpaqbmbm})} into the stochastic equation \hyperref[sdeftpmpaei]{(\ref*{sdeftpmpaei})}, we see that the process \hyperref[ttrrep]{(\ref*{ttrrep})} satisfies the crucial stochastic differential equation \hyperref[sdotpttrrep]{(\ref*{sdotpttrrep})}, for $0 \leqslant s \leqslant T$, with respect to the backwards filtration $(\mathcal{G}(T-s))_{0 \leqslant s \leqslant T}$.


\subsection{The proof of \texorpdfstring{\hyperref[thetsix]{Theorem \ref*{thetsix}}}{Theorem 3.6}} \label{cotpottpf}


We drop now the assumption $\mathrm{Q}(\mathds{R}^{n}) < \infty$, and write the Fokker-Planck equation \hyperref[fpeqnwfp]{(\ref*{fpeqnwfp})} as
\begin{equation} \label{pdeftupdpp}
\partial_{t} p(t,x) = \tfrac{1}{2} \Delta p(t,x) + \big\langle \nabla p(t,x) \, , \nabla \Psi(x)\big\rangle_{\mathds{R}^{n}} + p(t,x) \, \Delta \Psi(x), \qquad t > 0.
\end{equation}
The probability density function $p(t,x)$ can be represented in the form
\begin{equation} \label{tpdfptxcbritf}
p(t,x) = \ell(t,x) \, q(x) = \ell(t,x) \, \mathrm{e}^{ - 2 \Psi(x)}, \qquad t \geqslant 0,
\end{equation}
so we find that the likelihood ratio function $\ell(t,x)$ solves the \textit{backwards} Kolmogorov equation
\begin{equation} \label{fpdefflrf}
\partial_{t} \ell(t,x) = \tfrac{1}{2} \Delta \ell(t,x) - \big\langle \nabla \ell(t,x) \, , \nabla \Psi(x) \big\rangle_{\mathds{R}^{n}}, \qquad t > 0,
\end{equation}
a feature consonant with the fact that the dynamics of the likelihood ratio process are most transparent under \textit{time reversal}. This equation will allow us to develop an alternative way of deriving the dynamics \hyperref[sdotpttrrep]{(\ref*{sdotpttrrep})} and proving \hyperref[thetsix]{Theorem \ref*{thetsix}}, which does not rely on assumption \hyperref[wsiftpotpsotaa]{(\ref*{wsiftpotpsotaa})} and uses exclusively the probability measure $\mathds{P}$, as follows.

\begin{proof}[\bfseries \upshape Proof of \texorpdfstring{\hyperref[thetsix]{Theorem \ref*{thetsix}}}{}] We have to show that the stochastic process $(M(T-s))_{0 \leqslant s \leqslant T}$ in \hyperref[dollafipim]{(\ref*{dollafipim})} is a $\mathds{P}$-martingale of the backwards filtration $(\mathcal{G}(T-s))_{0 \leqslant s \leqslant T}$, is bounded in $L^{2}(\mathds{P})$, and admits the integral representation \hyperref[erdollafipim]{(\ref*{erdollafipim})}. 

\medskip

\noindent \fbox{\textsf{Step 1.}} Applying It\^{o}'s formula to the time-reversed likelihood ratio process \hyperref[ttrrep]{(\ref*{ttrrep})}, and using \hyperref[oortdfttrpp]{(\ref*{oortdfttrpp})} as well as the backwards Kolmogorov equation \hyperref[fpdefflrf]{(\ref*{fpdefflrf})} for the likelihood ratio function $\ell(t,x)$, we obtain the stochastic differential equation
\begin{equation} \label{qubmcfrsdeftpmpaei}
\frac{\textnormal{d} \ell\big(T-s,X(T-s)\big)}{\ell\big(T-s,X(T-s)\big)} = \Bigg\langle \frac{\nabla \ell\big(T-s,X(T-s)\big)}{\ell\big(T-s,X(T-s)\big)} \, , \, \textnormal{d}\overline{W}^\mathds{P}(T-s) \Bigg\rangle_{\mathds{R}^{n}} \mkern-9mu + \frac{\big\vert \nabla \ell\big(T-s,X(T-s)\big) \big\vert^{2}}{\ell\big(T-s,X(T-s)\big)^{2}} \, \textnormal{d}s
\end{equation}
as well as its logarithmic version
\begin{equation} \label{qubmcfrsdeftpmpaeilogv}
\textnormal{d} \log \ell\big(T-s,X(T-s)\big) = \Big\langle \nabla \log \ell\big(T-s,X(T-s)\big) \, , \, \textnormal{d}\overline{W}^\mathds{P}(T-s) \Big\rangle_{\mathds{R}^{n}} + \tfrac{1}{2} \big\vert \nabla \log \ell\big(T-s,X(T-s)\big) \big\vert^{2}  \textnormal{d}s,
\end{equation}
for $0 \leqslant s \leqslant T$, with respect to the backwards filtration $(\mathcal{G}(T-s))_{0 \leqslant s \leqslant T}$. This equation right above is nothing other than the desired stochastic differential equation \hyperref[sdotpttrrep]{(\ref*{sdotpttrrep})} from the beginning of \hyperref[ssgrodwudm]{Section \ref*{ssgrodwudm}}. 

\medskip

\noindent \fbox{\textsf{Step 2.}} We show first, that the process $(M(T-s))_{0 \leqslant s \leqslant T}$ of \hyperref[dollafipim]{(\ref*{dollafipim})}, with integral representation \hyperref[erdollafipim]{(\ref*{erdollafipim})}, is a continuous local $\mathds{P}$-martingale of the backwards filtration $(\mathcal{G}(T-s))_{0 \leqslant s \leqslant T}$.

\smallskip

By condition \hyperref[tsaosaojko]{\ref*{tsaosaojko}} of \hyperref[sosaojkoia]{Assumptions \ref*{sosaojkoia}} the function $(0,\infty) \ni t \mapsto \nabla \log \ell(t,x)$ is continuous for every $x \in \mathds{R}^{n}$. Together with the continuity of the paths of the canonical coordinate process $(X(t))_{t \geqslant 0}$, this implies that 
\begin{equation} \label{twtcotpotccptitttfitisifoas}
\int_{0}^{T-\varepsilon} \frac{\big\vert \nabla \ell\big(T-u,X(T-u)\big) \big\vert^{2}}{\ell\big(T-u,X(T-u)\big)^{2}} \, \textnormal{d}u < \infty, \qquad \mathds{P}\textnormal{-almost surely},
\end{equation}
for every $0 < \varepsilon \leqslant T$. On account of \hyperref[twtcotpotccptitttfitisifoas]{(\ref*{twtcotpotccptitttfitisifoas})}, the sequence of stopping times (with respect to the backwards filtration)
\begin{equation} \label{twtcdostptitttfitisifoas}
\tau_{n} \vcentcolon = \inf \Bigg\{ t \geqslant 0 \colon \, \int_{0}^{t} \frac{\big\vert \nabla \ell\big(T-u,X(T-u)\big) \big\vert^{2}}{\ell\big(T-u,X(T-u)\big)^{2}} \, \textnormal{d}u \, \geqslant n \,  \Bigg\} \wedge T, \qquad n \in \mathds{N}_{0}
\end{equation}
is non-decreasing and converges $\mathds{P}$-almost surely to $T$. Hence, according to \hyperref[sdotpttrrep]{(\ref*{sdotpttrrep})} and the definition in \hyperref[dollafipim]{(\ref*{dollafipim})}, the stopped process $(M^{\tau_{n}}(T-s))_{0 \leqslant s \leqslant T}$ is a uniformly integrable $\mathds{P}$-martingale of the backwards filtration, for every $n \in \mathds{N}_{0}$.

\medskip

\noindent \fbox{\textsf{Step 3.}} To show that, in fact, the process $(M(T-s))_{0 \leqslant s \leqslant T}$ is a true $\mathds{P}$-martingale, we have to rely on the finite free energy condition \hyperref[ffecaoo]{(\ref*{ffecaoo})} which, in the light of \hyperref[llreeef]{Lemma \ref*{llreeef}}, asserts that the relative entropy $H( P(0) \, \vert \, \mathrm{Q} )$ is finite. 

\smallskip 

Taking expectations with respect to the probability measure $\mathds{P}$ in \hyperref[sdotpttrrep]{(\ref*{sdotpttrrep})} at time $s = \tau_{n}$, and noting from \textsf{Step 2} that the stopped process $(M^{\tau_{n}}(T-s))_{0 \leqslant s \leqslant T}$ is a true $\mathds{P}$-martingale with respect to the backwards filtration $(\mathcal{G}(T-s))_{0 \leqslant s \leqslant T}$, we get
\begingroup
\addtolength{\jot}{0.7em}
\begin{align}
\mathds{E}_{\mathds{P}}\Bigg[\int_{0}^{\tau_{n}} \frac{1}{2} \frac{\big\vert \nabla \ell\big(T-u,X(T-u)\big) \big\vert^{2}}{\ell\big(T-u,X(T-u)\big)^{2}} \,  \textnormal{d}u\Bigg] 
&=  H\big( P(T-\tau_{n}) \, \vert \, \mathrm{Q} \big) - H\big( P(T) \, \vert \, \mathrm{Q} \big) \label{gfdmlmpapitupc} \\
&\leqslant H\big( P(0) \, \vert \, \mathrm{Q} \big) - H\big( P(T) \, \vert \, \mathrm{Q} \big), \label{mlmpapitupc}
\end{align}
\endgroup
for every $n \in \mathds{N}_{0}$. The inequality in \hyperref[mlmpapitupc]{(\ref*{mlmpapitupc})} is justified by the decrease of the relative entropy function $t \mapsto H( P(t) \, \vert \, \mathrm{Q})$, which we have from \hyperref[dotrefuc]{(\ref*{dotrefuc})}. However, in \hyperref[cortrlrrepzrp]{Corollary \ref*{cortrlrrepzrp}} we had assumed that $\mathrm{Q}$ is a probability measure, which is in general not the case under the \hyperref[sosaojkoia]{Assumptions \ref*{sosaojkoia}}. Nevertheless, the $\sigma$-finite measure $\mathrm{Q}$ and the stochastic differential equation \hyperref[sdeids]{(\ref*{sdeids})} induce a $\sigma$-finite measure $\mathds{Q}$ on the path space $\Omega = \mathcal{C}(\mathds{R}_{+};\mathds{R}^{n})$, which is invariant in the sense that at all times $t \geqslant 0$ its marginal distributions are equal to $\mathrm{Q}$. For this path measure $\mathds{Q}$, conditional expectations with respect to the canonical forward and backwards filtrations are well-defined, see \cite{Leo14}. As a consequence, the martingale assertion of \hyperref[ovtfjofmidpwcd]{Theorem \ref*{ovtfjofmidpwcd}} makes sense also in this $\sigma$-finite setting; for the details we refer to \cite[Chapter 1]{Tsc19}. Since Jensen's inequality is also valid for conditional expectations with respect to the $\sigma$-finite path measure $\mathds{Q}$ with marginals $\mathrm{Q}$, \hyperref[cortrlrrepzrp]{Corollary \ref*{cortrlrrepzrp}} remains true under the \hyperref[sosaojkoia]{Assumptions \ref*{sosaojkoia}} without the additional requirement $\mathrm{Q}(\mathds{R}^{n}) =1$.
 
\smallskip

Now, as we have shown the inequality \hyperref[mlmpapitupc]{(\ref*{mlmpapitupc})}, we pass to the limit as $n \rightarrow \infty$. Since the non-decreasing sequence of stopping times $(\tau_{n})_{n \geqslant 0}$ from \hyperref[twtcdostptitttfitisifoas]{(\ref*{twtcdostptitttfitisifoas})} converges $\mathds{P}$-almost surely to $T$, we deduce from \hyperref[mlmpapitupc]{(\ref*{mlmpapitupc})} and the monotone convergence theorem that
\begin{equation} \label{fecwfoel}
\mathds{E}_{\mathds{P}}\Bigg[\int_{0}^{T} \frac{1}{2} \frac{\big\vert \nabla \ell\big(T-u,X(T-u)\big) \big\vert^{2}}{\ell\big(T-u,X(T-u)\big)^{2}} \, \textnormal{d}u\Bigg] \leqslant H\big( P(0) \, \vert \, \mathrm{Q} \big) - H\big( P(T) \, \vert \, \mathrm{Q} \big) < \infty,
\end{equation}
because the initial relative entropy $H( P(0) \, \vert \, \mathrm{Q})$ is finite by assumption and $H( P(T) \, \vert \, \mathrm{Q})$ cannot take the value $-\infty$; see \hyperref[llreeef]{Lemma \ref*{llreeef}} and \hyperref[wdadotrewrttmq]{Appendix \ref*{wdadotrewrttmq}}. From \hyperref[fecwfoel]{(\ref*{fecwfoel})} we finally deduce that the stochastic integral in \hyperref[sdotpttrrep]{(\ref*{sdotpttrrep})} defines an $L^{2}(\mathds{P})$-bounded martingale for $0 \leqslant s \leqslant T$. 

\smallskip

Summing up, we conclude that the process $(M(T-s))_{0 \leqslant s \leqslant T}$ is a $L^{2}(\mathds{P})$-bounded martingale of the backwards filtration, satisfying \hyperref[erdollafipim]{(\ref*{erdollafipim})}. This completes the proof of \hyperref[thetsix]{Theorem \ref*{thetsix}}.
\end{proof}


\subsection{The proof of \texorpdfstring{\hyperref[thetthretv]{Theorem \ref*{thetthretv}}}{Theorem 3.8}} 


The first step in the proof of \hyperref[thetthretv]{Theorem \ref*{thetthretv}} is to compute the stochastic differentials of the time-reversed perturbed likelihood ratio process
\begin{equation} \label{trplrp}
\ell^{\beta}\big(T-s,X(T-s)\big) = \frac{p^{\beta}\big(T-s,X(T-s)\big)}{q\big(X(T-s)\big)} \, , \qquad 0 \leqslant s \leqslant T - t_{0},
\end{equation}
and its logarithm. 

\smallskip

By analogy with \hyperref[ptra]{Proposition \ref*{ptra}}, the following result is well known to hold under suitable regularity conditions, such as \hyperref[sosaojkoia]{Assumptions \ref*{sosaojkoia}}. Recall that $(W^{\beta}(t))_{t \geqslant t_{0}}$ denotes the $\mathds{P}^{\beta}$-Brownian motion (in the forward direction of time) defined in \hyperref[wpsdeids]{(\ref*{wpsdeids})}.

\begin{proposition} Under the \textnormal{\hyperref[sosaojkoia]{Assumptions \ref*{sosaojkoia}}}, we let $t_{0} \geqslant 0$ and $T > t_{0}$. The process 
\begin{equation} \label{dotowpbtmtpbm}
\overline{W}^{\mathds{P}^{\beta}}(T-s) \vcentcolon = W^{\beta}(T-s) - W^{\beta}(T) - \int_{0}^{s} \nabla \log p^{\beta}\big(T-u,X(T-u)\big) \, \textnormal{d}u
\end{equation}
for $0 \leqslant s \leqslant T-t_{0}$, is a Brownian motion of the backwards filtration $(\mathcal{G}(T-s))_{0 \leqslant s \leqslant T-t_{0}}$ under the probability measure $\mathds{P}^{\beta}$. Furthermore, the semimartingale decomposition of the time-reversed canonical coordinate process $(X(T-s))_{0 \leqslant s \leqslant T-t_{0}}$ is given by
\begingroup
\addtolength{\jot}{0.7em}
\begin{align}
\textnormal{d} X(T-s) 
&= \Big( \nabla \log p^{\beta}\big(T-s,X(T-s)\big) + \big(\nabla \Psi + \beta\big)\big(X(T-s)\big) \Big) \, \textnormal{d}s + \textnormal{d}\overline{W}^{\mathds{P}^{\beta}}(T-s) \label{rtdfttrppitop} \\
&= \Big( \nabla \log \ell^{\beta} \big(T-s,X(T-s)\big) - \big(\nabla \Psi - \beta\big)\big(X(T-s)\big) \Big) \, \textnormal{d}s + \textnormal{d}\overline{W}^{\mathds{P}^{\beta}}(T-s), \label{rtdfttrpp}
\end{align}
\endgroup
for $0 \leqslant s \leqslant T - t_{0}$, with respect to the backwards filtration $(\mathcal{G}(T-s))_{0 \leqslant s \leqslant T - t_{0}}$. 
\end{proposition}

We note next, how the Brownian motions $\big(\overline{W}^{\mathds{P}^{\beta}}(T-s)\big)_{0 \leqslant s \leqslant T-t_{0}}$ and $\big(\overline{W}^{\mathds{P}}(T-s)\big)_{0 \leqslant s \leqslant T-t_{0}}$, in reverse-time, are related.

\begin{lemma} \label{cotttrbmpapb} Under the \textnormal{\hyperref[sosaojkoia]{Assumptions \ref*{sosaojkoia}}}, we let $t_{0} \geqslant 0$ and $T > t_{0}$. For $0 \leqslant s \leqslant T-t_{0}$, we have
\begingroup
\addtolength{\jot}{0.7em}
\begin{align}
\textnormal{d}\big( \overline{W}^{\mathds{P}} - \overline{W}^{\mathds{P}^{\beta}}\big)(T-s) &=  \Bigg( \beta\big(X(T-s)\big) 
 + \nabla \log \Bigg( \frac{p^{\beta}\big(T-s,X(T-s)\big)}{p\big(T-s,X(T-s)\big)} \Bigg) \Bigg) \, \textnormal{d}s \label{dbowptmtowpbtmtfp} \\
 &=  \Bigg( \beta\big(X(T-s)\big) 
 + \nabla \log \Bigg( \frac{\ell^{\beta}\big(T-s,X(T-s)\big)}{\ell\big(T-s,X(T-s)\big)} \Bigg) \Bigg) \, \textnormal{d}s. \label{dbowptmtowpbtmt}
\end{align}
\endgroup
\begin{proof} It suffices to compare the equation \hyperref[poortdfttrpp]{(\ref*{poortdfttrpp})} with \hyperref[rtdfttrppitop]{(\ref*{rtdfttrppitop})}.
\end{proof}
\end{lemma}

\begin{remark} Later we shall apply \hyperref[cotttrbmpapb]{Lemma \ref*{cotttrbmpapb}} to the situation when $s$ is close to $T-t_{0}$. In this case the logarithmic gradients in \hyperref[dbowptmtowpbtmtfp]{(\ref*{dbowptmtowpbtmtfp})} and \hyperref[dbowptmtowpbtmt]{(\ref*{dbowptmtowpbtmt})} will become small in view of $p^{\beta}(t_{0},\, \cdot \,) = p(t_{0},\, \cdot \,)$, so that these logarithmic gradients will disappear in the limit $s \uparrow T-t_{0}$; see also \hyperref[hctclwittmeitpocasot]{Lemma \ref*{hctclwittmeitpocasot}} below. By contrast, the term $\beta(X(T-s))$ will not disappear in the limit $s \uparrow T-t_{0}$. Rather, it will tend to the random variable $\beta(X(t_{0}))$, which plays an important role in distinguishing between \hyperref[thetsixcorfes]{(\ref*{thetsixcorfes})} and \hyperref[thetsixcorfessl]{(\ref*{thetsixcorfessl})} in \hyperref[thetthretvcor]{Proposition \ref*{thetthretvcor}}.
\end{remark} 

\smallskip

Next, by analogy with \hyperref[cotpottpf]{Subsection \ref*{cotpottpf}}, for $t > t_{0}$, we write the perturbed Fokker-Planck equation \hyperref[pfpeq]{(\ref*{pfpeq})} as 
\begin{equation} \label{pfpeef}
\partial_{t} p^{\beta}(t,x) = \tfrac{1}{2} \Delta p^{\beta}(t,x) + \big\langle \nabla p^{\beta}(t,x) \, , \nabla \Psi(x) + \beta(x) \big\rangle_{\mathds{R}^{n}} + p^{\beta}(t,x) \, \big( \Delta \Psi(x) + \operatorname{div} \beta(x) \big).
\end{equation}
Using the relation
\begin{equation} \label{rppbaeb}
p^{\beta}(t,x) = \ell^{\beta}(t,x) \, q(x) = \ell^{\beta}(t,x) \, \mathrm{e}^{ - 2 \Psi(x)}, \qquad t \geqslant t_{0},
\end{equation}
determined computation shows that the perturbed likelihood ratio function $\ell^{\beta}(t,x)$ satisfies 
\begingroup
\addtolength{\jot}{0.7em}
\begin{equation} \label{pfpeeffe}
\begin{aligned}
\partial_{t} \ell^{\beta}(t,x) = \tfrac{1}{2} \Delta \ell^{\beta}(t,x) &+ \big\langle \nabla \ell^{\beta}(t,x) \, , \, \beta(x) - \nabla \Psi(x) \big\rangle_{\mathds{R}^{n}} \\
&+ \ell^{\beta}(t,x) \, \Big( \operatorname{div} \beta(x) -   \big\langle \beta(x) \, , \, 2 \, \nabla \Psi(x) \big\rangle_{\mathds{R}^{n}} \Big), \qquad t > t_{0};
\end{aligned}
\end{equation}
\endgroup
this is the analogue of the backwards Kolmogorov equation \hyperref[fpdefflrf]{(\ref*{fpdefflrf})} in this ``perturbed'' context, and reduces to \hyperref[fpdefflrf]{(\ref*{fpdefflrf})} when $\beta \equiv 0$.

\smallskip

With these preparations, we obtain the following stochastic differentials for our objects of interest.

\begin{lemma} \label{trplrpdald} Under the \textnormal{\hyperref[sosaojkoia]{Assumptions \ref*{sosaojkoia}}}, we let $t_{0} \geqslant 0$ and $T > t_{0}$. The time-reversed perturbed likelihood ratio process \textnormal{\hyperref[trplrp]{(\ref*{trplrp})}} and its logarithm satisfy the stochastic differential equations
\begingroup
\addtolength{\jot}{0.7em}
\begin{equation} \label{sdefttrplrp}
\begin{aligned}
&\frac{\textnormal{d} \ell^{\beta}\big(T-s,X(T-s)\big)}{\ell^{\beta}\big(T-s,X(T-s)\big)} 
= \Big( \big\langle \beta \, , \, 2 \, \nabla \Psi \big\rangle_{\mathds{R}^{n}} - \operatorname{div} \beta \Big)\big(X(T-s)\big) \, \textnormal{d}s \\
& \qquad \quad + \frac{\big\vert \nabla \ell^{\beta}\big(T-s,X(T-s)\big) \big\vert^{2}}{\ell^{\beta}\big(T-s,X(T-s)\big)^{2}} \, \textnormal{d}s \, + \, \Bigg\langle \frac{\nabla \ell^{\beta}\big(T-s,X(T-s)\big)}{\ell^{\beta}\big(T-s,X(T-s)\big)} \, , \, \textnormal{d}\overline{W}^{\mathds{P}^{\beta}}(T-s)\Bigg\rangle_{\mathds{R}^{n}}
\end{aligned}
\end{equation}
\endgroup
and
\begingroup
\addtolength{\jot}{0.7em}
\begin{equation} \label{sdefttrplrpail}
\begin{aligned}
&\textnormal{d} \log \ell^{\beta}\big(T-s,X(T-s)\big)
= \Big( \big\langle \beta \, , \, 2 \, \nabla \Psi \big\rangle_{\mathds{R}^{n}}  - \operatorname{div} \beta \Big)\big(X(T-s)\big) \, \textnormal{d}s  \\
&  \qquad \quad + \frac{1}{2} \frac{\big\vert \nabla \ell^{\beta}\big(T-s,X(T-s)\big) \big\vert^{2}}{\ell^{\beta}\big(T-s,X(T-s)\big)^{2}} \, \textnormal{d}s \, + \, \Bigg\langle \frac{\nabla \ell^{\beta}\big(T-s,X(T-s)\big)}{\ell^{\beta}\big(T-s,X(T-s)\big)} \, , \, \textnormal{d}\overline{W}^{\mathds{P}^{\beta}}(T-s) \Bigg\rangle_{\mathds{R}^{n}},
\end{aligned}
\end{equation}
\endgroup
respectively, for $0 \leqslant s \leqslant T - t_{0}$, with respect to the backwards filtration $(\mathcal{G}(T-s))_{0 \leqslant s \leqslant T - t_{0}}$.
\begin{proof} The equations \hyperref[sdefttrplrp]{(\ref*{sdefttrplrp})}, \hyperref[sdefttrplrpail]{(\ref*{sdefttrplrpail})} follow from It\^{o}'s formula together with \hyperref[rtdfttrpp]{(\ref*{rtdfttrpp})}, \hyperref[pfpeeffe]{(\ref*{pfpeeffe})}.
\end{proof}
\end{lemma}

We have assembled now all the ingredients needed for the proof of \hyperref[thetthretv]{Theorem \ref*{thetthretv}}. 

\begin{proof}[\bfseries \upshape Proof of \texorpdfstring{\hyperref[thetthretv]{Theorem \ref*{thetthretv}}}] Formally, the stochastic differential in \hyperref[sdefttrplrpail]{(\ref*{sdefttrplrpail})} of the time-reversed perturbed likelihood ratio process \hyperref[trplrp]{(\ref*{trplrp})} amounts to the conclusions \hyperref[fispdaftrps]{(\ref*{fispdaftrps})} -- \hyperref[erdollafipimps]{(\ref*{erdollafipimps})} of \hyperref[thetthretv]{Theorem \ref*{thetthretv}}. But we still have to substantiate the claim, that the stochastic process $(M^{\beta}(T-s))_{0 \leqslant s \leqslant T-t_{0}}$ defined in \hyperref[dollafipimps]{(\ref*{dollafipimps})} with representation \hyperref[erdollafipimps]{(\ref*{erdollafipimps})}, indeed yields a $\mathds{P}^{\beta}$-martingale of the backwards filtration $(\mathcal{G}(T-s))_{0 \leqslant s \leqslant T - t_{0}}$, which is bounded in $L^{2}(\mathds{P}^{\beta})$.

\medskip

\noindent \fbox{\textsf{Step 1.}} By analogy with \textsf{Step 2} in the proof of \hyperref[thetsix]{Theorem \ref*{thetsix}}, we show first, that the stochastic process $(M^{\beta}(T-s))_{0 \leqslant s \leqslant T-t_{0}}$ of \hyperref[dollafipimps]{(\ref*{dollafipimps})}, with integral representation \hyperref[erdollafipimps]{(\ref*{erdollafipimps})}, is a continuous local $\mathds{P}^{\beta}$-martingale of the backwards filtration $(\mathcal{G}(T-s))_{0 \leqslant s \leqslant T - t_{0}}$.

\smallskip

By condition \hyperref[naltsaosaojkos]{\ref*{naltsaosaojkos}} of \hyperref[sosaojkoia]{Assumptions \ref*{sosaojkoia}} the function $(t_{0},\infty) \ni t \mapsto \nabla \log \ell^{\beta}(t,x)$ is continuous for every $x \in \mathds{R}^{n}$. Together with the continuity of the paths of the canonical coordinate process $(X(t))_{t \geqslant 0}$, this implies that 
\begin{equation} \label{twtcopvbtitttfitisifoas}
\int_{0}^{T-t_{0}-\varepsilon} \frac{\big\vert \nabla \ell^{\beta}\big(T-u,X(T-u)\big) \big\vert^{2}}{\ell^{\beta}\big(T-u,X(T-u)\big)^{2}} \, \textnormal{d}u < \infty, \qquad \mathds{P}^{\beta}\textnormal{-almost surely},
\end{equation}
for every $0 < \varepsilon \leqslant T-t_{0}$. On account of \hyperref[twtcopvbtitttfitisifoas]{(\ref*{twtcopvbtitttfitisifoas})}, the sequence of stopping times (with respect to the backwards filtration)
\begin{equation} \label{twtcdostppvbfitisifoas}
\tau_{n}^{\beta} \vcentcolon = \inf \Bigg\{ t \geqslant 0 \colon \, \int_{0}^{t} \frac{\big\vert \nabla \ell^{\beta}\big(T-u,X(T-u)\big) \big\vert^{2}}{\ell^{\beta}\big(T-u,X(T-u)\big)^{2}} \, \textnormal{d}u \, \geqslant n \,  \Bigg\} \wedge (T-t_{0}), \qquad n \in \mathds{N}_{0}
\end{equation}
is non-decreasing and converges $\mathds{P}^{\beta}$-almost surely to $T-t_{0}$. Hence, according to \hyperref[sdefttrplrpail]{(\ref*{sdefttrplrpail})} and the definition in \hyperref[dollafipimps]{(\ref*{dollafipimps})}, the stopped process $\big((M^{\beta})^{\tau_{n}^{\beta}}(T-s)\big)_{0 \leqslant s \leqslant T-t_{0}}$ is a uniformly integrable $\mathds{P}^{\beta}$-martingale of the backwards filtration, for every $n \in \mathds{N}_{0}$.

\medskip

\noindent \fbox{\textsf{Step 2.}} We show that, in fact, the process $(M^{\beta}(T-s))_{0 \leqslant s \leqslant T-t_{0}}$ is a true $\mathds{P}^{\beta}$-martingale.

\smallskip 

Taking expectations with respect to the probability measure $\mathds{P}^{\beta}$ in \hyperref[sdefttrplrpail]{(\ref*{sdefttrplrpail})} at time $s = \tau_{n}^{\beta}$, and using that by \textsf{Step 1} the stopped process $\big((M^{\beta})^{\tau_{n}^{\beta}}(T-s)\big)_{0 \leqslant s \leqslant T-t_{0}}$ is a true $\mathds{P}^{\beta}$-martingale with respect to the backwards filtration $(\mathcal{G}(T-s))_{0 \leqslant s \leqslant T - t_{0}}$, we obtain
\begingroup
\addtolength{\jot}{0.7em}
\begin{align}
H\big( P^{\beta}(T-\tau_{n}^{\beta}) \, \vert \, \mathrm{Q} \big)  - H\big( P^{\beta}(T) \, \vert \, \mathrm{Q} \big) \ = \
&\mathds{E}_{\mathds{P}^{\beta}}\Bigg[\int_{0}^{\tau_{n}^{\beta}}\frac{1}{2}\frac{\big\vert \nabla \ell^{\beta}\big(T-u,X(T-u)\big) \big\vert^{2}}{\ell^{\beta}\big(T-u,X(T-u)\big)^{2}} \, \textnormal{d}u \Bigg] \label{tfea} \\
 + \ &\mathds{E}_{\mathds{P}^{\beta}}\Bigg[\int_{0}^{\tau_{n}^{\beta}} \Big(  \big\langle \beta \, , \, 2 \, \nabla \Psi \big\rangle_{\mathds{R}^{n}} - \operatorname{div} \beta \Big) \big(X(T-u)\big) \, \textnormal{d}u \Bigg], \label{tfeab}
\end{align}
\endgroup
for every $n \in \mathds{N}_{0}$. According to condition \hyperref[naltsaosaojkos]{\ref*{naltsaosaojkos}} of \hyperref[sosaojkoia]{Assumptions \ref*{sosaojkoia}}, the perturbation $\beta \colon \mathds{R}^{n} \rightarrow \mathds{R}^{n}$ is the gradient of a potential $B \colon \mathds{R}^{n} \rightarrow \mathds{R}$ of class $\mathcal{C}_{c}^{\infty}(\mathds{R}^{n};\mathds{R})$. Therefore, for every $n \in \mathds{N}_{0}$, the absolute value of the expectation in \hyperref[tfeab]{(\ref*{tfeab})} can be bounded by
\begin{equation} \label{tfeabb}
C_{1} \vcentcolon = \mathds{E}_{\mathds{P}^{\beta}}\Bigg[\int_{0}^{T-t_{0}} \Big\vert  \big\langle \beta \, , \, 2 \, \nabla \Psi \big\rangle_{\mathds{R}^{n}} - \operatorname{div} \beta \Big\vert \big(X(T-u)\big) \, \textnormal{d}u \Bigg] < \infty.
\end{equation}
We denote by $\mathrm{Q}^{\beta}$ the $\sigma$-finite measure on the Borel sets of $\mathds{R}^{n}$, whose density with respect to Lebesgue measure is
\begin{equation}
q^{\beta}(x) \vcentcolon = \mathrm{e}^{-2(\Psi + B)(x)}, \qquad x \in \mathds{R}^{n}.
\end{equation}
Now, by analogy with \hyperref[dotrefuc]{(\ref*{dotrefuc})} and the discussion of the $\sigma$-finite case in \textsf{Step 3} from the proof of \hyperref[thetsix]{Theorem \ref*{thetsix}}, we have 
\begin{equation} \label{dotrebefuc}
H\big( P^{\beta}(t) \, \vert \, \mathrm{Q}^{\beta} \big) \leqslant H\big( P^{\beta}(t_{0}) \, \vert \, \mathrm{Q}^{\beta} \big), \qquad t_{0} \leqslant t \leqslant T.
\end{equation}
Note that for this decrease of the relative entropy function $t \mapsto H( P^{\beta}(t) \, \vert \, \mathrm{Q}^{\beta})$ to be valid, it is necessary to ``take into account the perturbation for the flow as well as for the reference measure''. However, in view of the left-hand side of \hyperref[tfea]{(\ref*{tfea})}, we are interested in the behavior of the function $t \mapsto H( P^{\beta}(t) \, \vert \, \mathrm{Q})$. In order to compare these two, we define the constant
\begin{equation}  \label{dsaotrebefuc}
C_{2} \vcentcolon = 2 \cdot \max_{x \in \mathds{R}^{n}} \vert B(x) \vert < \infty.
\end{equation}
It is then straightforward to derive the estimate
\begin{equation} \label{dotrebefucb}
H\big( P^{\beta}(t) \, \vert \, \mathrm{Q} \big) - C_{2} \, \leqslant \, H\big( P^{\beta}(t) \, \vert \, \mathrm{Q}^{\beta} \big) \, \leqslant \, H\big( P^{\beta}(t) \, \vert \, \mathrm{Q} \big) + C_{2}, \qquad t_{0} \leqslant t \leqslant T.
\end{equation}
Combining \hyperref[dotrebefuc]{(\ref*{dotrebefuc})}, \hyperref[dotrebefucb]{(\ref*{dotrebefucb})} and using that $P^{\beta}(t_{0}) = P(t_{0})$, we get
\begin{equation} \label{dotrebefucbne}
H\big( P^{\beta}(t) \, \vert \, \mathrm{Q} \big) \, \leqslant \, H\big( P(t_{0}) \, \vert \, \mathrm{Q} \big) + 2C_{2}, \qquad t_{0} \leqslant t \leqslant T.
\end{equation}
Consequently, the left-hand side of \hyperref[tfea]{(\ref*{tfea})} can be dominated by
\begin{equation} \label{dotrebefucbnnbe}
C_{3} \vcentcolon = H\big( P(t_{0}) \, \vert \, \mathrm{Q} \big) + 2C_{2} - H\big( P^{\beta}(T) \, \vert \, \mathrm{Q} \big) < \infty.
\end{equation}
This constant cannot be equal to $+\infty$ because $H( P(t_{0}) \, \vert \, \mathrm{Q})$ is finite by the same argument as in the proof of \hyperref[thetsix]{Theorem \ref*{thetsix}}, and the relative entropy $H( P^{\beta}(T) \, \vert \, \mathrm{Q})$ cannot take the value $-\infty$; see \hyperref[fosmoldspv]{Lemma \ref*{fosmoldspv}} and \hyperref[wdadotrewrttmq]{Appendix \ref*{wdadotrewrttmq}}. Altogether, we obtain from \hyperref[tfea]{(\ref*{tfea})}, \hyperref[tfeab]{(\ref*{tfeab})} and the above considerations that 
\begin{equation} \label{dotrebefucbnnnbe}
\mathds{E}_{\mathds{P}^{\beta}}\Bigg[\int_{0}^{\tau_{n}^{\beta}}\frac{1}{2}\frac{\big\vert \nabla \ell^{\beta}\big(T-u,X(T-u)\big) \big\vert^{2}}{\ell^{\beta}\big(T-u,X(T-u)\big)^{2}} \, \textnormal{d}u \Bigg] \leqslant C_{1} + C_{3},
\end{equation}
for every $n \in \mathds{N}_{0}$. Since the non-decreasing sequence of stopping times $(\tau_{n}^{\beta})_{n \geqslant 0}$ from \hyperref[twtcdostppvbfitisifoas]{(\ref*{twtcdostppvbfitisifoas})} converges $\mathds{P}^{\beta}$-almost surely to $T-t_{0}$ as $n \rightarrow \infty$, we deduce from \hyperref[dotrebefucbnnnbe]{(\ref*{dotrebefucbnnnbe})} and the monotone convergence theorem that
\begin{equation} \label{fecwfoeltpsn}
\mathds{E}_{\mathds{P}^{\beta}}\big[F^{\beta}(t_{0})\big] = \mathds{E}_{\mathds{P}^{\beta}}\Bigg[\int_{0}^{T - t_{0}} \frac{1}{2} \frac{\big\vert \nabla \ell^{\beta}\big(T-u,X(T-u)\big) \big\vert^{2}}{\ell^{\beta}\big(T-u,X(T-u)\big)^{2}} \, \textnormal{d}u\Bigg] \leqslant C_{1} + C_{3} < \infty.
\end{equation}
From \hyperref[fecwfoeltpsn]{(\ref*{fecwfoeltpsn})} we finally obtain that the stochastic integral in \hyperref[sdefttrplrpail]{(\ref*{sdefttrplrpail})} defines an $L^{2}(\mathds{P}^{\beta})$-bounded martingale for $0 \leqslant s \leqslant T-t_{0}$. 

\smallskip

Summing up, we conclude that the process $(M^{\beta}(T-s))_{0 \leqslant s \leqslant T-t_{0}}$ is a $L^{2}(\mathds{P}^{\beta})$-bounded martingale of the backwards filtration $(\mathcal{G}(T-s))_{0 \leqslant s \leqslant T - t_{0}}$, admitting the representation \hyperref[erdollafipimps]{(\ref*{erdollafipimps})}. This completes the proof of \hyperref[thetthretv]{Theorem \ref*{thetthretv}}.
\end{proof}


\subsection{Some useful lemmas} \label{subsomusefullem}


In this subsection we collect some useful results needed in order to justify the claims \hyperref[ilpdepvhfv]{(\ref*{ilpdepvhfv})}, \hyperref[ilpdepvhfvsv]{(\ref*{ilpdepvhfvsv})} made in the course of the proof of \hyperref[thetsixcoraopv]{Corollary \ref*{thetsixcoraopv}}, and to complete the proof of \hyperref[thetthretvcor]{Proposition \ref*{thetthretvcor}} in \hyperref[nspofthetthretvcor]{Subsection \ref*{nspofthetthretvcor}}.

\medskip

First, let us recall the probability density function $(t,x) \mapsto p(t,x)$ from \hyperref[tpdfptxcbritf]{(\ref*{tpdfptxcbritf})}, its perturbed version $(t,x) \mapsto p^{\beta}(t,x)$ from \hyperref[rppbaeb]{(\ref*{rppbaeb})}, and the respective likelihood ratios $\ell(t,x)$, $\ell^{\beta}(t,x)$ from \hyperref[rndlr]{(\ref*{rndlr})}, \hyperref[dotplrfb]{(\ref*{dotplrfb})}, respectively. We introduce also the ``perturbed-to-unperturbed'' ratio
\begin{equation} \label{witpetounper}
Y^{\beta}(t,x) \vcentcolon = \frac{\ell^{\beta}(t,x)}{\ell(t,x)}  = \frac{p^{\beta}(t,x)}{p(t,x)}, \qquad (t,x) \in [t_{0},\infty) \times \mathds{R}^{n}.
\end{equation}

We recall the backwards Kolmogorov-type equations \hyperref[fpdefflrf]{(\ref*{fpdefflrf})}, \hyperref[pfpeeffe]{(\ref*{pfpeeffe})}. These lead to the partial differential equation 
\begingroup
\addtolength{\jot}{0.7em}
\begin{equation} \label{pdefoybptupr}
\begin{aligned}
\partial_{t} Y^{\beta}(t,x) = \tfrac{1}{2} \Delta Y^{\beta}(t,x) 
& + \big\langle \nabla Y^{\beta}(t,x) \, , \, \beta(x) + \nabla \log p(t,x) + \nabla \Psi(x) \big\rangle_{\mathds{R}^{n}} \\
& + Y^{\beta}(t,x) \, \Big( \operatorname{div} \beta(x) + \big\langle \beta(x) \, , \, \nabla \log p(t,x) \big\rangle_{\mathds{R}^{n}} \Big), \qquad t > t_{0},
\end{aligned}
\end{equation}
\endgroup
with $Y^{\beta}(t_{0}, \, \cdot \,) = 1$, for the ratio in \hyperref[witpetounper]{(\ref*{witpetounper})}. In conjunction with \hyperref[poortdfttrpp]{(\ref*{poortdfttrpp})}, this equation leads to the following backward dynamics.

\begin{lemma} \label{wsublnsibvn} Under the \textnormal{\hyperref[sosaojkoia]{Assumptions \ref*{sosaojkoia}}}, we let $t_{0} \geqslant 0$ and $T > t_{0}$. The time-reversed ratio process $\big(Y^{\beta}(T-s,X(T-s))\big)_{0 \leqslant s \leqslant T-t_{0}}$ and its logarithm satisfy the stochastic differential equations
\begingroup
\addtolength{\jot}{0.7em}
\begin{equation} \label{logybdbyb}
\begin{aligned}
&\frac{\textnormal{d}Y^{\beta}\big(T-s,X(T-s)\big)}{Y^{\beta}\big(T-s,X(T-s)\big)} = \Bigg\langle \frac{\nabla Y^{\beta}\big(T-s,X(T-s)\big)}{Y^{\beta}\big(T-s,X(T-s)\big)} \, , \, \textnormal{d}\overline{W}^{\mathds{P}}(T-s) - \beta\big(X(T-s)\big) \, \textnormal{d}s \Bigg\rangle_{\mathds{R}^{n}} \\
& \qquad \qquad \qquad - \bigg( \operatorname{div} \beta \big(X(T-s)\big) + \Big\langle \beta\big(X(T-s)\big) \, , \nabla \log p\big(T-s,X(T-s)\big) \Big\rangle_{\mathds{R}^{n}}  \bigg) \, \textnormal{d}s
\end{aligned}
\end{equation}
\endgroup
and
\begingroup
\addtolength{\jot}{0.7em}
\begin{equation} \label{logybetapro}
\begin{aligned}
&\textnormal{d} \log Y^{\beta}\big(T-s,X(T-s)\big) = \Bigg\langle \frac{\nabla Y^{\beta}\big(T-s,X(T-s)\big)}{Y^{\beta}\big(T-s,X(T-s)\big)} \, , \, \textnormal{d}\overline{W}^{\mathds{P}}(T-s) - \beta\big(X(T-s)\big) \, \textnormal{d}s \Bigg\rangle_{\mathds{R}^{n}} \\
& \qquad \qquad \qquad - \bigg( \operatorname{div} \beta \big(X(T-s)\big) + \Big\langle \beta\big(X(T-s)\big) \, , \nabla \log p\big(T-s,X(T-s)\big) \Big\rangle_{\mathds{R}^{n}}  \bigg) \, \textnormal{d}s \\
& \qquad \qquad \qquad \qquad \qquad \qquad  \qquad  \qquad \qquad \qquad - \frac{1}{2} \frac{\big\vert \nabla Y^{\beta}\big(T-s,X(T-s)\big) \big\vert^{2}}{Y^{\beta}\big(T-s,X(T-s)\big)^{2}} \, \textnormal{d}s,
\end{aligned}
\end{equation}
\endgroup
respectively, for $0 \leqslant s \leqslant T - t_{0}$, with respect to the backwards filtration $(\mathcal{G}(T-s))_{0 \leqslant s \leqslant T - t_{0}}$.
\end{lemma}

\begin{remark} We can obtain the dynamics of \hyperref[logybetapro]{(\ref*{logybetapro})}, thus also of \hyperref[logybdbyb]{(\ref*{logybdbyb})}, directly from \hyperref[sdefttrplrpail]{(\ref*{sdefttrplrpail})}, \hyperref[sdotpttrrep]{(\ref*{sdotpttrrep})}, just by subtracting and using \hyperref[cotttrbmpapb]{Lemma \ref*{cotttrbmpapb}}.
\end{remark}

We also need a preliminary control on $Y^{\beta}(\, \cdot \, , \, \cdot \,)$, which is the subject of the following \hyperref[hctclwittmeitpocasotpv]{Lemma \ref*{hctclwittmeitpocasotpv}}. This will be refined in \hyperref[hctclwittmeitpocasot]{Lemma \ref*{hctclwittmeitpocasot}} below.

\begin{lemma} \label{hctclwittmeitpocasotpv} Under the \textnormal{\hyperref[sosaojkoia]{Assumptions \ref*{sosaojkoia}}}, we let $t_{0} \geqslant 0$ and $T > t_{0}$. There is a real constant $C > 1$ such that 
\begin{equation} \label{flaarftpottpsfe}
\frac{1}{C} \leqslant Y^{\beta}(t,x) \leqslant C \, , \qquad (t,x) \in [t_{0},T] \times \mathds{R}^{n}.
\end{equation}
\begin{proof} In the forward direction of time, the canonical coordinate process $(X(t))_{t_{0} \leqslant t \leqslant T}$ on the path space $\Omega = \mathcal{C}([t_{0},T];\mathds{R}^{n})$ satisfies the stochastic equations \hyperref[sdeids]{(\ref*{sdeids})} and \hyperref[wpsdeids]{(\ref*{wpsdeids})} with initial distribution $P(t_{0})$ under the probability measures $\mathds{P}$ and $\mathds{P}^{\beta}$, respectively. Hence, the $\mathds{P}$-Brownian motion $(W(t))_{t_{0} \leqslant t \leqslant T}$ from \hyperref[sdeids]{(\ref*{sdeids})} can be represented as
\begin{equation}
W(t) - W(t_{0}) = W^{\beta}(t) - W^{\beta}(t_{0}) - \int_{t_{0}}^{t} \beta\big(X(u)\big) \, \textnormal{d}u, \qquad t_{0} \leqslant t \leqslant T,
\end{equation}
where $(W^{\beta}(t))_{t_{0} \leqslant t \leqslant T}$ is the $\mathds{P}^{\beta}$-Brownian motion appearing in \hyperref[wpsdeids]{(\ref*{wpsdeids})}. By the Girsanov theorem, this amounts, for $t_{0} \leqslant t \leqslant T$, to the likelihood ratio computation
\begin{equation} \label{tgtae}
Z(t) \vcentcolon = \frac{\textnormal{d}\mathds{P}^{\beta}}{\textnormal{d}\mathds{P}} \bigg \vert_{\mathcal{F}(t)} 
= \exp \Bigg( - \int_{t_{0}}^{t} \Big\langle \beta\big(X(u)\big) \, , \, \textnormal{d}W(u) \Big\rangle_{\mathds{R}^{n}} - \tfrac{1}{2} \int_{t_{0}}^{t} \big\vert \beta\big(X(u)\big) \big\vert^{2} \, \textnormal{d}u \Bigg).
\end{equation} 

\smallskip

Now, for each $(t,x) \in [t_{0},T] \times \mathds{R}^{n}$, the ratio $Y^{\beta}(t,x) = p^{\beta}(t,x) / p(t,x)$ equals the conditional expectation of the random variable \hyperref[tgtae]{(\ref*{tgtae})} with respect to the probability measure $\mathds{P}$, where we condition on $X(t) = x$; to wit,
\begin{equation}
Y^{\beta}(t,x) = \mathds{E}_{\mathds{P}}\big[ Z(t) \, \vert \, X(t) = x\big] \, , \qquad (t,x) \in [t_{0},T] \times \mathds{R}^{n}.
\end{equation}
Therefore, in order to obtain the estimate \hyperref[flaarftpottpsfe]{(\ref*{flaarftpottpsfe})}, it suffices to show that the log-density process $( \log Z(t))_{t_{0} \leqslant t \leqslant T}$ is uniformly bounded. Since the perturbation $\beta$ is smooth and has compact support, the Lebesgue integral inside the exponential of \hyperref[tgtae]{(\ref*{tgtae})} is uniformly bounded, as required. 

\smallskip

In order to handle the stochastic integral with respect to the $\mathds{P}$-Brownian motion $(W(u))_{t_{0} \leqslant u \leqslant t}$ inside the exponential \hyperref[tgtae]{(\ref*{tgtae})}, we invoke the assumption that the vector field $\beta$ equals the gradient of a potential $B \colon \mathds{R}^{n} \rightarrow \mathds{R}$, which is of class $\mathcal{C}^{\infty}(\mathds{R}^{n};\mathds{R})$ and has compact support. According to It\^{o}'s formula and \hyperref[sdeids]{(\ref*{sdeids})}, we can express the stochastic integral appearing in \hyperref[tgtae]{(\ref*{tgtae})} as
\begin{equation} \label{urotsiarl}
\int_{t_{0}}^{t} \Big\langle \beta\big(X(u)\big) \, , \, \textnormal{d}W(u) \Big\rangle_{\mathds{R}^{n}} = B\big(X(t)) - B\big(X(t_{0})\big) + \int_{t_{0}}^{t} \Big(\big\langle \beta \, , \, \nabla \Psi \big\rangle_{\mathds{R}^{n}} - \tfrac{1}{2} \operatorname{div} \beta \Big)\big(X(u)\big) \, \textnormal{d}u
\end{equation}
for $t_{0} \leqslant t \leqslant T$. At this stage it becomes obvious that the expression of \hyperref[urotsiarl]{(\ref*{urotsiarl})} is uniformly bounded. This completes the proof of \hyperref[hctclwittmeitpocasotpv]{Lemma \ref*{hctclwittmeitpocasotpv}}.
\end{proof}
\end{lemma}

The following \hyperref[hctclwittmeitpocasot]{Lemma \ref*{hctclwittmeitpocasot}} provides the crucial estimates \hyperref[ilpdepvhfv]{(\ref*{ilpdepvhfv})} and \hyperref[ilpdepvhfvsv]{(\ref*{ilpdepvhfvsv})}, needed in the proof of \hyperref[thetsixcoraopv]{Corollary \ref*{thetsixcoraopv}} from \hyperref[thetthretv]{Theorem \ref*{thetthretv}}, and of \hyperref[thetthretvcor]{Proposition \ref*{thetthretvcor}}.

\begin{lemma} \label{hctclwittmeitpocasot} Under the \textnormal{\hyperref[sosaojkoia]{Assumptions \ref*{sosaojkoia}}}, we let $t_{0} \geqslant 0$ and $T > t_{0}$. There is a constant $C > 0$ such that 
\begin{equation} \label{ilpdepvhfvwp}
\big\vert Y^{\beta}(T-s,x) - 1 \big\vert \leqslant C \, (T-t_{0}-s),
\end{equation}
as well as
\begin{equation} \label{ilpdepvhfvsvwp}
\mathds{E}_{\mathds{P}}\Bigg[\int_{s}^{T-t_{0}} \Big\vert \nabla \log Y^{\beta}\big(T-u,X(T-u)\big) \Big\vert^{2} \, \textnormal{d}u  \ \bigg\vert \ X(T-s) = x \Bigg] 
\leqslant C \, (T - t_{0} - s)^{2}, 
\end{equation}
hold for all $0 \leqslant s \leqslant T-t_{0}$ and $x \in \mathds{R}^{n}$. Furthermore, for every $t_{0} > 0$ and $x \in \mathds{R}^{n}$ we have the pointwise limiting assertion
\begin{equation} \label{taflawhplftw}
\lim_{s \uparrow T-t_{0}} \frac{\log Y^{\beta}(T-s,x)}{T-t_{0}-s}
= \operatorname{div} \beta(x) + \Big\langle \beta(x) \, , \, \nabla  \log p(t_{0},x) \Big\rangle_{\mathds{R}^{n}},
\end{equation}
where the fraction on the left-hand side of \textnormal{\hyperref[taflawhplftw]{(\ref*{taflawhplftw})}} is uniformly bounded on $[0,T-t_{0}] \times \mathds{R}^{n}$.
\begin{remark} The pointwise limiting assertion \hyperref[taflawhplftw]{(\ref*{taflawhplftw})} is the deterministic analogue of the trajectorial relation \hyperref[titmlwhtcfs]{(\ref*{titmlwhtcfs})} from \hyperref[thetthretvcor]{Proposition \ref*{thetthretvcor}}. In \hyperref[nspofthetthretvcor]{Subsection \ref*{nspofthetthretvcor}} below we will prove that the limiting assertion \hyperref[titmlwhtcfs]{(\ref*{titmlwhtcfs})} holds in $L^{1}$ under both $\mathds{P}$ and $\mathds{P}^{\beta}$, and is valid for all $t_{0} > 0$.
\end{remark}
\begin{proof} As $\log Y^{\beta} = \log \ell^{\beta} - \log \ell$, we obtain from \hyperref[fecwfoel]{(\ref*{fecwfoel})}, \hyperref[fecwfoeltpsn]{(\ref*{fecwfoeltpsn})} and \hyperref[flaarftpottpsfe]{(\ref*{flaarftpottpsfe})} that the martingale part of the process in \hyperref[logybetapro]{(\ref*{logybetapro})} is bounded in $L^{2}(\mathds{P})$, i.e.,
\begin{equation} \label{martpoflogybibilt}
\mathds{E}_{\mathds{P}}\Bigg[\int_{0}^{T - t_{0}} \frac{\big\vert \nabla Y^{\beta}\big(T-u,X(T-u)\big) \big\vert^{2}}{Y^{\beta}\big(T-u,X(T-u)\big)^{2}} \, \textnormal{d}u\Bigg] < \infty.
\end{equation}
Once again using \hyperref[flaarftpottpsfe]{(\ref*{flaarftpottpsfe})}, we compare $\nabla Y^{\beta} / Y^{\beta}$ with $\nabla Y^{\beta}$ to see that \hyperref[martpoflogybibilt]{(\ref*{martpoflogybibilt})} also implies
\begin{equation} \label{martpoflogybibiltnyb}
\mathds{E}_{\mathds{P}}\Bigg[\int_{0}^{T - t_{0}} \Big\vert \nabla Y^{\beta}\big(T-u,X(T-u)\big) \Big\vert^{2} \, \textnormal{d}u\Bigg] < \infty.
\end{equation}
According to \hyperref[logybdbyb]{(\ref*{logybdbyb})}, the time-reversed ratio process $\big(Y^{\beta}(T-s,X(T-s))\big)_{0 \leqslant s \leqslant T-t_{0}}$ satisfies the stochastic differential equation
\begingroup
\addtolength{\jot}{0.7em}
\begin{equation} \label{tpdbdoybe}
\begin{aligned}
&\textnormal{d}Y^{\beta}\big(T-s,X(T-s)\big) = \Big\langle \nabla Y^{\beta}\big(T-s,X(T-s)\big) \, , \, \textnormal{d}\overline{W}^{\mathds{P}}(T-s) - \beta\big(X(T-s)\big) \, \textnormal{d}s \Big\rangle_{\mathds{R}^{n}} \\
& \ - Y^{\beta}\big(T-s,X(T-s)\big)  \bigg( \operatorname{div} \beta \big(X(t-s)\big) + \Big\langle \beta\big(X(T-s)\big) \, , \nabla \log p\big(T-s,X(T-s)\big) \Big\rangle_{\mathds{R}^{n}}  \bigg) \, \textnormal{d}s
\end{aligned}
\end{equation}
\endgroup
for $0 \leqslant s \leqslant T - t_{0}$, with respect to the backwards filtration $(\mathcal{G}(T-s))_{0 \leqslant s \leqslant T - t_{0}}$. 

\smallskip

In view of \hyperref[martpoflogybibiltnyb]{(\ref*{martpoflogybibiltnyb})}, the martingale part in \hyperref[tpdbdoybe]{(\ref*{tpdbdoybe})} is bounded in $L^{2}(\mathds{P})$. As regards the drift terms of this equation, we observe that it vanishes when $X(T-s)$ takes values outside the compact support of the smooth vector field $\beta$. Consequently, the drift terms are bounded, i.e., the constant
\begin{equation} 
C_{1} \vcentcolon = \sup_{\substack{t_{0} \leqslant t \leqslant T \\ y \in \mathds{R}^{n}}} \Bigg\vert - Y^{\beta}(t,y)  \Bigg( \operatorname{div} \beta(y) + \Bigg\langle \beta(y) \, , \nabla \log p(t,y) + \frac{\nabla Y^{\beta}(t,y)}{Y^{\beta}(t,y)} \Bigg\rangle_{\mathds{R}^{n}} \, \Bigg)  \Bigg\vert
\end{equation}
is finite, and the processes
\begin{equation}
Y^{\beta}\big(T-s,X(T-s)\big) + C_{1} \, s \qquad \textnormal{ and } \qquad Y^{\beta}\big(T-s,X(T-s)\big) - C_{1} \, s 
\end{equation}
for $0 \leqslant s \leqslant T-t_{0}$, are a sub- and a supermartingale, respectively. We conclude that
\begin{equation} \label{tjpefftvotp}
\Big\vert \, Y^{\beta}(T-s,x) - \mathds{E}_{\mathds{P}}\Big[ Y^{\beta}\big(t_{0},X(t_{0})\big) \ \big\vert \ X(T-s) = x \Big] \, \Big\vert \leqslant  C_{1} \, (T-t_{0}-s)
\end{equation}
holds for all $0 \leqslant s \leqslant T-t_{0}$ and $x \in \mathds{R}^{n}$. Since $Y^{\beta}(t_{0}, \, \cdot \,) = 1$, this establishes the first estimate 
\begin{equation} \label{ilpdepvhfvwphiwtrt}
\big\vert Y^{\beta}(T-s,x) - 1 \big\vert \leqslant C_{1} \, (T-t_{0}-s).
\end{equation}

\smallskip

Now we turn our attention to the second estimate \hyperref[ilpdepvhfvsvwp]{(\ref*{ilpdepvhfvsvwp})}. We fix $0 \leqslant s \leqslant T-t_{0}$ and $x \in \mathds{R}^{n}$. By means of the stochastic differentials in \hyperref[logybetapro]{(\ref*{logybetapro})} and \hyperref[tpdbdoybe]{(\ref*{tpdbdoybe})}, we find that the expression
\begin{equation} \label{sodacebtbbsewiet}
\tfrac{1}{2} \, \mathds{E}_{\mathds{P}}\Bigg[\int_{s}^{T-t_{0}} \Big\vert \nabla \log Y^{\beta}\big(T-u,X(T-u)\big) \Big\vert^{2} \, \textnormal{d}u  \ \bigg\vert \ X(T-s) = x \Bigg] 
\end{equation}
is equal to
\begin{equation} \label{sodacebtbbs}
\log Y^{\beta}(T-s,x) - Y^{\beta}(T-s,x) + 1 + \mathds{E}_{\mathds{P}}\Bigg[\int_{s}^{T-t_{0}} G\big(T-u,X(T-u)\big) \, \textnormal{d}u  \ \bigg\vert \ X(T-s) = x \Bigg],
\end{equation}
where we have set
\begin{equation}
G(t,y) \vcentcolon = \big( Y^{\beta}(t,y)-1\big) \Bigg( \operatorname{div} \beta(y) + \Bigg\langle \beta(y) \, , \nabla \log p(t,y) + \frac{\nabla Y^{\beta}(t,y)}{Y^{\beta}(t,y)} \Bigg\rangle_{\mathds{R}^{n}} \, \Bigg)
\end{equation}
for $t_{0} \leqslant t \leqslant T$ and $y \in \mathds{R}^{n}$. Introducing the finite constant
\begin{equation} 
C_{2} \vcentcolon = \sup_{\substack{t_{0} \leqslant t \leqslant T \\ y \in \mathds{R}^{n}}} \Bigg\vert \operatorname{div} \beta(y) + \Bigg\langle \beta(y) \, , \nabla \log p(t,y) + \frac{\nabla Y^{\beta}(t,y)}{Y^{\beta}(t,y)} \Bigg\rangle_{\mathds{R}^{n}} \, \Bigg\vert
\end{equation}
and using the just proved estimate \hyperref[ilpdepvhfvwphiwtrt]{(\ref*{ilpdepvhfvwphiwtrt})}, we see that the absolute value of the conditional expectation appearing in \hyperref[sodacebtbbs]{(\ref*{sodacebtbbs})} can be bounded by $C_{1} \, C_{2} \, (T-t_{0}-s)^{2}$. In order to handle the remaining terms of \hyperref[sodacebtbbs]{(\ref*{sodacebtbbs})}, we apply the elementary inequality $\log p \leqslant p-1$, which is valid for all $p > 0$, and obtain
\begin{equation} \label{teiuftfa}
\log Y^{\beta}(T-s,x) - Y^{\beta}(T-s,x) + 1 \leqslant 0.
\end{equation}
This implies that the expression of \hyperref[sodacebtbbsewiet]{(\ref*{sodacebtbbsewiet})} is bounded by $C_{1} \, C_{2} \, (T-t_{0}-s)^{2}$, which establishes the second estimate \hyperref[ilpdepvhfvsvwp]{(\ref*{ilpdepvhfvsvwp})}. We also note that the elementary inequality \hyperref[teiuftfa]{(\ref*{teiuftfa})} in conjunction with the estimate \hyperref[ilpdepvhfvwphiwtrt]{(\ref*{ilpdepvhfvwphiwtrt})} shows that
\begin{equation}
\log Y^{\beta}(T-s,x) \leqslant C_{1} \, (T-t_{0}-s)   
\end{equation}
for all $0 \leqslant s \leqslant T-t_{0}$ and $x \in \mathds{R}^{n}$; this implies that the fraction on the left-hand side of \hyperref[taflawhplftw]{(\ref*{taflawhplftw})} is uniformly bounded on $[0,T-t_{0}] \times \mathds{R}^{n}$. 

\smallskip

Regarding the limiting assertion \hyperref[taflawhplftw]{(\ref*{taflawhplftw})}, we fix $t_{0} > 0$, $x \in \mathds{R}^{n}$ and $0 \leqslant s \leqslant T-t_{0}$, and take conditional expectations with respect to $X(T-s) = x$ in the integral version of the stochastic differential \hyperref[logybetapro]{(\ref*{logybetapro})}. On account of \hyperref[martpoflogybibilt]{(\ref*{martpoflogybibilt})}, the stochastic integral with respect to the $\mathds{P}$-Brownian motion $(\overline{W}^{\mathds{P}}(T-s))_{0 \leqslant s \leqslant T}$ in \hyperref[logybetapro]{(\ref*{logybetapro})} vanishes. Dividing by $T-t_{0}-s$ and passing to the limit as $s \uparrow T-t_{0}$, we can use the estimate \hyperref[ilpdepvhfvsvwp]{(\ref*{ilpdepvhfvsvwp})} to observe that the expression in the third line of \hyperref[logybetapro]{(\ref*{logybetapro})} vanishes in the limit. After applying the Cauchy–Schwarz inequality, we see that the normalized integral involving the perturbation $\beta$ appearing in the first line of \hyperref[logybetapro]{(\ref*{logybetapro})} can be bounded by
\begin{equation} \label{trtotflitfube}
\frac{1}{T-t_{0}-s} \int_{s}^{T-t_{0}} \Big\vert \nabla \log Y^{\beta}\big(T-u,X(T-u)\big) \Big\vert \cdot \big\vert \beta\big(X(T-u)\big) \big\vert \, \textnormal{d}u.
\end{equation}
By conditions \hyperref[tsaosaojko]{\ref*{tsaosaojko}}, \hyperref[naltsaosaojkos]{\ref*{naltsaosaojkos}} of \hyperref[sosaojkoia]{Assumptions \ref*{sosaojkoia}}, the function $(t,x) \mapsto \nabla \log Y^{\beta}(t,x)$ is continuous on $(0,\infty) \times \mathds{R}^{n}$, thus the expression in \hyperref[trtotflitfube]{(\ref*{trtotflitfube})} is uniformly bounded on the rectangle $[0,T-t_{0}] \times \operatorname{supp} \beta$. As $\log Y^{\beta}(t_{0}, \, \cdot \,) = 0$, it converges $\mathds{P}$-almost surely to zero, hence also
\begin{equation} 
\lim_{s \uparrow T-t_{0}} \, \mathds{E}_{\mathds{P}}\Bigg[\frac{1}{T-t_{0}-s} \int_{s}^{T-t_{0}} \Big\vert \nabla \log Y^{\beta}\big(T-u,X(T-u)\big) \Big\vert \cdot \big\vert \beta\big(X(T-u)\big) \big\vert \, \textnormal{d}u  \ \bigg\vert \ X(T-s) = x \Bigg] 
= 0.
\end{equation}
Finally, using continuity and uniform boundedness once again, the conditional expectations of the normalized integrals over the second line of \hyperref[logybetapro]{(\ref*{logybetapro})} converge to the right-hand side of \hyperref[taflawhplftw]{(\ref*{taflawhplftw})}, as claimed.
\end{proof}
\end{lemma}

\begin{remark}[\textsf{Stability of the entropy limits under perturbations}] \label{rkresolmz} The above \hyperref[hctclwittmeitpocasot]{Lemma \ref*{hctclwittmeitpocasot}} justifies the estimates \hyperref[ilpdepvhfv]{(\ref*{ilpdepvhfv})} and \hyperref[ilpdepvhfvsv]{(\ref*{ilpdepvhfvsv})}, which we have used in the proof of \hyperref[thetsixcoraopv]{Corollary \ref*{thetsixcoraopv}}. They were the crucial ingredients in the effort to show that the exceptional set for the limiting assertion \hyperref[rgtsflffflnv]{(\ref*{rgtsflffflnv})} does not change when passing from the unperturbed to the perturbed equation \hyperref[tatpvotgdbi]{(\ref*{tatpvotgdbi})}. It is now time to come back to this technical issue.

\smallskip

As a general observation, we stress that no worries about limits of difference quotients arise as long as we remain in the realm of an \textit{integral formulation} of our results, as opposed to passing to a \textit{differential formulation}. It is precisely the spirit of our basic trajectorial \hyperref[thetsix]{Theorems \ref*{thetsix}} and \hyperref[thetthretv]{\ref*{thetthretv}}, that they are naturally formulated in \textit{integral terms}.

\smallskip

We also note that the problem of exceptional points does not arise if we impose regularity assumptions strong enough, so that the limiting assertions \hyperref[flfffl]{(\ref*{flfffl})} and \hyperref[flffflpv]{(\ref*{flffflpv})} are valid \textit{for all $t_{0} > 0$}, or even \textit{for all $t_{0} \geqslant 0$}, instead of \textit{for Lebesgue-almost every $t_{0} \geqslant 0$}. For example, this follows if we impose, in addition to \hyperref[sosaojkoia]{Assumptions \ref*{sosaojkoia}}, the \textit{a priori} assumption that the relative Fisher information function $t \mapsto I(P(t) \, \vert \, \mathrm{Q})$ is continuous on $(0,\infty)$, or continuous on $[0,\infty)$, respectively. 

\smallskip

Having made these general observations, let us now be more technical and have a precise look at the exceptional sets in the framework of the regularity codified by \hyperref[sosaojkoia]{Assumptions \ref*{sosaojkoia}}. In the proof of \hyperref[thetsixcorao]{Corollary \ref*{thetsixcorao}} we have deduced from the Lebesgue differentiation theorem that the generalized de Bruijn identity \hyperref[flffflnv]{(\ref*{flffflnv})}, \hyperref[flfffl]{(\ref*{flfffl})} is valid outside a set of Lebesgue measure zero. In particular, let us recall from \hyperref[rethpoitwooslda]{Remark \ref*{rethpoitwooslda}} that $N$ denotes the set of exceptional points $t_{0} \geqslant 0$ for which the right-sided limiting assertion \hyperref[rgtsflffflnv]{(\ref*{rgtsflffflnv})}, i.e.,
\begin{equation} \label{remoeacm}
\lim_{t \downarrow t_{0}} \, \frac{H\big( P(t) \, \vert \, \mathrm{Q} \big) - H\big( P(t_{0}) \, \vert \, \mathrm{Q}\big)}{t-t_{0}} 
= - \tfrac{1}{2} \, \| a \|_{L^{2}(\mathds{P})}^{2},
\end{equation}
fails. We have shown in the equations \hyperref[llethar]{(\ref*{llethar})} and \hyperref[tciwsolafophvs]{(\ref*{tciwsolafophvs})} from the proof of \hyperref[thetsixcoraopv]{Corollary \ref*{thetsixcoraopv}} that the limiting assertion \hyperref[tatpvotgdbi]{(\ref*{tatpvotgdbi})}, \hyperref[flffflpv]{(\ref*{flffflpv})}, i.e.,
\begin{equation} \label{remoeacdfm}
\lim_{t \downarrow t_{0}} \, \frac{H\big( P^{\beta}(t) \, \vert \, \mathrm{Q} \big) - H\big( P^{\beta}(t_{0}) \, \vert \, \mathrm{Q}\big)}{t-t_{0}} 
= -  \tfrac{1}{2} \, \big\langle a , a + 2b \big\rangle_{L^{2}(\mathds{P})},
\end{equation}
is valid if and only if $t_{0} \in \mathds{R}_{+} \setminus N$. In other words, the limits in \hyperref[remoeacm]{(\ref*{remoeacm})} and \hyperref[remoeacdfm]{(\ref*{remoeacdfm})} have the same exceptional set $N$ of zero Lebesgue measure. Put another way, the entropy limit \hyperref[remoeacm]{(\ref*{remoeacm})} is stable under perturbations, in the sense that the corresponding perturbed entropy limit \hyperref[remoeacdfm]{(\ref*{remoeacdfm})} continues to be valid for the same points $t_{0} \in \mathds{R}_{+} \setminus N$.

\smallskip

Furthermore, in the proofs of \hyperref[thetsixcor]{Propositions \ref*{thetsixcor}} and \hyperref[thetthretvcor]{\ref*{thetthretvcor}} we have seen that the limiting assertions therein are valid, respectively, for those points $t_{0} \geqslant 0$ for which the generalized de Bruijn identity \textnormal{\hyperref[flfffl]{(\ref*{flfffl})}} does hold, and for $t_{0}  \in \mathds{R}_{+} \setminus N$.
\end{remark}

\smallskip

\begin{remark}[\textsf{Stability of the Wasserstein limits under perturbations}] \label{wvotarrkresolmz} Let us now pass to the limits of the difference quotients pertaining to the Wasserstein distance. We fix some $t_{0} \in \mathds{R}_{+} \setminus N$ so that the limiting assertion \hyperref[remoeacm]{(\ref*{remoeacm})}, and as a consequence also \hyperref[remoeacdfm]{(\ref*{remoeacdfm})}, are valid. Then the unperturbed Wasserstein limit 
\begin{equation} \label{osagswtuffasihnvasd}
\lim_{t \downarrow t_{0}} \, \frac{W_{2}\big(P(t),P(t_{0})\big)}{t - t_{0}} = \tfrac{1}{2}  \, \| a \|_{L^{2}(\mathds{P})}
\end{equation}
of \hyperref[osagswtuffasihnv]{(\ref*{osagswtuffasihnv})} is valid as well. This is remarkable, because a priori there is no significant relation between relative entropy and Wasserstein distance, except for the fact that in the limit the relative Fisher information appears on the right-hand sides of both \hyperref[remoeacm]{(\ref*{remoeacm})} and \hyperref[osagswtuffasihnvasd]{(\ref*{osagswtuffasihnvasd})}. Even more, the unperturbed Wasserstein limit \hyperref[ftlbotwditpc]{(\ref*{ftlbotwditpc})}, \hyperref[svpvompvv]{(\ref*{svpvompvv})}, i.e.,
\begin{equation} \label{ftlbotwditpcasd}
\lim_{t \downarrow t_{0}} \, \frac{W_{2}\big( P^{\beta}(t),P^{\beta}(t_{0})\big)}{t-t_{0}}
=  \tfrac{1}{2}  \, \| a + 2 b\|_{L^{2}(\mathds{P})},
\end{equation}
also holds for this point $t_{0} \in \mathds{R}_{+} \setminus N$. In other words, the Wasserstein limits are stable under perturbations in the same manner as the entropy limits are.

\smallskip

Summing up, not only do the limiting assertions \hyperref[remoeacm]{(\ref*{remoeacm})} and \hyperref[remoeacdfm]{(\ref*{remoeacdfm})} hold for every $t_{0} \in \mathds{R}_{+} \setminus N$, but so do also the limiting assertions \hyperref[osagswtuffasihnvasd]{(\ref*{osagswtuffasihnvasd})} and \hyperref[ftlbotwditpcasd]{(\ref*{ftlbotwditpcasd})} pertaining to the Wasserstein distance. We will prove these results in \hyperref[agswt]{Theorems \ref*{agswt}} and \hyperref[bvagswt]{\ref*{bvagswt}} of \hyperref[stwt]{Section \ref*{stwt}}.
\end{remark}


\subsection{Completing the proof of \texorpdfstring{\hyperref[thetthretvcor]{Proposition \ref*{thetthretvcor}}}{Proposition 3.14}} \label{nspofthetthretvcor}


On account of the preparations in \hyperref[subsomusefullem]{Subsection \ref*{subsomusefullem}} above, we are now able to complete the proof of \hyperref[thetthretvcor]{Proposition \ref*{thetthretvcor}} by establishing the remaining limiting assertions \hyperref[titmlwhtcfs]{(\ref*{titmlwhtcfs})} and \hyperref[thetsixcorfessl]{(\ref*{thetsixcorfessl})} therein. 

\begin{proof}[Proof of the assertion \texorpdfstring{\textnormal{\hyperref[titmlwhtcfs]{(\ref*{titmlwhtcfs})}}}{} in \texorpdfstring{\hyperref[thetthretvcor]{Proposition \ref*{thetthretvcor}}}{}:] Let $t_{0} > 0$ and select $T > t_{0}$. Using the notation of \hyperref[witpetounper]{(\ref*{witpetounper})} above, we have to calculate the limit 
\begin{equation} \label{tfwhtccftbsav}
\lim_{s \uparrow T-t_{0}} \, \frac{\log Y^{\beta}\big(T-s,X(T-s)\big)}{T-t_{0}-s}.
\end{equation}
Fix $0 \leqslant s \leqslant T-t_{0}$. According to the integral version of the stochastic differential \hyperref[logybetapro]{(\ref*{logybetapro})}, the fraction in \hyperref[tfwhtccftbsav]{(\ref*{tfwhtccftbsav})} is equal to the sum of the following four normalized integral terms \hyperref[tfwhtccftbsavfi]{(\ref*{tfwhtccftbsavfi})} -- \hyperref[tfwhtccftbsavfith]{(\ref*{tfwhtccftbsavfith})} and \hyperref[tfwhtccftbsavfifo]{(\ref*{tfwhtccftbsavfifo})}, whose behavior as $s \uparrow T-t_{0}$ we will study separately below. By conditions \hyperref[tsaosaojko]{\ref*{tsaosaojko}}, \hyperref[naltsaosaojkos]{\ref*{naltsaosaojkos}} of \hyperref[sosaojkoia]{Assumptions \ref*{sosaojkoia}}, the function $(t,x) \mapsto \nabla \log Y^{\beta}(t,x)$ is continuous on $(0,\infty) \times \mathds{R}^{n}$, thus the first expression 
\begin{equation} \label{tfwhtccftbsavfi}
\frac{1}{T-t_{0}-s} \int_{s}^{T-t_{0}} \bigg( \operatorname{div} \beta \big(X(T-u)\big) + \Big\langle \beta\big(X(T-u)\big) \, , \nabla \log p\big(T-u,X(T-u)\big) \Big\rangle_{\mathds{R}^{n}}  \bigg) \, \textnormal{d}u
\end{equation}
is uniformly bounded on $[0,T-t_{0}] \times \operatorname{supp} \beta$. Using continuity and uniform boundedness, we conclude that \hyperref[tfwhtccftbsavfi]{(\ref*{tfwhtccftbsavfi})} converges $\mathds{P}$-almost surely as well as in $L^{1}(\mathds{P})$ to the right-hand side of \hyperref[titmlwhtcfs]{(\ref*{titmlwhtcfs})}, as required. Thus it remains to show that the three remaining terms converge to zero. Using continuity and uniform boundedness once again, we deduce from $\log Y^{\beta}(t_{0}, \, \cdot \,) = 0$ that the second integral term 
\begin{equation} \label{tfwhtccftbsavfis}
\frac{1}{T-t_{0}-s} \int_{s}^{T-t_{0}} \Bigg\langle \frac{\nabla Y^{\beta}\big(T-u,X(T-u)\big)}{Y^{\beta}\big(T-u,X(T-u)\big)} \, , \, \beta\big(X(T-u)\big)  \Bigg\rangle_{\mathds{R}^{n}} \, \textnormal{d}u
\end{equation}
converges to zero $\mathds{P}$-almost surely and in $L^{1}(\mathds{P})$. Since $\log Y^{\beta}(t_{0}, \, \cdot \,) = 0$ and because the integrand is continuous, we see that the third expression
\begin{equation} \label{tfwhtccftbsavfith}
\frac{1}{T-t_{0}-s} \int_{s}^{T-t_{0}} \frac{1}{2} \frac{\big\vert \nabla Y^{\beta}\big(T-u,X(T-u)\big) \big\vert^{2}}{Y^{\beta}\big(T-u,X(T-u)\big)^{2}} \, \textnormal{d}u
\end{equation}
converges $\mathds{P}$-almost surely to zero. Furthermore, owing to \hyperref[hctclwittmeitpocasot]{Lemma \ref*{hctclwittmeitpocasot}}, there is a constant $C > 0$ such that 
\begin{equation} \label{ubfiswwwu}
\mathds{E}_{\mathds{P}}\Bigg[\frac{1}{T - t_{0} - s}\int_{s}^{T-t_{0}} \frac{\big\vert \nabla Y^{\beta}\big(T-u,X(T-u)\big) \big\vert^{2}}{Y^{\beta}\big(T-u,X(T-u)\big)^{2}} \, \textnormal{d}u \Bigg] 
\leqslant C \, (T - t_{0} - s)
\end{equation}
holds for all $0 \leqslant s \leqslant T-t_{0}$, which implies that \hyperref[tfwhtccftbsavfith]{(\ref*{tfwhtccftbsavfith})} converges to zero also in $L^{1}(\mathds{P})$. The fourth and last term is the stochastic integral
\begin{equation} \label{tfwhtccftbsavfifo}
- \frac{1}{T-t_{0}-s} \int_{s}^{T-t_{0}} \Bigg\langle \frac{\nabla Y^{\beta}\big(T-u,X(T-u)\big)}{Y^{\beta}\big(T-u,X(T-u)\big)} \, , \, \textnormal{d}\overline{W}^{\mathds{P}}(T-u) \Bigg\rangle_{\mathds{R}^{n}}.
\end{equation}
The expression \hyperref[tfwhtccftbsavfith]{(\ref*{tfwhtccftbsavfith})} converges to zero $\mathds{P}$-almost surely, and according to \hyperref[ubfiswwwu]{(\ref*{ubfiswwwu})} we have
\begin{equation} 
\mathds{E}_{\mathds{P}}\Bigg[\frac{1}{(T - t_{0} - s)^{2}}\int_{s}^{T-t_{0}} \frac{\big\vert \nabla Y^{\beta}\big(T-u,X(T-u)\big) \big\vert^{2}}{Y^{\beta}\big(T-u,X(T-u)\big)^{2}} \, \textnormal{d}u \Bigg] \leqslant C.
\end{equation}
By means of the It\^{o} isometry, we deduce that
\begin{equation} 
\lim_{s \uparrow T-t_{0}} \, \mathds{E}_{\mathds{P}}\Bigg[ \, \Bigg( \frac{1}{T-t_{0}-s} \int_{s}^{T-t_{0}} \Bigg\langle \frac{\nabla Y^{\beta}\big(T-u,X(T-u)\big)}{Y^{\beta}\big(T-u,X(T-u)\big)} \, , \, \textnormal{d}\overline{W}^{\mathds{P}}(T-u) \Bigg\rangle_{\mathds{R}^{n}}  \, \Bigg)^{2} \, \Bigg] = 0.
\end{equation}
In other words, the normalized stochastic integral of \hyperref[tfwhtccftbsavfifo]{(\ref*{tfwhtccftbsavfifo})} converges to zero in $L^{2}(\mathds{P})$.

\smallskip

Summing up, we have shown that the limiting assertion \hyperref[titmlwhtcfs]{(\ref*{titmlwhtcfs})} holds in $L^{1}(\mathds{P})$ and is valid for all $t_{0} > 0$. As we have seen in \hyperref[hctclwittmeitpocasotpv]{Lemma \ref*{hctclwittmeitpocasotpv}}, the probability measures $\mathds{P}$ and $\mathds{P}^{\beta}$ are equivalent, the Radon-Nikod\'{y}m derivatives $\frac{\textnormal{d}\mathds{P}^{\beta}}{\textnormal{d}\mathds{P}}$ and $\frac{\textnormal{d}\mathds{P}}{\textnormal{d}\mathds{P}^{\beta}}$ are bounded on the $\sigma$-algebra $\mathcal{F}(T) = \mathcal{G}(0)$, and therefore convergence in $L^{1}(\mathds{P})$ is equivalent to convergence in $L^{1}(\mathds{P}^{\beta})$. This completes the proof of the limiting assertion \hyperref[titmlwhtcfs]{(\ref*{titmlwhtcfs})}.
\end{proof}

\begin{proof}[Proof of the assertion \texorpdfstring{\textnormal{\hyperref[thetsixcorfessl]{(\ref*{thetsixcorfessl})}}}{} in \texorpdfstring{\hyperref[thetthretvcor]{Proposition \ref*{thetthretvcor}}}{}:] This is proved in very much the same way, as assertions \hyperref[thetsixcorfes]{(\ref*{thetsixcorfes})}, \hyperref[titmlwhtcfs]{(\ref*{titmlwhtcfs})}. The only novelty here, is the use of \hyperref[dbowptmtowpbtmtfp]{(\ref*{dbowptmtowpbtmtfp})} to pass to the $\mathds{P}$-Brownian motion $\big(\overline{W}^{\mathds{P}}(T-s)\big)_{0 \leqslant s \leqslant T-t_{0}}$ from the $\mathds{P}^{\beta}$-Brownian motion $\big(\overline{W}^{\mathds{P}^{\beta}}(T-s)\big)_{0 \leqslant s \leqslant T-t_{0}}$, and the reliance on \hyperref[fecwfoeltpsn]{(\ref*{fecwfoeltpsn})} to ensure that the resulting stochastic integral is a (square-integrable) $\mathds{P}$-martingale. We leave the details to the care of the diligent reader, or refer to \cite[Section 3.2]{Tsc19}.
\end{proof}

\subsection{The dynamics in the forward direction of time} 


For the sake of completeness, we calculate now the stochastic differentials of the relative entropy process \hyperref[stlrpd]{(\ref*{stlrpd})} and its perturbed counterpart of \hyperref[prerepitufdllpitufd]{(\ref*{prerepitufdllpitufd})} also in the forward direction of time, under the measures $\mathds{P}$ and $\mathds{P}^{\beta}$, respectively. It will turn out that we are able to derive \hyperref[thetone]{Theorems \ref*{thetone}} and \hyperref[thetthre]{\ref*{thetthre}} also by applying It\^{o}'s formula in the forward direction of time. But the relations between these theorems and the stochastic differentials will become less transparent than in reverse time. In fact, additional terms will show up in the forward direction of time, which on the contrary did not appear in the backward direction. One may still take expectations, but to obtain \hyperref[thetone]{Theorems \ref*{thetone}} and \hyperref[thetthre]{\ref*{thetthre}} one has to argue why the expectations of these additional terms vanish. This is straightforward in the unperturbed situation of \hyperref[fhlupdotlrpfu]{Lemma \ref*{fhlupdotlrpfu}}, but in the perturbed context of \hyperref[hltplrpsdef]{Lemma \ref*{hltplrpsdef}} one also has to rely on integration by parts. 

\smallskip

We first compute the differentials of the likelihood ratio process $\ell(t,X(t))$, $t \geqslant 0$ of \hyperref[rndlr]{(\ref*{rndlr})} and of its logarithm as in \hyperref[stlrpd]{(\ref*{stlrpd})}, in the forward direction of time. We start by recalling the backwards Kolmogorov equation \hyperref[fpdefflrf]{(\ref*{fpdefflrf})}. With its help, we can compute the forward dynamics of the likelihood ratio process \hyperref[rndlr]{(\ref*{rndlr})} in the following manner.

\begin{lemma} \label{fhlupdotlrpfu} Under the \textnormal{\hyperref[sosaojkoia]{Assumptions \ref*{sosaojkoia}}}, the likelihood ratio process \textnormal{\hyperref[rndlr]{(\ref*{rndlr})}} and its logarithm satisfy the stochastic differential equations
\begingroup
\addtolength{\jot}{0.7em}
\begin{align}
\textnormal{d} \ell\big(t,X(t)\big) 
&= \Delta \ell\big(t,X(t)\big) \, \textnormal{d}t + \Big\langle \nabla \ell\big(t,X(t)\big) \, , \, \textnormal{d}W(t) - 2 \, \nabla \Psi\big(X(t)\big) \, \textnormal{d}t \Big\rangle_{\mathds{R}^{n}}  \label{updotlrpfu} \\
&= 2 \, \partial_{t} \ell\big(t,X(t)\big) \, \textnormal{d}t + \Big\langle \nabla \ell\big(t,X(t)\big) \, , \, \textnormal{d}W(t) \Big\rangle_{\mathds{R}^{n}}  \label{updotlrpfus}
\end{align} 
\endgroup
and
\begingroup
\addtolength{\jot}{0.7em}
\begin{align}
\textnormal{d} \log \ell\big(t,X(t)\big) &= \Bigg( \ \frac{\Delta \ell\big(t,X(t)\big)}{\ell\big(t,X(t)\big)} 
- \Bigg\langle \frac{\nabla \ell\big(t,X(t)\big)}{\ell\big(t,X(t)\big)} \, , \, 2 \, \nabla \Psi\big(X(t)\big) \Bigg\rangle_{\mathds{R}^{n}}  \ \Bigg) \ \textnormal{d}t \label{lredotlrpfuup} \\
& \qquad \qquad \qquad \qquad  \qquad - \frac{1}{2} \frac{\big\vert \nabla \ell\big(t,X(t)\big)\big\vert^{2}}{\ell\big(t,X(t)\big)^{2}} \, \textnormal{d}t \, + \, \Bigg\langle \frac{\nabla \ell\big(t,X(t)\big)}{\ell\big(t,X(t)\big)} \, , \, \textnormal{d}W(t) \Bigg\rangle_{\mathds{R}^{n}}  \label{lredotlrpfuups} \\
&= \Big( \, 2 \, \partial_{t} \log \ell\big(t,X(t)\big) - \tfrac{1}{2} \big\vert \nabla \log \ell\big(t,X(t)\big)\big\vert^{2} \, \Big) \, \textnormal{d}t + \Big\langle \nabla \log \ell\big(t,X(t)\big) \, , \, \textnormal{d}W(t) \Big\rangle_{\mathds{R}^{n}}, \label{lredotlrpfuupst}
\end{align}
\endgroup
respectively, for $t \geqslant 0$, with respect to the forward filtration $(\mathcal{F}(t))_{t \geqslant 0}$.
\begin{proof} Applying It\^{o}'s formula and using the equations \hyperref[sdeids]{(\ref*{sdeids})}, \hyperref[fpdefflrf]{(\ref*{fpdefflrf})} shows \hyperref[updotlrpfu]{(\ref*{updotlrpfu})}, \hyperref[updotlrpfus]{(\ref*{updotlrpfus})}. One more application of It\^{o}'s formula leads to the stochastic equations \hyperref[lredotlrpfuup]{(\ref*{lredotlrpfuup})} -- \hyperref[lredotlrpfuupst]{(\ref*{lredotlrpfuupst})}.  
\end{proof}
\end{lemma}

In order to deduce \hyperref[thetone]{Theorem \ref*{thetone}} from \hyperref[fhlupdotlrpfu]{Lemma \ref*{fhlupdotlrpfu}} --- at least formally --- we take expectations in \hyperref[lredotlrpfuup]{(\ref*{lredotlrpfuup})} -- \hyperref[lredotlrpfuupst]{(\ref*{lredotlrpfuupst})} and use \hyperref[fpdefflrf]{(\ref*{fpdefflrf})} to observe that 
\begin{equation} \label{expzitfdotbntceo}
\mathds{E}_{\mathds{P}}\Bigg[ \, \frac{\Delta \ell\big(t,X(t)\big)}{\ell\big(t,X(t)\big)}
- \Bigg\langle \frac{\nabla \ell\big(t,X(t)\big)}{\ell\big(t,X(t)\big)} \, , \, 2 \, \nabla \Psi\big(X(t)\big) \Bigg\rangle_{\mathds{R}^{n}} \, \Bigg] =
\mathds{E}_{\mathds{P}}\Bigg[ \, \frac{2 \, \partial_{t} \ell\big(t,X(t)\big)}{\ell\big(t,X(t)\big)} \, \Bigg]= 0.
\end{equation}

\smallskip

Next, we calculate the differentials of the perturbed likelihood ratio process $\ell^{\beta}(t,X(t))$, $t \geqslant t_{0}$, as in \hyperref[pvdotplrfbs]{(\ref*{pvdotplrfbs})}, and of its logarithm as in \hyperref[prerepitufdllpitufd]{(\ref*{prerepitufdllpitufd})}, again in the forward direction. With the help of the ``perturbed'' backwards Kolmogorov equation \hyperref[pfpeeffe]{(\ref*{pfpeeffe})}, we obtain the forward dynamics of the perturbed likelihood ratio process \hyperref[pvdotplrfbs]{(\ref*{pvdotplrfbs})}, as follows.

\begin{lemma} \label{hltplrpsdef} Under the \textnormal{\hyperref[sosaojkoia]{Assumptions \ref*{sosaojkoia}}}, let $t_{0} \geqslant 0$. The perturbed likelihood ratio process \textnormal{\hyperref[pvdotplrfbs]{(\ref*{pvdotplrfbs})}} and its logarithm satisfy the stochastic differential equations
\begingroup
\addtolength{\jot}{0.7em}
\begin{align} 
\frac{\textnormal{d} \ell^{\beta}\big(t,X(t)\big)}{\ell^{\beta}\big(t,X(t)\big)} 
&= \Bigg( \ \frac{\Delta \ell^{\beta}\big(t,X(t)\big)}{\ell^{\beta}\big(t,X(t)\big)} - \Bigg\langle \frac{\nabla \ell^{\beta}\big(t,X(t)\big)}{\ell^{\beta}\big(t,X(t)\big)} \, , \, 2 \, \nabla \Psi\big(X(t)\big) \Bigg\rangle_{\mathds{R}^{n}} \ \Bigg) \ \textnormal{d}t \label{plrpsdef} \\
& \qquad \qquad + \Big( \operatorname{div} \beta - \big\langle \beta \, , \, 2 \, \nabla \Psi \big\rangle_{\mathds{R}^{n}} \Big)\big(X(t)\big) \, \textnormal{d}t \, + \, \Bigg\langle \frac{\nabla \ell^{\beta}\big(t,X(t)\big)}{\ell^{\beta}\big(t,X(t)\big)} \, , \, \textnormal{d}W^{\beta}(t) \Bigg\rangle_{\mathds{R}^{n}} \\
&= \Big( \, 2 \, \partial_{t} \log \ell^{\beta}\big(t,X(t)\big) - \Big\langle \nabla \log \ell^{\beta}\big(t,X(t)\big) \, , \, \beta\big(X(t)\big) \Big\rangle_{\mathds{R}^{n}} \, \Big) \, \textnormal{d}t \\
& \qquad \qquad + \Big( \big\langle \beta \, , \, 2 \, \nabla \Psi \big\rangle_{\mathds{R}^{n}} - \operatorname{div} \beta \Big)\big(X(t)\big) \, \textnormal{d}t \, + \, \Big\langle \nabla \log \ell^{\beta}\big(t,X(t)\big) \, , \, \textnormal{d}W^{\beta}(t) \Big\rangle_{\mathds{R}^{n}} 
\end{align}
\endgroup
and
\begingroup
\addtolength{\jot}{0.7em}
\begin{align} 
&\textnormal{d} \log \ell^{\beta}\big(t,X(t)\big) 
= \Bigg( \ \frac{\Delta \ell^{\beta}\big(t,X(t)\big)}{\ell^{\beta}\big(t,X(t)\big)} - \Bigg\langle \frac{\nabla \ell^{\beta}\big(t,X(t)\big)}{\ell^{\beta}\big(t,X(t)\big)} \, , \, 2 \, \nabla \Psi\big(X(t)\big) \Bigg\rangle_{\mathds{R}^{n}} \ \Bigg) \ \textnormal{d}t \label{logplrpsdeffe} \\
&+ \Bigg( \Big( \operatorname{div} \beta - \big\langle \beta \, , \, 2 \, \nabla \Psi \big\rangle_{\mathds{R}^{n}} \Big)\big(X(t)\big) - \frac{1}{2} \frac{\big\vert \nabla \ell^{\beta}\big(t,X(t)\big)\big\vert^{2}}{\ell^{\beta}\big(t,X(t)\big)^{2}} \ \Bigg) \ \textnormal{d}t \, + \, \Bigg\langle \frac{\nabla \ell^{\beta}\big(t,X(t)\big)}{\ell^{\beta}\big(t,X(t)\big)} \, , \, \textnormal{d}W^{\beta}(t) \Bigg\rangle_{\mathds{R}^{n}} \label{logplrpsdefse} \\
& \hspace{74.5pt} = \Big( \, 2 \, \partial_{t} \log \ell^{\beta}\big(t,X(t)\big) - \Big\langle \nabla \log \ell^{\beta}\big(t,X(t)\big) \, , \, \beta\big(X(t)\big) \Big\rangle_{\mathds{R}^{n}} \, \Big) \, \textnormal{d}t \label{thlogplrpsdef} \\
&+ \bigg(  \Big( \big\langle \beta \, , \, 2 \, \nabla \Psi \big\rangle_{\mathds{R}^{n}} - \operatorname{div} \beta \Big)\big(X(t)\big) - \tfrac{1}{2} \big\vert \nabla \log \ell^{\beta}\big(t,X(t)\big)\big\vert^{2} \, \bigg) \, \textnormal{d}t \, + \, \Big\langle \nabla \log \ell^{\beta}\big(t,X(t)\big) \, , \, \textnormal{d}W^{\beta}(t) \Big\rangle_{\mathds{R}^{n}}, \label{logplrpsdef}
\end{align}
\endgroup
respectively, for $t \geqslant t_{0}$, with respect to the forward filtration $(\mathcal{F}(t))_{t \geqslant t_{0}}$.
\begin{proof} Using \hyperref[wpsdeids]{(\ref*{wpsdeids})}, \hyperref[pfpeeffe]{(\ref*{pfpeeffe})} and It\^{o}'s formula, we obtain the stochastic equations \hyperref[plrpsdef]{(\ref*{plrpsdef})} -- \hyperref[logplrpsdef]{(\ref*{logplrpsdef})}.
\end{proof} 
\end{lemma}

The perturbed situation is not as nice as the unperturbed one, since according to \hyperref[pfpeeffe]{(\ref*{pfpeeffe})} we have 
\begin{equation} 
\Delta \ell^{\beta}(t,x) - \big\langle \nabla \ell^{\beta}(t,x) \, , \, 2 \, \nabla \Psi(x) \big\rangle_{\mathds{R}^{n}} \neq 2 \, \partial_{t} \ell^{\beta}(t,x)
\end{equation}
in general. However, integrating by parts shows that 
\begin{equation} \label{nexpzitfdotbntceo}
\mathds{E}_{\mathds{P}^{\beta}}\Bigg[ \, \frac{\Delta \ell^{\beta}\big(t,X(t)\big)}{\ell^{\beta}\big(t,X(t)\big)}
- \Bigg\langle \frac{\nabla \ell^{\beta}\big(t,X(t)\big)}{\ell^{\beta}\big(t,X(t)\big)} \, , \, 2 \, \nabla \Psi\big(X(t)\big) \Bigg\rangle_{\mathds{R}^{n}} \, \Bigg] = 0.
\end{equation}
Hence taking expectations in \hyperref[logplrpsdeffe]{(\ref*{logplrpsdeffe})}, \hyperref[logplrpsdefse]{(\ref*{logplrpsdefse})} allows to derive \hyperref[thetthre]{Theorem \ref*{thetthre}} from \hyperref[hltplrpsdef]{Lemma \ref*{hltplrpsdef}}, at least on a formal level. But as opposed to the backward direction of time, the identity \hyperref[nexpzitfdotbntceo]{(\ref*{nexpzitfdotbntceo})} does not hold any more when we take expectations conditionally on $X(t)$.


\section{The Wasserstein transport} \label{stwt} 


For the convenience of the reader we review in \hyperref[stwtbnat]{Subsections \ref*{stwtbnat}} and \hyperref[sstmdoaccinqws]{\ref*{sstmdoaccinqws}} some well-known results on quadratic Wasserstein transport \cite{AGS08,AG13}, in order to establish the limits \hyperref[agswtuffasih]{(\ref*{agswtuffasih})} and \hyperref[svpvompvv]{(\ref*{svpvompvv})} and complete the proofs of \hyperref[thetsixcorao]{Corollaries \ref*{thetsixcorao}} and \hyperref[thetsixcoraopv]{\ref*{thetsixcoraopv}}. For a detailed discussion of metric measure spaces and in particular Wasserstein spaces, see also the work \cite{Stu06a,Stu06b} by Sturm.

\smallskip

As we make precise statements regarding the points $t_{0} \geqslant 0$ at which the the limiting assertions \hyperref[agswtuffasih]{(\ref*{agswtuffasih})} and \hyperref[svpvompvv]{(\ref*{svpvompvv})} are valid (recall \hyperref[wvotarrkresolmz]{Remark \ref*{wvotarrkresolmz}} at this point), we provide detailed proofs of the relevant \hyperref[agswt]{Theorems \ref*{agswt}}, \hyperref[bvagswt]{\ref*{bvagswt}} in \hyperref[stwtbnatlbotwdt]{Subsection \ref*{stwtbnatlbotwdt}}.

\subsection{Basic notation and terminology} \label{stwtbnat}

We recall below the definitions of the quadratic Wasserstein space $\mathscr{P}_{2}(\mathds{R}^{n})$, and of the quadratic Wasserstein distance $W_{2}$. We follow the setting of \cite{AGS08}, from where we borrow most of the notation and terminology used in this section. Thus, for unexplained notions and definitions, the reader may consult this beautiful book.

\medskip

We denote by $\mathscr{P}(\mathds{R}^{n})$ the collection of probability measures on the Borel sets $\mathscr{B}(\mathds{R}^{n})$ of $\mathds{R}^{n}$. The \textit{quadratic Wasserstein space} $\mathscr{P}_{2}(\mathds{R}^{n})$ is the subset of $\mathscr{P}(\mathds{R}^{n})$ consisting of the probability measures on $\mathscr{B}(\mathds{R}^{n})$ with finite second moment, i.e., 
\begin{equation} \label{ptrn}
\mathscr{P}_{2}(\mathds{R}^{n}) \vcentcolon = \bigg\{ P \in \mathscr{P}(\mathds{R}^{n}) \colon \ \int_{\mathds{R}^{n}} \vert x \vert^{2} \, \textnormal{d}P(x) < \infty \bigg\}.
\end{equation}

\smallskip

If $p \colon \mathds{R}^{n} \rightarrow [0,\infty)$ is a probability density function on $\mathds{R}^{n}$, we can identify it with the probability measure $P \in \mathscr{P}(\mathds{R}^{n})$ having density $p$ with respect to Lebesgue measure on $\mathds{R}^{n}$. In particular, if $p$ is a probability density with finite second moment, i.e.,
\begin{equation} \label{pdptrn}
\int_{\mathds{R}^{n}} \vert x \vert^{2} \, p(x) \, \textnormal{d}x < \infty, 
\end{equation}
then we can identify $p$ with an element of $\mathscr{P}_{2}(\mathds{R}^{n})$.

\smallskip

We denote by $\Gamma(P_{1},P_{2})$ the collection of Kantorovich transport plans, that is, probability measures $\boldsymbol{\gamma}$ in $\mathscr{P}(\mathds{R}^{n} \times \mathds{R}^{n})$ with given marginals $P_{1}, P_{2} \in \mathscr{P}(\mathds{R}^{n})$. More precisely, if $\pi^{i} \colon \mathds{R}^{n} \times \mathds{R}^{n} \rightarrow \mathds{R}^{n}$ are the canonical projections, then $\pi_{\#}^{i} \boldsymbol{\gamma} = P_{i}$, for $i \in \{1,2\}$. The \textit{quadratic Wasserstein distance} between two probability measures $P_{1},P_{2} \in \mathscr{P}_{2}(\mathds{R}^{n})$ is defined by 
\begin{equation} \label{doqwd}
W_{2}^{2}(P_{1},P_{2}) \vcentcolon = \inf \bigg\{ \int_{\mathds{R}^{n} \times \mathds{R}^{n}} \vert x - y \vert^{2} \, \textnormal{d}\boldsymbol{\gamma}(x,y) \colon \ \boldsymbol{\gamma} \in \Gamma(P_{1},P_{2}) \bigg\}.
\end{equation}
The quadratic Wasserstein space $\mathscr{P}_{2}(\mathds{R}^{n})$, endowed with the quadratic Wasserstein distance $W_{2}$ just introduced, is a Polish space \cite[Proposition 7.1.5]{AGS08}. 

\subsection{The metric derivative of curves in the Wasserstein space} \label{sstmdoaccinqws}

\textit{In the present section we consider the solution $(p(t))_{t \geqslant 0}$ of the Fokker-Planck equation \textnormal{\hyperref[fpeqnwfp]{(\ref*{fpeqnwfp})}} with initial condition \textnormal{\hyperref[icnwfp]{(\ref*{icnwfp})}} as a curve in the quadratic Wasserstein space $\mathscr{P}_{2}(\mathds{R}^{n})$.} This is justified by \hyperref[fosmolds]{Lemma \ref*{fosmolds}}, on the basis of which the \hyperref[sosaojkoia]{Assumptions \ref*{sosaojkoia}} guarantee that $p(t) \in \mathscr{P}_{2}(\mathds{R}^{n})$ for all $t \geqslant 0$. For fixed $T \in (0,\infty)$, we define now the time-dependent velocity field 
\begin{equation} \label{tdvfvtx}
[0,T] \times \mathds{R}^{n} \ni (t,x) \longmapsto v(t,x) \vcentcolon = - \bigg( \frac{1}{2} \frac{\nabla p(t,x)}{p(t,x)} + \nabla \Psi(x) \bigg) = - \frac{1}{2} \frac{\nabla \ell(t,x)}{\ell(t,x)} \in \mathds{R}^{n}
\end{equation}
that consists of two parts: the drift $-\nabla \Psi( \, \cdot \,)$ of the underlying motion; and the speed $-\frac{1}{2} \frac{\nabla p(t, \, \cdot \,)}{p(t, \, \cdot \,)}$ of the transport induced by the diffusive motion with transition mechanism $p(t, \, \cdot \,)$, in the manner of \hyperref[a.8]{(\ref*{a.8})}. Then the Fokker-Planck equation \hyperref[fpeqnwfp]{(\ref*{fpeqnwfp})}, satisfied by the curve $(p(t))_{0 \leqslant t \leqslant T}$ of probability density functions in $\mathscr{P}_{2}(\mathds{R}^{n})$, can be cast as a \textit{continuity}, or \textit{linear transport}, \textit{equation}, namely,
\begin{equation} \label{cefpe}
\partial_{t} p(t,x) + \operatorname{div} \big( v(t,x) \, p(t,x) \big) = 0, \qquad (t,x) \in (0,T] \times \mathds{R}^{n}.
\end{equation}
According to \hyperref[rfi]{(\ref*{rfi})}, \hyperref[fecwfoel]{(\ref*{fecwfoel})} and by definition of the velocity field $v(t) \equiv v(t, \, \cdot \, )$, we have
\begin{equation} \label{icfwc}
\tfrac{1}{2} \int_{0}^{T} I\big( P(t) \, \vert \, \mathrm{Q}\big) \, \textnormal{d}t = 2 \int_{0}^{T} \bigg( \int_{\mathds{R}^{n}} \vert v(t,x) \vert^{2} \, p(t,x) \, \textnormal{d}x \bigg) \, \textnormal{d}t < \infty.
\end{equation}

\smallskip

The quadratic Wasserstein distance of \hyperref[doqwd]{(\ref*{doqwd})}, and the continuity equation \hyperref[cefpe]{(\ref*{cefpe})}, are tied together intimately. Indeed, for any two probability measures $P_{0},P_{1}$ in $\mathscr{P}_{2}(\mathds{R}^{n})$ that admit density functions $\rho_{0}(\, \cdot \,)$ and $\rho_{1}(\, \cdot \,)$, respectively, we have the ``minimum kinetic energy'' representation
\begin{equation} \label{5.7.bb}
W_{2}^{2}(P_{0},P_{1}) = \inf \int_{0}^{1} \bigg( \int_{\mathds{R}^{n}} \vert v(t,x) \vert^{2} \, \rho(t,x) \, \textnormal{d}x \bigg) \, \textnormal{d}t
\end{equation}
of \cite{BB00}. Here, the infimum is taken over all pairs of vector fields $(\rho,v)$, scalar and vector, respectively, that satisfy the equation \hyperref[cefpe]{(\ref*{cefpe})} as well as the initial and terminal conditions $\rho(0, \, \cdot \,) = \rho_{0}(\, \cdot \,)$, $\rho(1, \, \cdot \,) = \rho_{1}(\, \cdot \,)$. The representation \hyperref[5.7.bb]{(\ref*{5.7.bb})} provides a strong justification for the relevance of the Wasserstein distance in our context, which is indeed governed by an equation (Fokker-Planck) of transport type.

\smallskip

As $(p(t))_{0 \leqslant t \leqslant T}$ is a curve in the Wasserstein space $\mathscr{P}_{2}(\mathds{R}^{n})$ satisfying the continuity equation \hyperref[cefpe]{(\ref*{cefpe})} and the integrability condition \hyperref[icfwc]{(\ref*{icfwc})}, we can invoke Theorem 8.3.1 in \cite{AGS08}. This result relates absolutely continuous curves in $\mathscr{P}_{2}(\mathds{R}^{n})$ to the continuity equation. In particular, its second implication states that the curve $(p(t))_{0 \leqslant t \leqslant T}$ is absolutely continuous \cite[Definition 1.1.1]{AGS08}. As a consequence, for Lebesgue-almost every $t_{0} \in [0,T]$, its \textit{metric derivative} \cite[Theorem 1.1.2]{AGS08}  
\begin{equation} \label{dotmd}
\vert p' \vert(t_{0}) \vcentcolon = \lim_{t \rightarrow t_{0}} \, \frac{W_{2}\big(p(t),p(t_{0})\big)}{\vert t - t_{0} \vert}
\end{equation}
exists. Furthermore, \cite[Theorem 8.3.1]{AGS08} provides for Lebesgue-almost every $t_{0} \in [0,T]$ the estimate
\begin{equation} \label{wcfi}
\vert p' \vert(t_{0}) \leqslant \| v(t_{0}) \|_{L^{2}(P(t_{0}))}.
\end{equation}

\smallskip

On the other hand, according to condition \hyperref[nalwstasas]{\ref*{nalwstasas}} in \hyperref[sosaojkoianoew]{Assumptions \ref*{sosaojkoianoew}}, the time-dependent gradient vector field $v \colon [0,T] \times \mathds{R}^{n} \rightarrow \mathds{R}^{n}$ of \hyperref[tdvfvtx]{(\ref*{tdvfvtx})} is an element of the tangent space \cite[Definition 8.4.1]{AGS08} of $\mathscr{P}_{2}(\mathds{R}^{n})$ at the point $P(t) \in \mathscr{P}_{2}(\mathds{R}^{n})$, i.e.,
\begin{equation} \label{tanpcvecup}
v(t, \, \cdot \, ) \in \textnormal{Tan}_{P(t)} \mathscr{P}_{2}(\mathds{R}^{n}) \vcentcolon = \overline{\big\{ \nabla \varphi \colon \ \varphi \in \mathcal{C}_{c}^{\infty}(\mathds{R}^{n};\mathds{R})\big\}}^{L^{2}(P(t))}.
\end{equation}
Since $(p(t))_{0 \leqslant t \leqslant T}$ is an absolutely continuous curve in the quadratic Wasserstein space $\mathscr{P}_{2}(\mathds{R}^{n})$ satisfying the continuity equation \hyperref[cefpe]{(\ref*{cefpe})} for the time-dependent velocity field $v(t)$, which is tangent to the curve, we can apply Proposition 8.4.5 of \cite{AGS08}. This result characterizes tangent vectors to absolutely continuous curves, and entails for Lebesgue-almost every $t_{0} \in [0,T]$ the inequality
\begin{equation} \label{wcsi}
\vert p' \vert(t_{0}) \geqslant \| v(t_{0}) \|_{L^{2}(P(t_{0}))}.
\end{equation}

Combining \hyperref[wcfi]{(\ref*{wcfi})} and \hyperref[wcsi]{(\ref*{wcsi})}, we obtain for Lebesgue-almost every $t_{0} \in [0,T]$ the equality
\begin{equation} \label{wcne}
\vert p' \vert(t_{0}) = \| v(t_{0}) \|_{L^{2}(P(t_{0}))}.
\end{equation}
This relates the strength of the time-dependent velocity field $v(t, \, \cdot \,)$ in \hyperref[cefpe]{(\ref*{cefpe})}, to the rate of change, or metric derivative as in \hyperref[dotmd]{(\ref*{dotmd})}, of the Wasserstein distance along the curve $(p(t))_{0 \leqslant t \leqslant T}$ --- justifying in this manner the introduction and relevance of the Wasserstein distance in this context.

\subsection{The local behavior of the Wasserstein distance} \label{stwtbnatlbotwdt}

In the previous section we derived along the lines of \cite{AGS08} the explicit representation \hyperref[wcne]{(\ref*{wcne})} of the metric derivative of the quadratic Wasserstein distance along the Fokker-Planck probability density flow. This limit exists for Lebesgue-almost every $t_{0} \geqslant 0$. The following result provides accurate information regarding the points $t_{0} \geqslant 0$ at which the metric derivative \hyperref[wcne]{(\ref*{wcne})} exists (cf.\ \hyperref[wvotarrkresolmz]{Remark \ref*{wvotarrkresolmz}}); we record its instructive proof.

\begin{theorem}[\textsf{Local behavior of the quadratic Wasserstein distance}] \label{agswt} Under the \textnormal{\hyperref[sosaojkoianoew]{Assumptions \ref*{sosaojkoianoew}}}, let $t_{0} \geqslant 0$ be such that the generalized de Bruijn identity \textnormal{\hyperref[flffflnv]{(\ref*{flffflnv})}}, \textnormal{\hyperref[flfffl]{(\ref*{flfffl})}} is valid. Then we have 
\begin{equation} \label{agswtuff}
\lim_{t \rightarrow t_{0}} \, \frac{W_{2}\big(P(t),P(t_{0})\big)}{\vert t - t_{0} \vert} 
= \bigg( \mathds{E}_{\mathds{P}}\Big[ \, \big\vert v\big(t_{0},X(t_{0})\big) \big\vert^{2} \, \Big] \bigg)^{1/2} 
= \tfrac{1}{2} \, \sqrt{I\big( P(t_{0}) \, \vert \, \mathrm{Q}\big)}.
\end{equation}
\end{theorem}

\bigskip

Instead of \hyperref[agswt]{Theorem \ref*{agswt}}, we will prove the more general \hyperref[bvagswt]{Theorem \ref*{bvagswt}} below, which amounts to the perturbed version of \hyperref[agswt]{Theorem \ref*{agswt}}. The latter then simply follows by setting $\beta \equiv 0$ in the statement of \hyperref[bvagswt]{Theorem \ref*{bvagswt}}.

\smallskip

In order to formulate \hyperref[bvagswt]{Theorem \ref*{bvagswt}}, we consider the solution $(p^{\beta}(t))_{t \geqslant t_{0}}$ of the perturbed Fokker-Planck equation \hyperref[pfpeq]{(\ref*{pfpeq})} with initial condition \hyperref[pic]{(\ref*{pic})}. Again, according to \hyperref[fosmoldspv]{Lemma \ref*{fosmoldspv}}, this solution can be viewed as a curve in the quadratic Wasserstein space $\mathscr{P}_{2}(\mathds{R}^{n})$. Just as before, we define the time-dependent perturbed velocity field 
\begin{equation} \label{pbtdvfvb}
[t_{0},T] \times \mathds{R}^{n} \ni (t,x) \longmapsto v^{\beta}(t,x) \vcentcolon = - \bigg( \frac{1}{2} \frac{\nabla p^{\beta}(t,x)}{p^{\beta}(t,x)} + \nabla \Psi(x) + \beta(x) \bigg) \in \mathds{R}^{n}.
\end{equation} 
Then the perturbed Fokker-Planck equation \hyperref[pfpeq]{(\ref*{pfpeq})}, satisfied by the perturbed curve $(p^{\beta}(t))_{t_{0} \leqslant t \leqslant T}$, can once again be written as a continuity equation, to wit
\begin{equation} 
\partial_{t} p^{\beta}(t,x) + \operatorname{div} \big( v^{\beta}(t,x) \, p^{\beta}(t,x) \big) = 0, \qquad (t,x) \in (t_{0},T] \times \mathds{R}^{n}.
\end{equation}
At this point, let us recall that we have required the perturbation $\beta \colon \mathds{R}^{n} \rightarrow \mathds{R}^{n}$ to be a \textit{gradient vector field}, i.e., of the form $\beta = \nabla B$ for some smooth potential $B \colon \mathds{R}^{n} \rightarrow \mathds{R}$ with compact support. Since $p(t_{0}, \, \cdot \,) = p^{\beta}(t_{0}, \, \cdot \,)$, at time $t_{0}$ the vector fields of \hyperref[tdvfvtx]{(\ref*{tdvfvtx})} and \hyperref[pbtdvfvb]{(\ref*{pbtdvfvb})} are related via
\begin{equation} \label{rbvbavwitob}
v^{\beta}(t_{0},x) = v(t_{0},x) - \nabla B(x), \qquad x \in \mathds{R}^{n}.
\end{equation}
Using the regularity assumption that the potential $B$ is of class $\mathcal{C}_{c}^{\infty}(\mathds{R}^{n};\mathds{R})$, we conclude from \hyperref[tanpcvecup]{(\ref*{tanpcvecup})} and \hyperref[rbvbavwitob]{(\ref*{rbvbavwitob})} that the perturbed vector field $v^{\beta}(t_{0}, \, \cdot \,)$ is an element of the tangent space of $\mathscr{P}_{2}(\mathds{R}^{n})$ at the point $P^{\beta}(t_{0}) = P(t_{0}) \in \mathscr{P}_{2}(\mathds{R}^{n})$, i.e.,
\begin{equation} \label{perttaspaver}
v^{\beta}(t_{0}, \, \cdot \,) \in \textnormal{Tan}_{P^{\beta}(t_{0})} \mathscr{P}_{2}(\mathds{R}^{n}) 
= \overline{\big\{ \nabla \varphi^{\beta} \colon \ \varphi^{\beta} \in \mathcal{C}_{c}^{\infty}(\mathds{R}^{n};\mathds{R})\big\}}^{L^{2}(P^{\beta}(t_{0}))}.
\end{equation}

\smallskip

After these preparations we can formulate the perturbed version of \hyperref[agswt]{Theorem \ref*{agswt}} as follows.

\begin{theorem}[\textsf{Local behavior of the quadratic Wasserstein distance under perturbations}] \label{bvagswt} Under the \textnormal{\hyperref[sosaojkoianoew]{Assumptions \ref*{sosaojkoianoew}}}, for every $t_{0} \in \mathds{R}_{+} \setminus N$ we have
\begin{equation} \label{pvosagswtuff}
\lim_{t \downarrow t_{0}} \, \frac{W_{2}\big( P^{\beta}(t),P^{\beta}(t_{0})\big)}{t-t_{0}} 
= \bigg( \mathds{E}_{\mathds{P}}\Big[ \, \big\vert v^{\beta}\big(t_{0},X(t_{0})\big) \big\vert^{2} \, \Big] \bigg)^{1/2} 
= \tfrac{1}{2}  \, \| a + 2 b\|_{L^{2}(\mathds{P})},
\end{equation}
where $a = \nabla \log \ell(t_{0},X(t_{0}))$ and $b = \beta(X(t_{0}))$.
\end{theorem}

\begin{remark} With this notation, the rightmost side of \hyperref[agswtuff]{(\ref*{agswtuff})} is $\frac{1}{2} \|a\|_{L^{2}(\mathds{P})}$. Since $X(t_{0})$ has the same probability distribution under $\mathds{P}$, as it does under $\mathds{P}^{\beta}$, the expectation $\mathds{E}_{\mathds{P}}$ appearing in \hyperref[pvosagswtuff]{(\ref*{pvosagswtuff})} can be replaced by $\mathds{E}_{\mathds{P}^{\beta}}$. Let us also recall from \hyperref[rethpoitwooslda]{Remark \ref*{rethpoitwooslda}} the exceptional set $N$ consisting of those points $t_{0} \geqslant 0$ for which the limiting assertion \hyperref[rgtsflffflnv]{(\ref*{rgtsflffflnv})} fails.  
\end{remark}

\begin{proof}[Proof of \texorpdfstring{\hyperref[bvagswt]{Theorem \ref*{bvagswt}}}{}] The second equality in \hyperref[pvosagswtuff]{(\ref*{pvosagswtuff})} is apparent from the definition of the time-dependent perturbed velocity field $\big(v^{\beta}(t, \, \cdot \,)\big)_{t \geqslant t_{0}}$ from \hyperref[pbtdvfvb]{(\ref*{pbtdvfvb})} above. The delicate point is to show that the limiting assertion \hyperref[pvosagswtuff]{(\ref*{pvosagswtuff})} is valid for every $t_{0} \in \mathds{R}_{+} \setminus N$. 

\smallskip

In order to see this, let us fix some $t_{0} \in \mathds{R}_{+} \setminus N$ so that the limiting identity \hyperref[rgtsflffflnv]{(\ref*{rgtsflffflnv})} is valid. In the following steps we prove that then also the limiting assertion \hyperref[pvosagswtuff]{(\ref*{pvosagswtuff})} does hold.

\medskip

\noindent \fbox{\textsf{Step 1.}} The vector field $v^{\beta}(t_{0}, \, \cdot \,)$ induces a family of \textit{linearized transport maps} $(\Upsilon_{t}^{\beta})_{t \geqslant t_{0}}$ defined by
\begin{equation} \label{dofaijkophi}
\Upsilon_{t}^{\beta}(x) \vcentcolon = x + (t-t_{0}) \cdot v^{\beta}(t_{0},x), \qquad x \in \mathds{R}^{n}
\end{equation}
in the manner of \hyperref[hlotlse]{(\ref*{hlotlse})}, and we denote by $P_{\Upsilon}^{\beta}(t)$ the image measure of $P^{\beta}(t_{0}) = P(t_{0})$ under the transport map $\Upsilon_{t}^{\beta} \colon \mathds{R}^{n} \rightarrow \mathds{R}^{n}$; i.e.,
\begin{equation} 
P_{\Upsilon}^{\beta}(t) \vcentcolon = (\Upsilon_{t}^{\beta})_{\#} P^{\beta}(t_{0}), \qquad t \geqslant t_{0}.
\end{equation}
To motivate the subsequent arguments, let us first pretend that, for all $t > t_{0}$ sufficiently close to $t_{0}$, the map $\Upsilon_{t}^{\beta}$ is the \textit{optimal quadratic Wasserstein transport} from $P^{\beta}(t_{0})$ to $P_{\Upsilon}^{\beta}(t)$; i.e.,
\begin{equation}
W_{2}^{2}\big(P_{\Upsilon}^{\beta}(t),P^{\beta}(t_{0})\big) = \mathds{E}_{\mathds{P}^{\beta}}\Big[ \, \big\vert \Upsilon_{t}^{\beta}\big(X(t_{0})\big) - X(t_{0}) \big\vert^{2} \, \Big] = \mathds{E}_{\mathds{P}}\Big[ \, \big\vert \Upsilon_{t}^{\beta}\big(X(t_{0})\big) - X(t_{0}) \big\vert^{2} \, \Big],
\end{equation}
where we have used in the last equality the fact that $X(t_{0})$ has the same distribution under $\mathds{P}^{\beta}$ as it does under $\mathds{P}$. Then, on account of \hyperref[dofaijkophi]{(\ref*{dofaijkophi})}, we could conclude that
\begin{equation} \label{svrlinfifat}
\lim_{t \downarrow t_{0}} \,\frac{W_{2}\big(P_{\Upsilon}^{\beta}(t),P^{\beta}(t_{0})\big)}{t-t_{0}} = \bigg( \mathds{E}_{\mathds{P}}\Big[ \, \big\vert v^{\beta}\big(t_{0},X(t_{0})\big) \big\vert^{2} \, \Big] \bigg)^{1/2}
= \tfrac{1}{2}  \, \| a + 2 b\|_{L^{2}(\mathds{P})}.
\end{equation} 
Furthermore, let us suppose that we can show the limiting identity
\begin{equation} \label{svrlinfifatsvwct}
\lim_{t \downarrow t_{0}} \, \frac{W_{2}\big(P^{\beta}(t),P_{\Upsilon}^{\beta}(t)\big)}{t-t_{0}} = 0.
\end{equation} 
Using \hyperref[svrlinfifat]{(\ref*{svrlinfifat})} and \hyperref[svrlinfifatsvwct]{(\ref*{svrlinfifatsvwct})}, we would now derive the desired equality \hyperref[pvosagswtuff]{(\ref*{pvosagswtuff})}. Indeed, invoking the triangle inequality for the quadratic Wasserstein distance we obtain 
\begin{equation} \label{tiotwwsdfv}
\lim_{t \downarrow t_{0}} \, \frac{W_{2}\big(P_{\Upsilon}^{\beta}(t),P^{\beta}(t_{0})\big)}{t-t_{0}} \, \leqslant \, \lim_{t \downarrow t_{0}} \, \frac{W_{2}\big(P_{\Upsilon}^{\beta}(t),P^{\beta}(t)\big)}{t-t_{0}} \, + \, \liminf_{t \downarrow t_{0}} \, \frac{W_{2}\big(P^{\beta}(t),P^{\beta}(t_{0})\big)}{t-t_{0}},
\end{equation}
and one more application of it yields
\begin{equation} \label{tiotwwsdfvsv}
\limsup_{t \downarrow t_{0}} \,\frac{W_{2}\big(P^{\beta}(t),P^{\beta}(t_{0})\big)}{t-t_{0}} \, \leqslant \,  \lim_{t \downarrow t_{0}} \, \frac{W_{2}\big(P^{\beta}(t),P_{\Upsilon}^{\beta}(t)\big)}{t-t_{0}} \, + \, \lim_{t \downarrow t_{0}} \, \frac{W_{2}\big(P_{\Upsilon}^{\beta}(t),P^{\beta}(t_{0})\big)}{t-t_{0}}.
\end{equation}

\medskip

\noindent \fbox{\textsf{Step 2.}} The bad news at this point, is that there is little reason why, for $t > t_{0}$ sufficiently close to $t_{0}$, the map $\Upsilon_{t}^{\beta}$ defined in \hyperref[dofaijkophi]{(\ref*{dofaijkophi})} of \textsf{Step 1} should be optimal with respect to quadratic Wasserstein transportation costs; i.e., by \textit{Brenier's theorem} \cite{Bre91}, equal to the gradient of a convex function. The good news is that we can reduce the general case to the situation of optimal transports $\Upsilon_{t}^{\beta}$ as in \textsf{Step 1} by localizing the vector field $v^{\beta}(t_{0}, \, \cdot \,)$ as well as the transport maps $(\Upsilon_{t}^{\beta})_{t \geqslant t_{0}}$ to compact subsets of $\mathds{R}^{n}$ (\textsf{Steps 2} -- \textsf{4}); and that, after these localizations have been carried out, an analogue of the identity \hyperref[svrlinfifatsvwct]{(\ref*{svrlinfifatsvwct})} also holds, allowing us to complete the argument (\textsf{Steps 5} -- \textsf{7}).

\smallskip

To this end, we first recall that the perturbation $\beta \colon \mathds{R}^{n} \rightarrow \mathds{R}^{n}$ is a gradient vector field, i.e., of the form $\beta = \nabla B$ for some smooth, compactly supported potential $B \colon \mathds{R}^{n} \rightarrow \mathds{R}$. Thus the vector field $v^{\beta}(t_{0}, \, \cdot \,)$ from \hyperref[pbtdvfvb]{(\ref*{pbtdvfvb})}, \hyperref[rbvbavwitob]{(\ref*{rbvbavwitob})} can be represented as a gradient, namely
\begin{equation} 
v^{\beta}(t_{0},x) = - \nabla \Big( \tfrac{1}{2} \log \ell(t_{0},x) + B(x) \Big), \qquad x \in \mathds{R}^{n}.
\end{equation} 
Even more, according to \hyperref[perttaspaver]{(\ref*{perttaspaver})}, it is an element of the tangent space 
\begin{equation}
\textnormal{Tan}_{P^{\beta}(t_{0})} \mathscr{P}_{2}(\mathds{R}^{n}) 
= \overline{\big\{ \nabla \varphi^{\beta} \colon \ \varphi^{\beta} \in \mathcal{C}_{c}^{\infty}(\mathds{R}^{n};\mathds{R})\big\}}^{L^{2}(P^{\beta}(t_{0}))}
\end{equation}
of the quadratic Wasserstein space $\mathscr{P}_{2}(\mathds{R}^{n})$ at the point $P^{\beta}(t_{0}) \in \mathscr{P}_{2}(\mathds{R}^{n})$. Therefore we can choose a sequence of potential functions $\big(\varphi_{m}^{\beta}(t_{0}, \cdot \,)\big)_{m \geqslant 1} \subseteq \mathcal{C}_{c}^{\infty}(\mathds{R}^{n};\mathds{R})$ such that
\begin{equation} \label{aotvfgosfvgosfwcs}
\lim_{m \rightarrow \infty} \, \mathds{E}_{\mathds{P}}\bigg[ \ \Big\vert v^{\beta}\big((t_{0},X(t_{0})\big) - \nabla \varphi_{m}^{\beta}\big(t_{0},X(t_{0})\big) \Big\vert^{2} \ \bigg] = 0.
\end{equation}

\smallskip

Next, for each $m \in \mathds{N}$, we define the \textit{localized gradient vector fields}
\begin{equation} \label{lvfvbrtox}
v_{m}^{\beta}(t_{0},x) \vcentcolon = \nabla \varphi_{m}^{\beta}(t_{0},x), \qquad x \in \mathds{R}^{n}.
\end{equation} 
By construction, these have compact support and approximate the gradient vector field $v^{\beta}(t_{0}, \, \cdot \,)$ in $L^{2}(P(t_{0}))$, as in \hyperref[aotvfgosfvgosfwcs]{(\ref*{aotvfgosfvgosfwcs})}.

\smallskip 

Finally, for every $m \in \mathds{N}$, the localized gradient vector field $v_{m}^{\beta}(t_{0}, \, \cdot \,)$ induces a family of \textit{localized linear transports} $(\Upsilon_{t}^{\beta,m})_{t \geqslant t_{0}}$ defined by analogy with \hyperref[dofaijkophi]{(\ref*{dofaijkophi})} via
\begin{equation} \label{defphbrtaopws}
\Upsilon_{t}^{\beta,m}(x) \vcentcolon = x + (t-t_{0}) \cdot v_{m}^{\beta}(t_{0},x), \qquad x \in \mathds{R}^{n}.
\end{equation}
We denote by $P_{\Upsilon}^{\beta,m}(t)$ the image measure of $P^{\beta}(t_{0}) = P(t_{0})$ under this localized linear transport map $\Upsilon_{t}^{\beta,m} \colon \mathds{R}^{n} \rightarrow \mathds{R}^{n}$; i.e.,
\begin{equation} \label{delltmpphbrtf}
P_{\Upsilon}^{\beta,m}(t) \vcentcolon = (\Upsilon_{t}^{\beta,m})_{\#} P^{\beta}(t_{0}), \qquad t \geqslant t_{0}.
\end{equation} 

\medskip

\noindent \fbox{\textsf{Step 3.}} We claim that, for every $m \in \mathds{N}$, there exists some $\varepsilon_{m} > 0$ such that for all $t > t_{0}$ with $\vert t - t_{0} \vert < \varepsilon_{m}$, the localized linear transport map $\Upsilon_{t}^{\beta,m} \colon \mathds{R}^{n} \rightarrow \mathds{R}^{n}$ constructed in \textsf{Step 2} defines an optimal quadratic Wasserstein transport from $P^{\beta}(t_{0})$ to $P_{\Upsilon}^{\beta,m}(t)$. Hence, by \textit{Brenier's theorem} (\cite{Bre91}, \cite[Theorem 2.12]{Vil03}), we have to show that $\Upsilon_{t}^{\beta,m}$ is the \textit{gradient of a convex function}, for all $t > t_{0}$ sufficiently near $t_{0}$. 

\smallskip

Indeed, from the definitions in \hyperref[lvfvbrtox]{(\ref*{lvfvbrtox})}, \hyperref[defphbrtaopws]{(\ref*{defphbrtaopws})} we see that the functions $\Upsilon_{t}^{\beta,m}$ are gradients for all $m \in \mathds{N}$ and $t \geqslant t_{0}$. More precisely, we have
\begin{equation}
\Upsilon_{t}^{\beta,m}(x) = \nabla \Big( \tfrac{1}{2} \vert x \vert^{2} + (t-t_{0}) \cdot \varphi_{m}^{\beta}(t_{0},x) \Big), \qquad x \in \mathds{R}^{n}.
\end{equation}
Therefore, it remains to show that the function $\frac{1}{2} \vert \cdot \vert^{2} + (t-t_{0}) \cdot \varphi_{m}^{\beta}(t_{0}, \, \cdot \,)$ is convex for every $m \in \mathds{N}$ and $t > t_{0}$ sufficiently close to $t_{0}$. The Hessian matrix of this function is given by
\begin{equation} \label{tmatsppsdfttsm}
I_{n} + (t-t_{0}) \cdot \textnormal{Hess}\big(\varphi_{m}^{\beta}(t_{0},x)\big), \qquad x \in \mathds{R}^{n}.
\end{equation}
In order to deduce the desired convexity, we have to verify that the Hessian matrix of \hyperref[tmatsppsdfttsm]{(\ref*{tmatsppsdfttsm})} is positive semidefinite for all $t > t_{0}$ sufficiently near $t_{0}$, uniformly in $x \in \mathds{R}^{n}$. To see this, let us fix $m \in \mathds{N}$. Now the identity matrix $I_{n}$ is positive definite and the Hessian matrix of the function $\varphi_{m}^{\beta}(t_{0}, \, \cdot \,)$ is symmetric. Furthermore, we recall that the smooth function $\varphi_{m}^{\beta}(t_{0}, \, \cdot \,)$ has \textit{compact support}, which is crucial in order to justify the present argument. Checking the defining condition guaranteeing positive semidefiniteness of the matrix in \hyperref[tmatsppsdfttsm]{(\ref*{tmatsppsdfttsm})} for unit vectors and using compactness as well as continuity, we obtain the existence of some $\varepsilon_{m} > 0$ such that, for all $t > t_{0}$ with $\vert t - t_{0} \vert < \varepsilon_{m}$, the Hessian matrix of \hyperref[tmatsppsdfttsm]{(\ref*{tmatsppsdfttsm})} is positive semidefinite (in fact, positive definite).

\medskip

\noindent \fbox{\textsf{Step 4.}} From \textsf{Step 3} we know that, for every $m \in \mathds{N}$, there exists some $\varepsilon_{m} > 0$ such that for all $t > t_{0}$ with $\vert t - t_{0} \vert < \varepsilon_{m}$, the localized map $\Upsilon_{t}^{\beta,m}$ is the optimal transport from $P^{\beta}(t_{0})$ to $P_{\Upsilon}^{\beta,m}(t)$ with respect to quadratic Wasserstein costs. Therefore, we can apply the considerations of \textsf{Step 1} to the optimal map $\Upsilon_{t}^{\beta,m}$ in \hyperref[defphbrtaopws]{(\ref*{defphbrtaopws})}, and deduce that
\begin{equation} \label{svrlinfifatnv}
\lim_{t \downarrow t_{0}} \,\frac{W_{2}\big(P_{\Upsilon}^{\beta,m}(t),P^{\beta}(t_{0})\big)}{t-t_{0}} = \bigg( \mathds{E}_{\mathds{P}}\Big[ \, \big\vert v_{m}^{\beta}\big(t_{0},X(t_{0})\big) \big\vert^{2} \, \Big] \bigg)^{1/2}
\end{equation} 
holds for every $m \in \mathds{N}$. Invoking \hyperref[aotvfgosfvgosfwcs]{(\ref*{aotvfgosfvgosfwcs})} and \hyperref[lvfvbrtox]{(\ref*{lvfvbrtox})}, we obtain from this
\begin{equation} 
\lim_{m \rightarrow \infty} \, \lim_{t \downarrow t_{0}} \,\frac{W_{2}\big(P_{\Upsilon}^{\beta,m}(t),P^{\beta}(t_{0})\big)}{t-t_{0}} = \bigg( \mathds{E}_{\mathds{P}}\Big[ \, \big\vert v^{\beta}\big(t_{0},X(t_{0})\big) \big\vert^{2} \, \Big] \bigg)^{1/2} = \tfrac{1}{2}  \, \| a + 2 b\|_{L^{2}(\mathds{P})}.
\end{equation} 
From the inequalities \hyperref[tiotwwsdfv]{(\ref*{tiotwwsdfv})} and \hyperref[tiotwwsdfvsv]{(\ref*{tiotwwsdfvsv})} of \textsf{Step 1} (with $P_{\Upsilon}^{\beta,m}(t)$ instead of $P_{\Upsilon}^{\beta}(t)$) it follows that, in order to conclude \hyperref[pvosagswtuff]{(\ref*{pvosagswtuff})}, it remains to establish the analogue
\begin{equation} \label{nvotlisvrlinfifatsvwct}
\lim_{m \rightarrow \infty} \, \lim_{t \downarrow t_{0}} \, \frac{W_{2}\big(P^{\beta}(t),P_{\Upsilon}^{\beta,m}(t)\big)}{t-t_{0}} = 0
\end{equation} 
of the identity \hyperref[svrlinfifatsvwct]{(\ref*{svrlinfifatsvwct})}.

\medskip

\noindent \fbox{\textsf{Step 5.}} The time-dependent velocity field $\big(v^{\beta}(t, \, \cdot \,)\big)_{t \geqslant t_{0}}$ induces a \textit{curved flow} $(\Phi_{t}^{\beta})_{t \geqslant t_{0}}$, which is characterized by  
\begin{equation} \label{cotcfl}
\tfrac{\textnormal{d}}{\textnormal{d}t} \, \Phi_{t}^{\beta} = v^{\beta}(t,\Phi_{t}^{\beta}) \quad \textnormal{ for all } t \geqslant t_{0} \, , \qquad \Phi_{t_{0}}^{\beta} = \textnormal{Id}_{\mathds{R}^{n}}.
\end{equation}
Then, for every $t \geqslant t_{0}$, the map $\Phi_{t}^{\beta} \colon \mathds{R}^{n} \rightarrow \mathds{R}^{n}$ transports the measure $P^{\beta}(t_{0}) = P(t_{0})$ to $P^{\beta}(t)$, i.e., $(\Phi_{t}^{\beta})_{\#} P^{\beta}(t_{0}) = P^{\beta}(t)$.

\smallskip

The localized linear transports $\Upsilon_{t}^{\beta,m} \colon \mathds{R}^{n} \rightarrow \mathds{R}^{n}$ defined in \hyperref[defphbrtaopws]{(\ref*{defphbrtaopws})} of \textsf{Step 2} transport $P^{\beta}(t_{0})$ to $P_{\Upsilon}^{\beta,m}(t)$, see \hyperref[delltmpphbrtf]{(\ref*{delltmpphbrtf})}. As $P^{\beta}(t_{0})$ and $P_{\Upsilon}^{\beta,m}(t)$ have full support and are absolutely continuous with respect to Lebesgue measure, the inverse map $(\Upsilon_{t}^{\beta,m})^{-1} \colon \mathds{R}^{n} \rightarrow \mathds{R}^{n}$ is well-defined and satisfies
\begin{equation} 
\big((\Upsilon_{t}^{\beta,m})^{-1}\big)_{\#} P_{\Upsilon}^{\beta,m}(t) = P^{\beta}(t_{0}), \qquad t \geqslant t_{0}.
\end{equation} 

\smallskip

Recall from \textsf{Step 4} that our remaining task is to prove \hyperref[nvotlisvrlinfifatsvwct]{(\ref*{nvotlisvrlinfifatsvwct})}. To this end, we have to construct maps $\mathfrak{X}_{t}^{\beta,m} \colon \mathds{R}^{n} \rightarrow \mathds{R}^{n}$ that transport $P_{\Upsilon}^{\beta,m}(t)$ to $P^{\beta}(t)$, i.e., $(\mathfrak{X}_{t}^{\beta,m})_{\#} P_{\Upsilon}^{\beta,m}(t) = P^{\beta}(t)$, and satisfy
\begin{equation} \label{twhteaoncmne}
\lim_{m \rightarrow \infty} \, \lim_{t \downarrow t_{0}} \, \frac{1}{t-t_{0}} \, \Bigg( \, \mathds{E}_{\mathds{P}_{\Upsilon}^{\beta,m}}\bigg[ \ \Big\vert \mathfrak{X}_{t}^{\beta,m}\big(X(t)\big) - X(t) \Big\vert^{2} \ \bigg]  \, \Bigg)^{1/2} = 0  \, ,
\end{equation}
where $\mathds{P}_{\Upsilon}^{\beta,m}$ denotes a probability measure on the path space under which the random variable $X(t)$ has distribution $P_{\Upsilon}^{\beta,m}(t)$ as in \hyperref[delltmpphbrtf]{(\ref*{delltmpphbrtf})}. We define now the candidate maps
\begin{equation} \label{defphitcamap}
\mathfrak{X}_{t}^{\beta,m} \vcentcolon = \Phi_{t}^{\beta} \circ \big( \Upsilon_{t}^{\beta,m} \big)^{-1}, \qquad t \geqslant t_{0}
\end{equation}
for this job, recall that $(\Upsilon_{t}^{\beta,m})^{-1}$ transports $P_{\Upsilon}^{\beta,m}(t)$ to $P^{\beta}(t_{0})$ while $\Phi_{t}^{\beta}$ transports $P^{\beta}(t_{0})$ to $P^{\beta}(t)$, and conclude that $\mathfrak{X}_{t}^{\beta,m}$ of \hyperref[defphitcamap]{(\ref*{defphitcamap})} transports $P_{\Upsilon}^{\beta,m}(t)$ to $P^{\beta}(t)$; thus, we have 
\begin{equation} \label{olwtwhteaoncmne}
\mathds{E}_{\mathds{P}_{\Upsilon}^{\beta,m}}\bigg[ \ \Big\vert \mathfrak{X}_{t}^{\beta,m}\big(X(t)\big) - X(t) \Big\vert^{2} \ \bigg] 
= \mathds{E}_{\mathds{P}}\bigg[ \ \Big\vert \Phi_{t}^{\beta}\big(X(t_{0})\big) - \Upsilon_{t}^{\beta,m}\big(X(t_{0})\big) \Big\vert^{2} \ \bigg].
\end{equation}
Combining \hyperref[twhteaoncmne]{(\ref*{twhteaoncmne})} and \hyperref[olwtwhteaoncmne]{(\ref*{olwtwhteaoncmne})}, we see that we have to establish
\begin{equation}  \label{olebaewtwhteaoncmne}
\lim_{m \rightarrow \infty} \, \lim_{t \downarrow t_{0}} \, \frac{1}{(t-t_{0})^{2}} \, \mathds{E}_{\mathds{P}}\bigg[ \ \Big\vert \Phi_{t}^{\beta}\big(X(t_{0})\big) - \Upsilon_{t}^{\beta,m}\big(X(t_{0})\big) \Big\vert^{2} \ \bigg] = 0.
\end{equation}
Using \hyperref[defphbrtaopws]{(\ref*{defphbrtaopws})} and the elementary inequality $\vert x+y \vert^{2} \leqslant 2(\vert x \vert^{2}+\vert y \vert^{2})$, for $x,y \in \mathds{R}^{n}$, we derive the estimate
\begingroup
\addtolength{\jot}{0.7em}
\begin{align}
\tfrac{1}{2} \, \big\vert \Phi_{t}^{\beta}(x) - \Upsilon_{t}^{\beta,m}(x) \big\vert^{2} \, \leqslant \, &(t-t_{0})^2 \cdot \vert v^{\beta}(t_{0},x) - v_{m}^{\beta}(t_{0},x)\vert^{2} \label{olwtwhteaoncmneb} \\
\, + \, &\Big\vert \big(\Phi_{t}^{\beta}(x) - x \big) - (t-t_{0}) \cdot v^{\beta}(t_{0},x) \Big\vert^{2}. \label{olwtwhteaoncmnea} 
\end{align}
\endgroup
Therefore, in order to establish \hyperref[olebaewtwhteaoncmne]{(\ref*{olebaewtwhteaoncmne})}, it suffices to show the limiting assertions \hyperref[liidofstsetwahfo]{(\ref*{liidofstsetwahfo})} and \hyperref[bolwtwhteaoncmne]{(\ref*{bolwtwhteaoncmne})} below; these correspond to \hyperref[olwtwhteaoncmneb]{(\ref*{olwtwhteaoncmneb})} and \hyperref[olwtwhteaoncmnea]{(\ref*{olwtwhteaoncmnea})}, respectively.

\smallskip

The first limiting identity
\begin{equation} \label{liidofstsetwahfo}
\lim_{m \rightarrow \infty} \, \mathds{E}_{\mathds{P}}\bigg[ \ \Big\vert v^{\beta}\big((t_{0},X(t_{0})\big) - v_{m}^{\beta}\big(t_{0},X(t_{0})\big) \Big\vert^{2} \ \bigg] = 0
\end{equation}
we already have from \hyperref[aotvfgosfvgosfwcs]{(\ref*{aotvfgosfvgosfwcs})}, \hyperref[lvfvbrtox]{(\ref*{lvfvbrtox})} of \textsf{Step 2}.

\medskip

\noindent \fbox{\textsf{Step 6.}} Our final task is to justify that
\begin{equation} \label{bolwtwhteaoncmne}
\lim_{t \downarrow t_{0}} \, \mathds{E}_{\mathds{P}}\bigg[ \ \Big\vert \tfrac{1}{t-t_{0}} \Big(\Phi_{t}^{\beta}\big(X(t_{0})\big) - X(t_{0}) \Big) - v^{\beta}\big(t_{0},X(t_{0})\big) \Big\vert^{2} \ \bigg] = 0.
\end{equation}
To this end, we first note that by \hyperref[cotcfl]{(\ref*{cotcfl})} we have the identity
\begin{equation} \label{iotcfitdfe}
\Phi_{t}^{\beta}(x) = x + \int_{t_{0}}^{t} v^{\beta}\big(u,\Phi_{u}^{\beta}(x)\big) \, \textnormal{d}u, \qquad x \in \mathds{R}^{n},
\end{equation}
for all $t \geqslant t_{0}$. On account of it we see that the expectation in \hyperref[bolwtwhteaoncmne]{(\ref*{bolwtwhteaoncmne})} is equal to
\begin{equation} 
\mathds{E}_{\mathds{P}}\Bigg[ \ \bigg\vert \frac{1}{t-t_{0}} \int_{t_{0}}^{t} v^{\beta}\Big(u,\Phi_{u}^{\beta}\big(X(t_{0})\big)\Big) \, \textnormal{d}u - v^{\beta}\big(t_{0},X(t_{0})\big) \bigg\vert^{2} \ \Bigg].
\end{equation}
As $\Phi_{t}^{\beta}$ transports $P^{\beta}(t_{0})$ to $P^{\beta}(t)$, and because the random variable $X(t_{0})$ has the same distribution under $\mathds{P}^{\beta}$ as it does under $\mathds{P}$, i.e., $P^{\beta}(t_{0})=P(t_{0})$, this expectation can also be expressed with respect to the probability measure $\mathds{P}^{\beta}$, and it thus suffices to show the limiting assertion
\begin{equation} \label{tsctznuiv}
\lim_{t \downarrow t_{0}} \, \mathds{E}_{\mathds{P}^{\beta}}\Bigg[ \ \bigg\vert \frac{1}{t-t_{0}} \int_{t_{0}}^{t} v^{\beta}\big(u,X(u)\big) \, \textnormal{d}u - v^{\beta}\big(t_{0},X(t_{0})\big) \bigg\vert^{2} \ \Bigg] = 0.
\end{equation}
For this purpose, we first observe that by the continuity of the paths of the canonical coordinate process $(X(t))_{t \geqslant 0}$, the family of random variables 
\begin{equation} \label{uivrvone}
\Bigg( \ \bigg\vert \frac{1}{t-t_{0}} \int_{t_{0}}^{t} v^{\beta}\big(u,X(u)\big) \, \textnormal{d}u - v^{\beta}\big(t_{0},X(t_{0})\big) \bigg\vert^{2} \ \Bigg)_{t \geqslant t_{0}}
\end{equation}
converges $\mathds{P}^{\beta}$-almost surely to zero, as $t \downarrow t_{0}$. In order to show that their expectations also converge to zero, i.e., that \hyperref[tsctznuiv]{(\ref*{tsctznuiv})} does hold, we have to verify that the family of \hyperref[uivrvone]{(\ref*{uivrvone})} is uniformly integrable with respect to $\mathds{P}^{\beta}$. As the random variable $\vert v^{\beta}(t_{0},X(t_{0})) \vert^{2}$ belongs to $L^{1}(\mathds{P}^{\beta})$, and we have
\begin{equation}
\bigg\vert \frac{1}{t-t_{0}} \int_{t_{0}}^{t} v^{\beta}\big(u,X(u)\big) \, \textnormal{d}u \bigg\vert^{2} \, \leqslant \,
\frac{1}{t-t_{0}} \int_{t_{0}}^{t} \big\vert v^{\beta}\big(u,X(u)\big) \big\vert^{2} \, \textnormal{d}u, \qquad t \geqslant t_{0}
\end{equation}
by Jensen's inequality, it is sufficient to prove the uniform integrability of the family 
\begin{equation} \label{eqouofuindsf}
\Bigg( \, \frac{1}{t-t_{0}} \int_{t_{0}}^{t} \big\vert v^{\beta}\big(u,X(u)\big) \big\vert^{2} \, \textnormal{d}u  \, \Bigg)_{t \geqslant t_{0}}.
\end{equation}
Invoking the definition of the time-dependent velocity field $\big(v^{\beta}(t, \, \cdot \,)\big)_{t \geqslant t_{0}}$ in \hyperref[pbtdvfvb]{(\ref*{pbtdvfvb})} and the fact that the perturbation $\beta$ is smooth and compactly supported, the uniform integrability of the family in \hyperref[eqouofuindsf]{(\ref*{eqouofuindsf})} above, is equivalent to the uniform integrability of the family
\begin{equation} \label{tfseqsbuitg}
\Bigg( \, \frac{1}{t-t_{0}} \int_{t_{0}}^{t} \frac{\big\vert \nabla \ell^{\beta}\big(u,X(u)\big) \big\vert^{2}}{\ell^{\beta}\big(u,X(u)\big)^{2}} \, \textnormal{d}u  \, \Bigg)_{t \geqslant t_{0}}.
\end{equation}
Now by continuity, the family of \hyperref[tfseqsbuitg]{(\ref*{tfseqsbuitg})} converges $\mathds{P}^{\beta}$-almost surely to $\vert \nabla \log \ell(t_{0},X(t_{0}))\vert^{2}$. Thus, to establish this uniform integrability, it suffices to show that the family of random variables in \hyperref[tfseqsbuitg]{(\ref*{tfseqsbuitg})} converges in $L^{1}(\mathds{P}^{\beta})$. Hence, in view of \textit{Scheff\'{e}'s lemma} (\hyperref[whaelsl]{Lemma \ref*{whaelsl}}), it remains to check that the corresponding expectations also converge. But at this point we use for the first time our choice of $t_{0} \in \mathds{R}_{+} \setminus N$ and recall \hyperref[tciwsolafophvs]{(\ref*{tciwsolafophvs})}, \hyperref[tcigwwsolafophvs]{(\ref*{tcigwwsolafophvs})} from the proof of \hyperref[thetsixcoraopv]{Corollary \ref*{thetsixcoraopv}}, which gives us 
\begin{equation}
\lim_{t \downarrow t_{0}}  \, \mathds{E}_{\mathds{P}^{\beta}}\Bigg[ \frac{1}{t-t_{0}} \int_{t_{0}}^{t} \frac{\big\vert \nabla \ell^{\beta}\big(u,X(u)\big) \big\vert^{2}}{\ell^{\beta}\big(u,X(u)\big)^{2}} \, \textnormal{d}u \Bigg] = 
\mathds{E}_{\mathds{P}}\Bigg[ \ \frac{\big\vert \nabla \ell\big(t_{0},X(t_{0})\big) \big\vert^{2}}{\ell\big(t_{0},X(t_{0})\big)^{2}} \ \Bigg],
\end{equation}
as required. This completes the proof of the claim made in the beginning of \textsf{Step 6}.

\smallskip

Summing up, in light of \hyperref[olwtwhteaoncmneb]{(\ref*{olwtwhteaoncmneb})}, \hyperref[olwtwhteaoncmnea]{(\ref*{olwtwhteaoncmnea})} from \textsf{Step 5}, the limiting assertions \hyperref[liidofstsetwahfo]{(\ref*{liidofstsetwahfo})} and \hyperref[bolwtwhteaoncmne]{(\ref*{bolwtwhteaoncmne})} imply the limiting behavior \hyperref[olebaewtwhteaoncmne]{(\ref*{olebaewtwhteaoncmne})}. According to the results of \textsf{Steps 4} and \textsf{5}, the latter also entails the validity of the limiting identity \hyperref[nvotlisvrlinfifatsvwct]{(\ref*{nvotlisvrlinfifatsvwct})}, which completes the proof of \hyperref[bvagswt]{Theorem \ref*{bvagswt}}.
\end{proof}

Equipped with \hyperref[bvagswt]{Theorem \ref*{bvagswt}}, we can now easily deduce \hyperref[agswt]{Theorem \ref*{agswt}}.

\begin{proof}[Proof of \texorpdfstring{\hyperref[agswt]{Theorem \ref*{agswt}}}{}] The second equality in \hyperref[agswtuff]{(\ref*{agswtuff})} follows from the representation of the relative Fisher information in \hyperref[merfi]{(\ref*{merfi})} and the definition of the time-dependent velocity field $\big(v(t, \, \cdot \,)\big)_{t \geqslant t_{0}}$ in \hyperref[tdvfvtx]{(\ref*{tdvfvtx})}. The first equality in \hyperref[agswtuff]{(\ref*{agswtuff})} is a direct consequence of \hyperref[bvagswt]{Theorem \ref*{bvagswt}}. One just has to set $\beta \equiv 0$ in \hyperref[pvosagswtuff]{(\ref*{pvosagswtuff})}. However, the careful reader might note that the limit in \hyperref[pvosagswtuff]{(\ref*{pvosagswtuff})} is only from the right, while the limit in \hyperref[agswtuff]{(\ref*{agswtuff})} is two-sided. But the only reason for considering right-sided limits in \hyperref[bvagswt]{Theorem \ref*{bvagswt}}, was the presence of the perturbation $\beta$ at time $t \geqslant t_{0}$. If there is no such perturbation, one can replace all limits from the right by two-sided ones. This completes the proof of \hyperref[agswt]{Theorem \ref*{agswt}}.
\end{proof}


\newpage 


\setkomafont{section}{\large}
\setkomafont{subsection}{\normalsize}

\begin{appendices}



\section{Bachelier's work relating Brownian motion to the heat equation} \label{ahnbwrbmtthe}


In this section, which is only of historical interest, we point out that Bachelier already had some thoughts on \textit{``horizontal transport of probability measures''} in his dissertation \textit{``Th{\'e}orie de la sp{\'e}culation''} \cite{Bac00,Bac06}, which he defended in 1900.

In this work he was the first to consider a \textit{mathematical model} of Brownian motion. Bachelier argued using infinitesimals by visualizing Brownian motion $(W(t))_{t \geqslant 0}$ as an infinitesimal version of a random walk. Suppose that the grid in space is given by 
\begin{equation} \label{grid}
\ldots, \ x_{n-2}, \ x_{n-1}, \ x_{n}, \ x_{n+1}, \ x_{n+2}, \ \ldots
\end{equation}
having the same (infinitesimal) distance $\Delta x = x_{n}-x_{n-1}$, for all $n$, and such that at time $t$ these points have (infinitesimal) probabilities 
\begin{equation} \label{gridp}
\ldots, \ p_{n-2}^{t}, \ p_{n-1}^{t}, \ p_{n}^{t}, \ p_{n+1}^{t}, \ p_{n+2}^{t}, \ \ldots
\end{equation}
under the random walk under consideration. What are the probabilities 
\begin{equation} \label{gridpp}
\ldots, \ p_{n-2}^{t+\Delta t}, \ p_{n-1}^{t+\Delta t}, \ p_{n}^{t+\Delta t}, \ p_{n+1}^{t+\Delta t}, \ p_{n+2}^{t+\Delta t}, \ \ldots
\end{equation}
of these points at time $t+\Delta t$? 

The random walk moves half of the mass $p_{n}^{t}$, sitting on $x_{n}$ at time $t$, to the point $x_{n+1}$. En revanche, it moves half of the mass $p_{n+1}^{t}$, sitting on $x_{n+1}$ at time $t$, to the point $x_{n}$. The net difference between $p_{n}^{t}/2$ and $p_{n+1}^{t}/2$, which Bachelier has no scruples to identify with 
\begin{equation}
-\!\tfrac{1}{2} \, (p^{t})'(x_{n}) \, \Delta x = - \tfrac{1}{2} \, (p^{t})'(x_{n+1}) \, \Delta x,
\end{equation}
is therefore transported from the interval $(-\infty,x_{n}]$ to $[x_{n+1},\infty)$. In Bachelier's own words, this is very nicely captured by the following passage of his thesis: 

\medskip

\textit{``Each price $x$ during an element of time radiates towards its neighboring price an amount of probability proportional to the difference of their probabilities. I say proportional because it is necessary to account for the relation of $\Delta x$ to $\Delta t$. The above law can, by analogy with certain physical theories, be called the law of radiation or diffusion of probability.''}

\medskip

Passing formally to the continuous limit and denoting by 
\begin{equation}
P(t,x) = \int_{-\infty}^{x} p(t,z) \, \textnormal{d}z
\end{equation}
the distribution function associated to the Gaussian density function $p(t,x)$, Bachelier deduces in an intuitively convincing way the relation
\begin{equation} \label{lbhef}
\frac{\partial P}{\partial t} = \frac{1}{2} \frac{\partial p}{\partial x},
\end{equation}
where we have normalized the relation between $\Delta x$ and $\Delta t$ to obtain the constant $1/2$. By differentiating \hyperref[lbhef]{(\ref*{lbhef})} with respect to $x$ one obtains the usual heat equation 
\begin{equation} \label{heflbhef}
\frac{\partial p}{\partial t} = \frac{1}{2} \frac{\partial^{2} p}{\partial x^{2}}
\end{equation}
for the density function $p(t,x)$. Of course, the heat equation was known to Bachelier, and he notes regarding \hyperref[heflbhef]{(\ref*{heflbhef})}: \textit{``C'est une {\'e}quation de Fourier.''}

But let us still remain with the form \hyperref[lbhef]{(\ref*{lbhef})} of the heat equation and analyze its message in terms of \textit{``horizontal transport of probability measures''}. To accomplish the movement of mass $-\frac{1}{2} \, p'(t,x) \, \textnormal{d}x$ from $(-\infty,x]$ to $[x,\infty)$ one is naturally led to define the flow induced by the velocity field
\begin{equation} \label{a.8}
v(t,x) \vcentcolon = - \frac{1}{2} \, \frac{p'(t,x)}{p(t,x)},
\end{equation}
which has the natural interpretation as the ``speed'' of the transport induced by $p(t,x)$. We thus encounter \textit{in nuce} the ubiquitous ``score function'' $\nabla p(t,x) / p(t,x)$ appearing throughout all the above considerations. We also note that an ``infinitesimal transport'' on $\mathds{R}$ is automatically an optimal transport. Intuitively this corresponds to the geometric insight in the one-dimensional case that the transport lines of infinitesimal length cannot cross each other.

\medskip

Let us go one step beyond Bachelier's thoughts and consider the relation of the above infinitesimal Wasserstein transport to time reversal (which Bachelier had not yet considered in his solitary exploration of Brownian motion). Visualizing again the grid \hyperref[grid]{(\ref*{grid})} and the corresponding probabilities \hyperref[gridp]{(\ref*{gridp})} and \hyperref[gridpp]{(\ref*{gridpp})}, a moment's reflection reveals that the transport from $p^{t+\Delta t}$ to $p^{t}$, i.e., in reverse direction, is accomplished by going from $x_{n}$ to $x_{n+1}$ with probability $\frac{1}{2} + \frac{p'(t,x)}{p(t,x)} \, \textnormal{d}x$ and from $x_{n+1}$ to $x_{n}$ with probability $\frac{1}{2} - \frac{p'(t,x)}{p(t,x)} \, \textnormal{d}x$, with the identifications $x = x_{n} = x_{n+1}$, and $\textnormal{d}x = \Delta x$. In other words, the above Brownian motion $(W(t))_{t \geqslant 0}$ considered in reverse direction $(W(T-s))_{0 \leqslant s \leqslant T}$ is \textit{not} a Brownian motion, as the transition probabilities are not $(1/2,1/2)$ any more. Rather, one has to correct these probabilities by a term which --- once again --- involves our familiar score function $\nabla p(t,x) / p(t,x)$ (compare \hyperref[dotowptmtpbm]{(\ref*{dotowptmtpbm})} above). At this stage, it should come as no surprise, that the passage to reverse time is closely related to the Wasserstein transport induced by $p(t,x)$. 

\medskip

Let us play this infinitesimal reasoning one more time, in order to visualize the Fontbona-Jourdain result (\hyperref[ovtfjofmidpwcd]{Theorem \ref*{ovtfjofmidpwcd}}). Arguing in the reverse direction of time, we may ask the following question: how do we have to choose the transition probabilities to go from $x$ at time $t + \Delta t$ to either $x + \textnormal{d}x$ or $x - \textnormal{d}x$ at time $t$, so that the density process $p(t,x)$ becomes a martingale in reverse time under these transition probabilities? As the difference between the probabilities $p(t,x+ \textnormal{d}x)$ and $p(t,x- \textnormal{d}x)$ equals $2 \, p'(t,x) \, \textnormal{d}x$ (up to terms of smaller order than $\textnormal{d}x$) we conclude that the transition probabilities have to be changed from $(1/2,1/2)$ to
\begin{equation}
\bigg(\frac{1}{2} - \frac{p'(t,x)}{p(t,x)} \, \textnormal{d}x \, , \, \frac{1}{2} + \frac{p'(t,x)}{p(t,x)} \, \textnormal{d}x\bigg)
\end{equation}
in order to counterbalance this difference of probabilities (again up to terms of smaller order than $\textnormal{d}x$). In other words, we have found again precisely the same transition probabilities which we had encountered in the context of the reversed Brownian process $(W(T-s))_{0 \leqslant s \leqslant T}$. This provides some intuition for the Fontbona-Jourdain assertion that $(p(T-s,W(T-s))_{0 \leqslant s \leqslant T}$ is a martingale in the reverse direction of time.

\medskip

We finish the section by returning to Bachelier's thesis. The \textit{rapporteur} of Bachelier's dissertation was no lesser a figure than Henri Poincar{\'e}. Apparently he was aware of the enormous potential of the section \textit{``Rayonnement de la probabilit{\'e}''} in Bachelier's thesis, when he added to his very positive report the handwritten phrase: \textit{``On peut regretter que M. Bachelier n'ait pas d{\'e}velopp{\'e} davantage cette partie de sa th{\`e}se.''} That is: One might regret that Mr. Bachelier did not develop further this part of his thesis.


\section{The proofs of \texorpdfstring{\hyperref[fosmolds]{Lemmas \ref*{fosmolds}}}{Lemmas 2.1} and \texorpdfstring{\hyperref[fosmoldspv]{\ref*{fosmoldspv}}}{3.3}} \label{fosmoldspola}


\begin{proof}[Proof of \texorpdfstring{\hyperref[fosmolds]{Lemma \ref*{fosmolds}}}] Let the real constants $c, R \geqslant 0$ be as in condition \hyperref[naltsaosaojko]{\ref*{naltsaosaojko}} of \hyperref[sosaojkoia]{Assumptions \ref*{sosaojkoia}}, and denote
\begin{equation} \label{tfewtgotpa}
m_{R} \vcentcolon = \max_{\vert x \vert \leqslant R} \, \big\vert n -  \big\langle x \, , \, 2 \, \nabla \Psi(x) \big\rangle_{\mathds{R}^{n}} \big\vert < \infty \, , \qquad \quad
\tau_{k} \vcentcolon = \inf \big\{ t \geqslant 0 \colon \, \vert X(t) \vert > k \big\}
\end{equation}
for integers $k > R$. It\^{o}'s formula gives
\begin{equation} \label{poltpofe}
\textnormal{d} \vert X(t) \vert^{2} = \Big(n -  \Big\langle X(t) \, , \, 2 \, \nabla \Psi\big(X(t)\big) \Big\rangle_{\mathds{R}^{n}} \Big) \, \textnormal{d}t + \big\langle 2 \, X(t) \, , \, \textnormal{d}W(t) \big\rangle_{\mathds{R}^{n}}
\end{equation}
for $t \geqslant 0$. We define $\varphi_{k}(t) \vcentcolon = \mathds{E}_{\mathds{P}}\big[ \, \vert X(t \wedge \tau_{k}) \vert^{2} \, \big]$ and $\varphi(t) \vcentcolon = \mathds{E}_{\mathds{P}}\big[ \, \vert X(t) \vert^{2} \, \big]$. Taking expectations in \hyperref[poltpofe]{(\ref*{poltpofe})} yields
\begingroup
\addtolength{\jot}{0.7em}
\begin{alignat}{2}
\varphi_{k}(t) &= \varphi(0)
&&+ \mathds{E}_{\mathds{P}} \bigg[ \int_{0}^{t \wedge \tau_{k}} \Big(n - \Big\langle X(u) \, , \, 2 \, \nabla \Psi\big(X(u)\big) \Big\rangle_{\mathds{R}^{n}} \, \Big) \, \mathds{1}_{\{\vert X(u) \vert \leqslant R\}} \, \textnormal{d}u \bigg] \\
& &&+ \mathds{E}_{\mathds{P}} \bigg[ \int_{0}^{t \wedge \tau_{k}} \Big(n -  \Big\langle X(u) \, , \, 2 \, \nabla \Psi\big(X(u)\big) \Big\rangle_{\mathds{R}^{n}} \, \Big) \, \mathds{1}_{\{\vert X(u) \vert > R\}} \, \textnormal{d}u \bigg] \label{tltwtgotpa} \\
&\leqslant \varphi(0) &&+ m_{R} \, \mathds{E}_{\mathds{P}}[t \wedge \tau_{k}] + \mathds{E}_{\mathds{P}} \bigg[ \int_{0}^{t \wedge \tau_{k}} \big(n + 2 \, c \, \vert X(u) \vert^{2} \, \big) \, \textnormal{d}u \bigg] \\
&\leqslant \varphi(0) &&+ (m_{R} + n) \, t + 2 \, c \int_{0}^{t} \varphi_{k}(u) \, \textnormal{d}u.
\end{alignat}
\endgroup
The Gronwall inequality gives now
\begin{equation} \label{aptgigntf}
\varphi_{k}(t) \leqslant g(t) \vcentcolon = \varphi(0) + (m_{R} + n) \, t + 2 \, c \int_{0}^{t} \big( \varphi(0) + (m_{R} + n) \, u \big) \, \mathrm{e}^{2c(t-u)} \, \textnormal{d}u.
\end{equation}
According to the second-moment condition in \hyperref[ffecaoo]{(\ref*{ffecaoo})}, the quantity $g(t)$ is finite for all $t \geqslant 0$, and independent of $k$; letting $k \uparrow \infty$ in \hyperref[aptgigntf]{(\ref*{aptgigntf})}, we get 
\begin{equation}
\varphi(t) = \mathds{E}_{\mathds{P}}\big[ \, \vert X(t) \vert^{2} \, \big] \leqslant g(t) < \infty, \qquad t \geqslant 0.
\end{equation}
In other words, we have that $P(t) \in \mathscr{P}_{2}(\mathds{R}^{n})$ for all $t \geqslant 0$. 
\end{proof}

\begin{proof}[Proof of \texorpdfstring{\hyperref[fosmoldspv]{Lemma \ref*{fosmoldspv}}}] The proof of \hyperref[fosmoldspv]{Lemma \ref*{fosmoldspv}} follows by analogy with the proof of \hyperref[fosmolds]{Lemma \ref*{fosmolds}} above. Indeed, one just has to add the perturbation $\beta$ to the gradient $\nabla \Psi$ in the expressions \hyperref[tfewtgotpa]{(\ref*{tfewtgotpa})} -- \hyperref[tltwtgotpa]{(\ref*{tltwtgotpa})}, write $W^{\beta}(t)$ instead of $W(t)$ in \hyperref[poltpofe]{(\ref*{poltpofe})}, and replace all the $\mathds{P}$-expectations by expectations with respect to the probability measure $\mathds{P}^{\beta}$. Then the constant $m_{R}$ in \hyperref[tfewtgotpa]{(\ref*{tfewtgotpa})} is still finite and, because of its compact support, the perturbation $\beta$ in the expression \hyperref[tltwtgotpa]{(\ref*{tltwtgotpa})} vanishes, provided $R$ is chosen large enough. Hence, by the same token as above, we conclude that $P^{\beta}(t) \in \mathscr{P}_{2}(\mathds{R}^{n})$ for all $t \geqslant t_{0}$.  
\end{proof}


\section{Relative entropy with respect to a \texorpdfstring{$\sigma$}{sigma}-finite measure \texorpdfstring{$\mathrm{Q}$}{Q}} \label{wdadotrewrttmq}


For two probability measures $\mathcal{P}$ and $\mathcal{Q}$ on $\mathscr{B}(\mathds{R}^{n})$, the \textit{relative entropy} of $\mathcal{P}$ with respect to $\mathcal{Q}$ is defined as
\begin{equation} \label{doreokldbtpm}
H(\mathcal{P} \, \vert \, \mathcal{Q}) \vcentcolon = \int_{\mathds{R}^{n}} \log \bigg( \frac{\textnormal{d}\mathcal{P}}{\textnormal{d}\mathcal{Q}} \bigg) \, \textnormal{d}\mathcal{P} \in [0,\infty]
\end{equation}
if $\mathcal{P}$ is absolutely continuous with respect to $\mathcal{Q}$, and as $H(\mathcal{P} \, \vert \, \mathcal{Q}) \vcentcolon = \infty$ if this is not the case. 

\smallskip

Let us consider the $\sigma$-finite measure $\mathrm{Q}$ on $\mathscr{B}(\mathds{R}^{n})$ with density $\mathds{R}^{n} \ni x \mapsto q(x) = \mathrm{e}^{-2 \Psi(x)}$, introduced in \hyperref[snaas]{Section \ref*{snaas}}. Following the approach of \cite[Section 2]{Leo14}, we shall demonstrate that the same definition of relative entropy $H( P \, \vert \, \mathrm{Q})$ applies to the reference measure $\mathrm{Q}$, provided that the probability measure $P$ is an element of $\mathscr{P}_{2}(\mathds{R}^{n})$ --- with the only difference that the quantity \hyperref[doreokldbtpm]{(\ref*{doreokldbtpm})} now takes values in $(-\infty,\infty]$.

To this end, we let $P$ be a probability measure in the quadratic Wasserstein space $\mathscr{P}_{2}(\mathds{R}^{n})$. The non-negativity of the potential $\Psi$ implies that for the function $\mathds{R}^{n} \ni x \mapsto f(x) \vcentcolon = \mathrm{e}^{-\vert x \vert^{2}}$ we have $\mathds{E}_{\mathrm{Q}}[f] \in (0,\infty)$. Following \cite[Section 2]{Leo14}, we let $\mathcal{Q}$ be the probability measure on $\mathscr{B}(\mathds{R}^{n})$ having probability density function $f / \mathds{E}_{\mathrm{Q}}[f]$ with respect to the measure $\mathrm{Q}$, so that
\begin{equation} 
\frac{\textnormal{d}P}{\textnormal{d}\mathrm{Q}} = \frac{f}{\mathds{E}_{\mathrm{Q}}[f]} \frac{\textnormal{d}P}{\textnormal{d}\mathcal{Q}}.
\end{equation}
Taking first logarithms and then expectations with respect to $P$ on both sides of this equation yields the formula
\begin{equation} \label{fftrewrtapum}
H(P \, \vert \, \mathrm{Q}) = H(P \, \vert \, \mathcal{Q}) - \int_{\mathds{R}^{n}} \vert x \vert^{2} \, \textnormal{d}P(x) - \log \bigg( \int_{\mathds{R}^{n}} \mathrm{e}^{- \vert x \vert^{2} - 2 \Psi(x)} \, \textnormal{d}x \bigg),
\end{equation}
which is justified by \hyperref[doreokldbtpm]{(\ref*{doreokldbtpm})} and the fact that $P \in \mathscr{P}_{2}(\mathds{R}^{n})$ as well as $\mathds{E}_{\mathrm{Q}}[f] \in (0,\infty)$. In particular, we see that the right-hand side of \hyperref[fftrewrtapum]{(\ref*{fftrewrtapum})} takes values in the interval $(-\infty,\infty]$. Summing up, we can define well the relative entropy $H(P \, \vert \, \mathrm{Q})$ as in \hyperref[doreokldbtpm]{(\ref*{doreokldbtpm})} provided that $P \in \mathscr{P}_{2}(\mathds{R}^{n})$, even when the $\sigma$-finite measure $\mathrm{Q}$ has infinite total mass.

\begin{remark} Wherever in this paper the relative entropy $H(P \, \vert \, \mathrm{Q})$ is considered for some $\sigma$-finite measure $\mathrm{Q}$ on $\mathscr{B}(\mathds{R}^{n})$ with density $\mathds{R}^{n} \ni x \mapsto q(x) = \mathrm{e}^{-2 \Psi(x)}$, the probability measure $P$ will always be assumed to belong to $\mathscr{P}_{2}(\mathds{R}^{n})$. This is in accordance with \hyperref[fosmolds]{Lemmas \ref*{fosmolds}} and \hyperref[fosmoldspv]{\ref*{fosmoldspv}}, as well as \hyperref[hlotl]{Lemma \ref*{hlotl}}. In the latter, the constant-speed geodesic $(P_{t})_{0 \leqslant t \leqslant 1}$ joining two probability measures $P_{0}$ and $P_{1}$ in $\mathscr{P}_{2}(\mathds{R}^{n})$ is considered. Therefore, in all situations relevant to us, the relative entropy $H(P \, \vert \, \mathrm{Q})$ is well-defined and takes values in the interval $(-\infty,\infty]$.
\end{remark}


\section{A measure-theoretic result} \label{apsecamtr} 


In the proofs of \hyperref[thetsixcor]{Propositions \ref*{thetsixcor}} and \hyperref[thetthretvcor]{\ref*{thetthretvcor}} we have used a result about conditional expectations, which we will formulate and prove below. We place ourselves on a probability space $(\Omega,\mathcal{F},\mathds{P})$ endowed with a left-continuous filtration $(\mathcal{F}(t))_{t \geqslant 0}$. We first state the following result, which is known as \textit{Scheff\'{e}'s lemma} \cite[5.10]{Wil91}.

\begin{lemma}[\textsf{Scheff\'{e}'s lemma}] \label{whaelsl} For a sequence of integrable random variables $(X_{n})_{n \in \mathds{N}}$ which converges almost surely to another integrable random variable $X$, convergence of the $L^{1}(\mathds{P})$-norms \textnormal{(}i.e., $\lim_{n \rightarrow \infty} \mathds{E}[\vert X_{n} \vert] = \mathds{E}[\vert X \vert ]$\textnormal{)} is equivalent to convergence in $L^{1}(\mathds{P})$ \textnormal{(}i.e., $\lim_{n \rightarrow \infty} \mathds{E}[\vert X_{n} - X\vert] = 0$\textnormal{)}. 
\end{lemma}

\begin{proposition} \label{probamtr} Let $(B(t))_{0 \leqslant t \leqslant T}$ and $(C(t))_{0 \leqslant t \leqslant T}$ be adapted continuous processes, which are non-negative and uniformly bounded, respectively. Define the process $(A(t))_{0 \leqslant t \leqslant T}$ as their primitive, i.e.,
\begin{equation}
A(t) = \int_{0}^{t} \big( B(u) + C(u) \big) \, \textnormal{d}u, \qquad 0 \leqslant t \leqslant T
\end{equation}
and assume that $\mathds{E}\big[\int_{0}^{T}  B(u)  \, \textnormal{d}u\big]$ is finite. By the Lebesgue differentiation theorem, for Lebesgue-almost every $t_{0} \in [0,T]$, we have
\begin{equation} \label{titlocotn}
\lim_{t \rightarrow t_{0}} \, \mathds{E}\Bigg[ \frac{A(t)-A(t_{0})}{t-t_{0}}\Bigg] = \lim_{t \rightarrow t_{0}} \, \mathds{E}\Bigg[ \frac{1}{t-t_{0}}\int_{t_{0}}^{t} \big( B(u) + C(u) \big) \, \textnormal{d}u \Bigg]  = \mathds{E}\big[ B(t_{0}) + C(t_{0})\big].
\end{equation}
Now fix a ``Lebesgue point'' $t_{0} \in [0,T]$ for which \textnormal{\hyperref[titlocotn]{(\ref*{titlocotn})}} does hold. Then we have the analogous limiting assertion for the conditional expectations, i.e.,
\begin{equation} \label{probamtre}
\lim_{t \uparrow t_{0}} \, \frac{\mathds{E}\big[A(t_{0})-A(t) \ \vert \ \mathcal{F}(t)\big]}{t_{0}-t} = 
\lim_{t \downarrow t_{0}} \, \frac{\mathds{E}\big[A(t)-A(t_{0}) \ \vert \ \mathcal{F}(t_{0})\big]}{t-t_{0}} =
B(t_{0}) + C(t_{0}),
\end{equation}
where the limits exist in $L^{1}(\mathds{P})$.
\begin{proof} Using the uniform boundedness of the process $(C(t))_{0 \leqslant t \leqslant T}$, it is easy to see that the existence of the limit in \hyperref[titlocotn]{(\ref*{titlocotn})} only depends on the process $(B(t))_{0 \leqslant t \leqslant T}$. Therefore we can assume without loss of generality that $C(t) \equiv 0$ for all $0 \leqslant t \leqslant T$.

\smallskip

Fix a Lebesgue point $t_{0} \in [0,T]$ for which \hyperref[titlocotn]{(\ref*{titlocotn})} does hold. As the process $(B(t))_{0 \leqslant t \leqslant T}$ is continuous, the fundamental theorem of calculus ensures that the limit 
\begin{equation} \label{tlsfftftoc}
\lim_{t \rightarrow t_{0}} \, \frac{A(t) - A(t_{0})}{t-t_{0}} = \lim_{t \rightarrow t_{0}} \, \frac{1}{t-t_{0}} \int_{t_{0}}^{t} B(u) \, \textnormal{d}u = B(t_{0})
\end{equation}
exists almost surely. Since the random variables appearing in \hyperref[tlsfftftoc]{(\ref*{tlsfftftoc})} are integrable, and we already have the convergence of the $L^{1}(\mathds{P})$-norms from \hyperref[titlocotn]{(\ref*{titlocotn})}, \hyperref[whaelsl]{Lemma \ref*{whaelsl}} allows us to conclude that the convergence of \hyperref[tlsfftftoc]{(\ref*{tlsfftftoc})} holds also in $L^{1}(\mathds{P})$, i.e.,
\begin{equation} \label{titlocotncitn}
\lim_{t \rightarrow t_{0}} \, \bigg\| \frac{A(t)-A(t_{0})}{t-t_{0}} - B(t_{0}) \bigg\|_{L^{1}(\mathds{P})} = 0.
\end{equation}
From \hyperref[titlocotncitn]{(\ref*{titlocotncitn})} we can deduce now the $L^{1}(\mathds{P})$-convergence of \hyperref[probamtre]{(\ref*{probamtre})} as follows. Regarding the second limit in \hyperref[probamtre]{(\ref*{probamtre})}, for $t > t_{0}$, we find
\begingroup
\addtolength{\jot}{0.7em}
\begin{align}
\Bigg\| \frac{\mathds{E}\big[A(t)-A(t_{0}) \ \vert \ \mathcal{F}(t_{0})\big]}{t-t_{0}} - B(t_{0})\Bigg\|_{L^{1}(\mathds{P})}
&= \Bigg\| \, \mathds{E}\bigg[ \frac{A(t)-A(t_{0})}{t-t_{0}} - B(t_{0}) \ \Big\vert \ \mathcal{F}(t_{0}) \bigg] \, \Bigg\|_{L^{1}(\mathds{P})} \\
&\leqslant \bigg\| \frac{A(t)-A(t_{0})}{t-t_{0}} - B(t_{0}) \bigg\|_{L^{1}(\mathds{P})} \, ; \label{tewctz}
\end{align}
\endgroup
and according to \hyperref[titlocotncitn]{(\ref*{titlocotncitn})}, the expression in \hyperref[tewctz]{(\ref*{tewctz})} converges to zero as $t \downarrow t_{0}$. Similarly, to handle the first limit in \hyperref[probamtre]{(\ref*{probamtre})}, we use for $t < t_{0}$ the estimate
\begingroup
\addtolength{\jot}{0.7em}
\begin{align}
\Bigg\| \frac{\mathds{E}\big[A(t_{0})-A(t) \ \vert \ \mathcal{F}(t)\big]}{t_{0}-t} - B(t_{0})\Bigg\|_{L^{1}(\mathds{P})}
&\leqslant \bigg\| \frac{A(t_{0})-A(t)}{t_{0}-t} - B(t_{0}) \bigg\|_{L^{1}(\mathds{P})}  \label{tewctzrhsa} \\
 & \qquad + \big\| \mathds{E}\big[  B(t_{0}) \, \vert \, \mathcal{F}(t) \big]  - B(t_{0})  \big\|_{L^{1}(\mathds{P})}. \label{tewctzrhssb}
\end{align}
\endgroup
As $t \uparrow t_{0}$, the expression on the right-hand side of \hyperref[tewctzrhsa]{(\ref*{tewctzrhsa})} converges to zero as before. The same is true also for the term in \hyperref[tewctzrhssb]{(\ref*{tewctzrhssb})}, on account of \cite[Theorem 9.4.8]{Chu01} and the left-continuity of the filtration $(\mathcal{F}(t))_{t \geqslant 0}$. This completes the proof of \hyperref[probamtr]{Proposition \ref*{probamtr}}.
\end{proof}
\end{proposition}

\section{The proof of the Fontbona-Jourdain theorem} \label{atfjo}


\begin{proof}[Proof of \texorpdfstring{\hyperref[ovtfjofmidpwcd]{Theorem \ref*{ovtfjofmidpwcd}}}{} \textnormal{\cite{FJ16}}] For $0 \leqslant s \leqslant T$, we define the random variable $N(T-s)$ as the conditional expectation of the random variable 
\begin{equation}
\ell\big(0,X(0)\big) = \frac{p\big(0,X(0)\big)}{q\big(X(0)\big)} \in L^{1}(\mathds{Q})
\end{equation}
with respect to the backwards filtration $(\mathcal{G}(T-s))_{0 \leqslant s \leqslant T}$, i.e.,
\begin{equation}
N(T-s) \vcentcolon = \mathds{E}_{\mathds{Q}}\Big[ \ell\big(0,X(0)\big) \ \big\vert \ \mathcal{G}(T-s) \Big], \qquad 0 \leqslant s \leqslant T. 
\end{equation}
Obviously the process $(N(T-s))_{0 \leqslant s \leqslant T}$ is a martingale of the backwards filtration $(\mathcal{G}(T-s))_{0 \leqslant s \leqslant T}$ under the probability measure $\mathds{Q}$. Now we make the following elementary, but crucial, observation: as the stochastic process $(X(t))_{0 \leqslant t \leqslant T}$, which solves the stochastic differential equation \hyperref[sdeids]{(\ref*{sdeids})}, is a \textit{Markov process}, the time-reversed process $(X(T-s))_{0 \leqslant s \leqslant T}$ is a Markov process, too, under the probability measure $\mathds{P}$ as well as under $\mathds{Q}$. Hence
\begin{equation}
N(T-s) = \mathds{E}_{\mathds{Q}}\Big[ \ell\big(0,X(0)\big) \ \big\vert \ X(T-s) \Big], \qquad 0 \leqslant s \leqslant T. 
\end{equation}
We have to show that this last conditional expectation equals $\ell(T-s,X(T-s))$. To this end, we fix $s \in [0,T]$ as well as a Borel set $A \subseteq \mathds{R}^{n}$, and denote by $\uppi(T-s;x,A)$ the transition probability of the event $\{X(T-s) \in A\}$, conditionally on $X(0) = x$. Note that this transition probability does not depend on whether we consider the process $(X(t))_{0 \leqslant t \leqslant T}$ under $\mathds{P}$ or under $\mathds{Q}$. Then we find 
\begin{equation} \label{otfjtcocef}
\mathds{E}_{\mathds{Q}}\Bigg[ \frac{p\big(0,X(0)\big)}{q\big(X(0)\big)} \ \mathds{1}_{A}\big(X(T-s)\big) \Bigg] 
= \int_{\mathds{R}^{n}} \frac{p(0,x)}{q(x)} \, \uppi(T-s;x,A) \, q(x) \, \textnormal{d}x 
= P(T-s)[A].
\end{equation}
Note also that 
\begin{equation} \label{otfjtcoces}
\mathds{E}_{\mathds{Q}}\Bigg[ \frac{p\big(T-s,X(T-s)\big)}{q\big(X(T-s)\big)} \ \mathds{1}_{A}\big(X(T-s)\big) \Bigg]
= P(T-s)[A].
\end{equation}
Because the Borel set $A \subseteq \mathds{R}^{n}$ is arbitrary, we deduce from \hyperref[otfjtcocef]{(\ref*{otfjtcocef})} and \hyperref[otfjtcoces]{(\ref*{otfjtcoces})} that
\begin{equation}
\mathds{E}_{\mathds{Q}}\Bigg[ \frac{p\big(0,X(0)\big)}{q\big(X(0)\big)} \ \bigg \vert \ X(T-s) \Bigg] = \frac{p\big(T-s,X(T-s)\big)}{q\big(X(T-s)\big)}
= \ell\big(T-s,X(T-s)\big).
\end{equation}
This completes the proof of \textnormal{\hyperref[ovtfjofmidpwcd]{Theorem \ref*{ovtfjofmidpwcd}}}.
\end{proof}



\section{The proof of \texorpdfstring{\hyperref[hlotl]{Lemma \ref*{hlotl}}}{Lemma 3.19}} \label{polhlotl}


\begin{proof}[Proof of \texorpdfstring{\hyperref[hlotl]{Lemma \ref*{hlotl}}}{}] In order to show \hyperref[hlotlte]{(\ref*{hlotlte})}, we recall the notation of \hyperref[hlotlse]{(\ref*{hlotlse})} and consider the time-dependent velocity field 
\begin{equation}
[0,1] \times \mathds{R}^{n} \ni (t,\xi) \longmapsto v_{t}(\xi) \vcentcolon = \gamma\Big( \big(T_{t}^{\gamma}\big)^{-1}(\xi)\Big) \in \mathds{R}^{n},
\end{equation}
which is well-defined $P_{t}$-almost everywhere, for every $t \in [0,1]$. Then $(v_{t})_{0 \leqslant t \leqslant 1}$ is the velocity field associated with $(T_{t}^{\gamma})_{0 \leqslant t \leqslant 1}$, i.e.,
\begin{equation}
T_{t}^{\gamma}(x) = x + \int_{0}^{t} v_{\theta}\big(T_{\theta}^{\gamma}(x)\big) \, \textnormal{d}\theta,
\end{equation}
on account of \hyperref[hlotlse]{(\ref*{hlotlse})}. Let $p_{t}(\, \cdot \,)$ be the probability density function of the probability measure $P_{t}$ in \hyperref[hlotlse]{(\ref*{hlotlse})}. Then, according to \cite[Theorem 5.34]{Vil03}, the function $p_{t}(\, \cdot \,)$ satisfies the continuity equation
\begin{equation} 
\partial_{t} p_{t}(x) + \operatorname{div} \big( v_{t}(x) \, p_{t}(x) \big) = 0, \qquad (t,x) \in (0,1) \times \mathds{R}^{n},
\end{equation}
which can be written equivalently as 
\begin{equation} \label{ceiikyjl}
- \partial_{t} p_{t}(x) = \operatorname{div} \big( v_{t}(x) \big) \, p_{t}(x) + \big\langle v_{t}(x) \, , \nabla p_{t}(x) \big\rangle_{\mathds{R}^{n}}, \qquad (t,x) \in (0,1) \times \mathds{R}^{n}.
\end{equation}
Recall that $X_{0}$ is a random variable with probability distribution $P_{0}$ on the probability space $(S,\mathcal{S},\nu)$. Then the integral equation
\begin{equation} \label{ieahwv}
X_{t} = X_{0} + \int_{0}^{t} v_{\theta}(X_{\theta}) \, \textnormal{d}\theta, \qquad 0 \leqslant t \leqslant 1
\end{equation}
defines random variables $X_{t}$ with probability distributions $P_{t} = (T_{t}^{\gamma})_{\#}(P_{0})$ for $t \in [0,1]$, as in \hyperref[hlotlse]{(\ref*{hlotlse})}. We have now
\begin{equation} 
\textnormal{d} p_{t}(X_{t}) 
= \partial_{t} p_{t}(X_{t}) \, \textnormal{d}t + \big\langle \nabla p_{t}(X_{t}) \, , \, \textnormal{d}X_{t} \big\rangle_{\mathds{R}^{n}}
= - p_{t}(X_{t}) \operatorname{div} \big( v_{t}(X_{t}) \big) \, \textnormal{d}t
\end{equation}
on account of \hyperref[ceiikyjl]{(\ref*{ceiikyjl})}, \hyperref[ieahwv]{(\ref*{ieahwv})}, thus also
\begin{equation} \label{dlogpd}
\textnormal{d} \log p_{t}(X_{t}) = - \operatorname{div} \big( v_{t}(X_{t}) \big) \, \textnormal{d}t, \qquad 0 \leqslant t \leqslant 1.
\end{equation}
Recall now the function $q(x) = \mathrm{e}^{ - 2 \Psi(x)}$, for which
\begin{equation} \label{dlogqd}
\textnormal{d} \log q( X_{t} ) 
= -  \big\langle 2 \, \nabla \Psi(X_{t}) \, , \, \textnormal{d}X_{t} \big\rangle_{\mathds{R}^{n}}  
= -  \big\langle 2 \, \nabla \Psi(X_{t}) \, , \, v_{t}(X_{t}) \big\rangle_{\mathds{R}^{n}} \, \textnormal{d}t.
\end{equation}
For the likelihood ratio function $\ell_{t}(\, \cdot \,)$ of \hyperref[nlrfitramfcc]{(\ref*{nlrfitramfcc})} we get from \hyperref[dlogpd]{(\ref*{dlogpd})} and \hyperref[dlogqd]{(\ref*{dlogqd})} that
\begin{equation} \label{hwwtte}
\textnormal{d} \log \ell_{t}(X_{t}) =  \big\langle 2 \, \nabla \Psi(X_{t}) \, , \, v_{t}(X_{t}) \big\rangle_{\mathds{R}^{n}} \, \textnormal{d}t \, - \, \operatorname{div} \big( v_{t}(X_{t}) \big) \, \textnormal{d}t, \qquad 0 \leqslant t \leqslant 1.
\end{equation}
Taking expectations in the integral version of \hyperref[hwwtte]{(\ref*{hwwtte})}, we obtain that the difference 
\begin{equation}
H(P_{t} \, \vert \, \mathrm{Q})  - H(P_{0} \, \vert \, \mathrm{Q}) 
= \mathds{E}_{\nu}\big[ \log \ell_{t}(X_{t})\big] - \mathds{E}_{\nu}\big[ \log \ell_{0}(X_{0}) \big]
\end{equation}
is equal to
\begin{equation}
\mathds{E}_{\nu}\bigg[\int_{0}^{t} \Big(  \big\langle 2 \, \nabla \Psi(X_{\theta}) \, , \, v_{\theta}(X_{\theta}) \big\rangle_{\mathds{R}^{n}} - \operatorname{div} \big( v_{\theta}(X_{\theta})\big) \Big) \, \textnormal{d}\theta \bigg]
\end{equation}
for $t \in [0,1]$. Consequently,
\begin{equation} \label{liadn}
\lim_{t \downarrow 0} \frac{H(P_{t} \, \vert \, \mathrm{Q}) - H(P_{0} \, \vert \, \mathrm{Q})}{t} 
= \mathds{E}_{\nu}\Big[    \big\langle 2 \, \nabla \Psi(X_{0}) \, , \, v_{0}(X_{0}) \big\rangle_{\mathds{R}^{n}} - \operatorname{div} \big(v_{0}(X_{0})\big) \Big].
\end{equation}
Integrating by parts, we see that 
\begingroup
\addtolength{\jot}{0.7em}
\begin{align}
\mathds{E}_{\nu}\big[\operatorname{div} \big(v_{0}(X_{0})\big)\big] 
&= \int_{\mathds{R}^{n}} \operatorname{div} \big(v_{0}(x)\big) \, p_{0}(x) \, \textnormal{d}x = - \int_{\mathds{R}^{n}} \big\langle v_{0}(x) \, , \nabla p_{0}(x) \big\rangle_{\mathds{R}^{n}} \, \textnormal{d}x  \label{ibppo} \\
&= - \big\langle \nabla \log p_{0}(X_{0}) \, , \, v_{0}(X_{0}) \big\rangle_{L^{2}(\nu)}. \label{ibppopt}
\end{align}
\endgroup
Recalling \hyperref[liadn]{(\ref*{liadn})}, and combining it with the relation $\nabla \log \ell_{t}(x) = \nabla \log p_{t}(x) + 2 \, \nabla \Psi(x)$, as well as with \hyperref[ibppo]{(\ref*{ibppo})} and \hyperref[ibppopt]{(\ref*{ibppopt})}, we get
\begin{equation} 
\lim_{t \downarrow 0} \frac{H(P_{t} \, \vert \, \mathrm{Q}) - H(P_{0} \, \vert \, \mathrm{Q})}{t} 
= \big\langle \nabla \log \ell_{0}(X_{0}) \, , \, v_{0}(X_{0}) \big\rangle_{L^{2}(\nu)}.
\end{equation}
Since $v_{0} = \gamma$, this leads to \hyperref[hlotlte]{(\ref*{hlotlte})}. 
\end{proof}


\section{Time reversal of diffusions} \label{atrod}


We review in the present section the theory of time reversal for diffusion processes developed by F\"ollmer \cite{Foe85,Foe86}, Haussmann and Pardoux \cite{HP86}, and Pardoux \cite{Par86}. This section can be read independently of the rest of the paper; it does not present novel results.


\subsection{Introduction}


It is very well known that the Markov property is invariant under time reversal. In other words, a Markov process remains a Markov process under time reversal (e.g., \cite[Exercise E60.41, p.\ 162]{RW00a}). On the other hand, it is also well known that the strong Markov property is not necessarily preserved under time reversal (e.g., \cite[p.\ 330]{RW00a}), and neither is the semimartingale property (e.g., \cite{Wal82}). The reason for such failure is the same in both cases: after reversing time, ``we may know too much''. Thus, the following questions arise rather naturally:

\smallskip
 
\textit{Given a diffusion process \textnormal{(}in particular, a strong Markov semimartingale with continuous paths\textnormal{)} $X = (X(t))_{0 \leqslant t \leqslant T}$ with certain specific drift and dispersion characteristics, under what conditions might the time-reversed process 
\begin{equation} \label{1}
\widehat{X}(s) \vcentcolon = X(T-s), \qquad 0 \leqslant s \leqslant T,
\end{equation}
also be a diffusion? if it happens to be, what are the characteristics of the time-reversed diffusion?}

\smallskip

Such questions go back at least to Boltzmann \cite{Bol96,Bol98a,Bol98b}, Schr{\"o}dinger \cite{Sch31, Sch32} and Kolmogorov \cite{Kol37}; they were dealt with systematically by Nelson \cite{Nel01} (see also Carlen \cite{Car84}) in the context of Nelson's dynamical theories for Brownian motion and diffusion. There is now a rather complete theory that answers these questions and provides, as a kind of ``bonus'', some rather unexpected results as well. It was developed in the context of theories of filtering, interpolation and extrapolation, where such issues arise naturally --- most notably Haussmann and Pardoux \cite{HP86}, and Pardoux \cite{Par86}. Very interesting related results in a non-Markovian context, but with dispersion structure given by the identity matrix, have been obtained by F{\"o}llmer \cite{Foe85,Foe86}. Let us refer also to the papers \cite{Nag64,NM79} dealing with time reversal of Markov processes, and to the book \cite{Nag93} on diffusion theory. In what follows, this theory is presented in the spirit of the expository paper by Meyer \cite{Mey94}. 


\subsection{The setting} \label{ssts}


We place ourselves on a filtered probability space $(\Omega,\mathcal{F},\mathds{P})$, $\mathds{F} = (\mathcal{F}(t))_{0 \leqslant t \leqslant T}$ rich enough to support an $\mathds{R}^{d}$-valued Brownian motion $W = ( W_{1}, \ldots, W_{d})'$ adapted to $\mathds{F}$, as well as an independent $\mathcal{F}(0)$-measurable random vector $\xi = (\xi_{1}, \ldots, \xi_{n})' \colon \Omega \rightarrow \mathds{R}^{n}$. In fact, we shall assume that $\mathds{F}$ \textit{is} the filtration generated by these two objects, in the sense that we shall take 
\[
\mathcal{F}(t) = \sigma \big( \xi, W(\theta) \colon \, 0 \leqslant \theta \leqslant t \big), \qquad 0 \leqslant t \leqslant T,
\]
modulo $\mathds{P}$-augmentation. Next, we assume that the system of stochastic equations
\begin{equation} \label{2}
X_{i}(t) = \xi_{i} + \int_{0}^{t} b_{i}\big(\theta,X(\theta)\big) \, \textnormal{d}\theta + \sum_{\nu = 1}^{d} \int_{0}^{t} \mathrm{s}_{i \nu} \big(\theta, X(\theta)\big) \, \textnormal{d}W_{\nu}(\theta), \qquad 0 \leqslant t \leqslant T,
\end{equation}
for $i = 1, \ldots, n$ admits a pathwise unique, strong solution. It is then well known that the resulting continuous process $X = ( X_{1}, \ldots, X_{n} )'$ is $\mathds{F}$-adapted (the strong solvability of the equation \hyperref[2]{(\ref*{2})}), which implies that we have also 
\begin{equation} \label{2a}
\mathcal{F}(t) = \sigma \big( X(\theta), W(\theta) \colon \, 0 \leqslant \theta \leqslant t \big) = \sigma \big( X(0), W(t) - W(\theta) \colon \, 0 \leqslant \theta \leqslant t \big) 
\end{equation}
modulo $\mathds{P}$-augmentation, for $0 \leqslant t \leqslant T$; as well as that $X$ has the strong Markov property, and is thus a diffusion process with drifts $b_{i}( \, \cdot \, , \, \cdot \, )$ and dispersions $\mathrm{s}_{i \nu}( \, \cdot \, , \, \cdot \, )$, $i= 1, \ldots, n$, $\nu = 1, \ldots, d$. We shall denote the $(i,j)^{\textnormal{th}}$ entry of the covariance matrix $\mathrm{a}(t,x) \vcentcolon = \mathrm{s}(t,x) \, \mathrm{s}'(t,x)$ by
\[
\mathrm{a}_{ij}(t,x) \vcentcolon = \sum_{\nu=1}^{d} \mathrm{s}_{i\nu}(t,x) \, \mathrm{s}_{j\nu}(t,x), \qquad 1 \leqslant i,j \leqslant n.
\]

These characteristics are given mappings from $[0,T] \times \mathds{R}^{n}$ into $\mathds{R}$ with sufficient smoothness; in particular, such that the probability density function $p(t, \, \cdot \,) \colon \mathds{R}^{n} \rightarrow (0,\infty)$ in
\[
\mathds{P}\big[ X(t) \in A \big] = \int_{A} p(t,x) \, \textnormal{d}x, \qquad A \in \mathscr{B}(\mathds{R}^{n}),
\]
is smooth. Sufficient conditions on the drift $b_{i}( \, \cdot \, , \, \cdot \, )$ and dispersion $\mathrm{s}_{i \nu}( \, \cdot \, , \, \cdot \, )$ characteristics that lead to such smoothness, are provided by the H{\"o}rmander hypoellipticity conditions; see for instance \cite{Bel95}, \cite{Nua06} for this result, as well as \cite{Rog85} for a very simple argument in the one-dimensional case ($n=d=1$), and to the case of Langevin type equation \hyperref[sdeids]{(\ref*{sdeids})} for arbitrary $n \in \mathds{N}$. We refer to \cite{Fri75}, \cite{RW00b} or \cite{KS98} for the basics of the theory of stochastic equations of the form \hyperref[2]{(\ref*{2})}.

The probability density function $p(t, \, \cdot \,) \colon \mathds{R}^{n} \rightarrow (0,\infty)$ solves the forward Kolmogorov \cite{Kol31} equation \cite[p.\ 149]{Fri75}
\begin{equation} \label{FPE}
\partial_{t} p(t,x) = \frac{1}{2} \sum_{i,j=1}^{n} D_{ij}^{2} \big( \mathrm{a}_{ij}(t,x) \, p(t,x) \big) - \sum_{i=1}^{n} D_{i} \big( b_{i}(t,x) \, p(t,x) \big), \qquad (t,x) \in (0, T] \times \mathds{R}^{n}.
\end{equation}
If the drift and dispersion characteristics do not depend on time, and an invariant probability measure exists for the diffusion process of \hyperref[2]{(\ref*{2})}, the density function $p( \, \cdot \,)$ of this measure solves the stationary version of this forward Kolmogorov equation, to wit
\begin{equation} \label{FPK}
\frac{1}{2} \sum_{i,j=1}^{n} D_{ij}^{2} \big( \mathrm{a}_{ij}(x) \, p(x) \big) = \sum_{i=1}^{n} D_{i} \big( b_{i}(x) \, p(x) \big), \qquad x \in \mathds{R}^{n}.
\end{equation}


\subsection{Time reversal and the backwards filtration}


Consider now the family of $\sigma$-algebras $(\widehat{\mathcal{F}}(t))_{0 \leqslant t \leqslant T}$ given by
\begin{equation} \label{3}
\widehat{\mathcal{F}}(t) \vcentcolon = \sigma \big( X(\theta), W(\theta) - W(t) \colon \, t \leqslant \theta \leqslant T\big), \qquad 0 \leqslant t \leqslant T.
\end{equation}
It is not hard to see that the $\sigma$-algebra in \hyperref[3]{(\ref*{3})} is expressed equivalently as
\begingroup
\addtolength{\jot}{0.7em}
\begin{align}
\widehat{\mathcal{F}}(t) 
&= \sigma \big( X(t), W(\theta) - W(t) \colon \, t \leqslant \theta \leqslant T \big) = \sigma \big( X(t), W(\theta) - W(T) \colon \, t \leqslant \theta \leqslant T\big) \nonumber \\
&= \sigma \big( X(T), W(\theta) - W(t) \colon \, t \leqslant \theta \leqslant T \big) = \sigma \big(X(T)\big) \vee \mathcal{H}(t). \label{3.a}
\end{align}
\endgroup
Here, the $\sigma$-algebra generated by the Brownian increments after time $t$, namely,
\begin{equation} \label{4}
\mathcal{H}(t) \vcentcolon = \sigma \big( W(\theta) - W(t) \colon \, t \leqslant \theta \leqslant T \big), \qquad 0 \leqslant t \leqslant T,
\end{equation}
is independent of the random vector $X(t)$. The time-reversed processes $\widehat{X}$ as in \hyperref[1]{(\ref*{1})}, as well as 
\begin{equation} \label{5}
\widetilde{W}(s) \vcentcolon =  W(T-s) - W(T), \qquad 0 \leqslant s \leqslant T,
\end{equation}
are both adapted to the \textit{backwards filtration} $\widehat{\mathds{F}} \vcentcolon = \big(\widehat{\mathcal{F}}(T-s)\big)_{0 \leqslant s \leqslant T}$,  where 
\begingroup
\addtolength{\jot}{0.7em}
\begin{equation} \label{bffhgna}
\begin{aligned} 
\widehat{\mathcal{F}}(T-s) 
&= \sigma \big( X(T-u), W(T-u) - W(T-s) \colon \, 0 \leqslant u \leqslant s \big) \\
&= \sigma \big( \widehat{X}(u), \widetilde{W}(u) - \widetilde{W}(s) \colon \, 0 \leqslant u \leqslant s \big) 
\end{aligned}
\end{equation}
\endgroup
from \hyperref[3]{(\ref*{3})}. Note that, by complete analogy with \hyperref[2a]{(\ref*{2a})}, we have also
\begin{equation} \label{5.aaa} 
\widehat{\mathcal{F}}(T-s) = \sigma \big( X(T), W(T-u) - W(T-s) \colon \, 0 \leqslant u \leqslant s \big) = \sigma \big( \widehat{X}(0) \big) \vee  \mathcal{H}(T-s)   
\end{equation}
on account of \hyperref[3.a]{(\ref*{3.a})}, where  
\begin{equation} \label{5.aa} 
\mathcal{H}(T-s) = \sigma \big( W(T-u) - W(T-s) \colon \, 0 \leqslant u \leqslant s \big) = \sigma \big( \widetilde{W}(u) - \widetilde{W}(s) \colon \, 0 \leqslant u \leqslant s \big).
\end{equation}
In words: the $\sigma$-algebra $\widehat{\mathcal{F}}(T-s)$ is generated by the terminal value $X(T)$ of the forward process (i.e., by the original value $\widehat{X}(0)$ of the backward process) and by the increments of the time-reversed process $\widetilde{W}$ on $[0,s]$; see the expressions right above. Furthermore, the $\sigma$-algebra $\widehat{\mathcal{F}}(T-s)$ measures all the random variables $\widehat{X}(u)$, $u \in [0,s]$.  

\begin{remark} In fact, the time-reversed process $\widetilde{W}$ \textit{is a Brownian motion of the backwards filtration} $\mathds{H} \vcentcolon = (\mathcal{H}(T-s))_{0 \leqslant s \leqslant T} \subseteq \widehat{\mathds{F}}$ as in \hyperref[5.aa]{(\ref*{5.aa})}, generated by the increments of $W$ after time $T-s$, $0 \leqslant s \leqslant T$. This is because it is a martingale of this filtration, has continuous paths, and its quadratic variation is that of Brownian motion (L{\'e}vy's theorem \cite[Theorem 5.1]{KS98}). In the next subsection we shall see that the process $\widetilde{W}$ \textit{is only a semimartingale of the larger backwards filtration} $\widehat{\mathds{F}} = \big(\widehat{\mathcal{F}}(T-s)\big)_{0 \leqslant s \leqslant T}$, and identify its semimartingale decomposition. 
\end{remark}


\subsection{Some remarkable Brownian motions}


Following the exposition and ideas in \cite{Mey94}, we start with a couple of observations. First, for every $t \in [0,T]$ and every integrable, $\widehat{\mathcal{F}}(t)$-measurable random variable $\mathcal{K}$, we have
\begin{equation} \label{ps}
\mathds{E}\big[\mathcal{K} \, \vert \, \mathcal{F}(t)\big] = \mathds{E}\big[ \mathcal{K} \, \vert \, X(t)\big], \quad \textnormal{almost surely.}
\end{equation}
Secondly, we fix a function $G \in \mathcal{C}_{0}^{\infty}(\mathds{R}^{n})$ and a time-point $t \in (0,T]$, and define
\[
g(\theta,x) \vcentcolon = \mathds{E}\big[ G\big(X(t)\big) \, \vert \, X(\theta) = x \big], \qquad (\theta,x) \in [0,t] \times \mathds{R}^{n}.
\]
Invoking the Markov property of $X$, we deduce that the process 
\[
g\big(\theta,X(\theta)\big) = \mathds{E}\big[ G\big(X(t)\big) \,  \vert \, X(\theta)\big] =  \mathds{E}\big[ G\big(X(t)\big) \, \vert \, \mathcal{F}(\theta)\big], \qquad 0 \leqslant \theta \leqslant t
\]
is an $\mathds{F}$-martingale, and obtain
\[
G\big(X(t)\big) - g\big(\theta,X(\theta)\big) = g\big(t,X(t)\big) - g\big(\theta,X(\theta)\big) = \sum_{i=1}^{n} \sum_{\nu=1}^{d} \int_{\theta}^{t} D_{i} g\big(v,X(v)\big) \, \mathrm{s}_{i\nu}\big(v,X(v)\big) \, \textnormal{d} W_{\nu}(v).
\]
For every index $\nu = 1, \ldots, d$ this gives, after integrating by parts, 
\begingroup
\addtolength{\jot}{0.7em}
\begin{align*}
&\mathds{E}\big[ \big( W_{\nu}(t) - W_{\nu}(\theta)\big) \cdot G\big(X(t)\big) \big] = \mathds{E}\Big[ \big( W_{\nu}(t) - W_{\nu}(\theta) \big) \cdot \Big( g\big(t,X(t)\big) - g\big(\theta,X(\theta)\big) \Big) \Big] \\
& \, = \mathds{E}\bigg[\sum_{i=1}^{n}  \int_{\theta}^{t}  D_{i} g\big(v,X(v)\big) \, \mathrm{s}_{i\nu}\big(v,X(v)\big) \, \textnormal{d}v\bigg] 
= \sum_{i=1}^{n} \int_{\theta}^{t} \int_{\mathds{R}^{n}} \big( D_{i} g \cdot \mathrm{s}_{i\nu} \big)(v,x) \, p(v,x) \, \textnormal{d}x \, \textnormal{d}v \\
& \, = - \sum_{i=1}^{n}  \int_{\theta}^{t} \int_{\mathds{R}^{n}} g(v,x) \, D_{i} \big( p(v,x) \, \mathrm{s}_{i \nu}(v,x) \big) \, \textnormal{d}x\, \textnormal{d}v 
= -  \int_{\theta}^{t} \int_{\mathds{R}^{n}} g(v,x) \, \operatorname{div}\big(p(v,x) \, \overline{\mathrm{s}}_{\nu}(v,x)\big) \, \textnormal{d}x \, \textnormal{d}v \\
& \, = - \int_{\theta}^{t} \mathds{E} \bigg[ g\big(v,X(v)\big) \cdot \frac{\operatorname{div}(p \, \overline{\mathrm{s}}_{\nu})}{p}\big(v,X(v)\big) \bigg] \, \textnormal{d}v 
= - \mathds{E} \bigg[ G\big(X(t)\big) \cdot  \int_{\theta}^{t} \frac{\operatorname{div}(p \, \overline{\mathrm{s}}_{\nu})}{p}\big(v,X(v)\big) \, \textnormal{d}v \bigg].
\end{align*}
\endgroup
Here $\overline{\mathrm{s}}_{\nu}(v, \, \cdot \, )$ is the $\nu^{\textnormal{th}}$ column vector of the dispersion matrix. Comparing the first and last expressions in the above string of equalities, we see that with $0 \leqslant \theta \leqslant t$ we have 
\begin{equation} \label{5.a}
\mathds{E} \bigg[ G\big(X(t)\big) \cdot  \bigg( W_{\nu}(t) - W_{\nu}(\theta) + \int_{\theta}^{t} \frac{\operatorname{div}(p \, \overline{\mathrm{s}}_{\nu})}{p}\big(v,X(v)\big) \, \textnormal{d}v \bigg) \bigg] = 0
\end{equation}
for every $G \in \mathcal{C}_{0}^{\infty}(\mathds{R}^{n})$, and thus by extension for every bounded, measurable $G \colon \mathds{R}^{n} \rightarrow \mathds{R}$.

\medskip

\begin{theorem} \label{Thm1} The vector process $B = ( B_{1}, \ldots, B_{d})'$ defined as
\begingroup
\addtolength{\jot}{0.7em}
\begin{align} 
B_{\nu}(s) \vcentcolon =& \, \widetilde{W}_{\nu}(s) - \int_{0}^{s} \frac{\operatorname{div}(p \, \overline{\mathrm{s}}_{\nu})}{p}\big(T-u,\widehat{X}(u)\big) \, \textnormal{d} u \label{6} \\
=& \, W_{\nu}(T-s) - W_{\nu}(T) - \int_{T-s}^{T} \frac{\operatorname{div}( p \, \overline{\mathrm{s}}_{\nu} )}{p}\big(v,X(v)\big) \, \textnormal{d} v, \qquad 0 \leqslant s \leqslant T, \label{6aa}
\end{align}
\endgroup
for $\nu = 1, \ldots, d$, is a Brownian motion of the backwards filtration $\widehat{\mathds{F}} = \big(\widehat{\mathcal{F}}(T-s)\big)_{0 \leqslant s \leqslant T}$.
\end{theorem}

\begin{remark} The Brownian motion process $B$ is thus independent of $\widehat{\mathcal{F}}(T)$, and therefore also of the $\widehat{\mathcal{F}}(T)$-measurable random variable $X(T)$. A bit more generally,
\[
\big\{ B(T-\theta) - B(T-t) \colon \, 0 \leqslant \theta \leqslant t \big\} \quad \textnormal{ is independent of } \quad \widehat{\mathcal{F}}(t) \supseteq \sigma \big( X(v) \colon \, t \leqslant v \leqslant T \big).
\]
Note also from \hyperref[6aa]{(\ref*{6aa})} that
\begin{equation*} 
B_{\nu}(T-\theta) - B_{\nu}(T-t) = W_{\nu}(\theta) - W_{\nu}(t) - \int_{\theta}^{t} \frac{\operatorname{div}(p \, \overline{\mathrm{s}}_{\nu})}{p}\big(v,X(v)\big) \, \textnormal{d}v, \qquad 0 \leqslant \theta \leqslant t.
\end{equation*}
\end{remark}

Reversing time once again, we obtain the following corollary of \hyperref[Thm1]{Theorem \ref*{Thm1}}.

\begin{corollary} The $\mathds{F}$-adapted vector process $V = ( V_{1}, \ldots, V_{d})'$ with components 
\begin{equation}
V_{\nu}(t) \vcentcolon = B_{\nu}(T-t) - B_{\nu}(T) =  W_{\nu}(t) + \int_{0}^{t} \frac{\operatorname{div}(p \, \overline{\mathrm{s}}_{\nu})}{p}\big(v, X(v)\big) \, \textnormal{d}v, \qquad 0 \leqslant t \leqslant T, \label{8}
\end{equation}
for $\nu = 1, \ldots, d$, is yet another Brownian motion \textnormal{(}with respect to its own filtration $\mathds{F}^{V} \subseteq \mathds{F}$\textnormal{)}. This process is independent of the random variable $X(T)$; and a bit more generally, for every $t \in (0,T]$, the $\sigma$-algebra 
\begin{equation} \label{8aa}
\mathcal{F}^{V}(t) \vcentcolon = \sigma\big( V(\theta) \colon \, 0 \leqslant \theta \leqslant t \big)
\end{equation}
generated by present-and-past values of $V$, is independent of $\sigma(X(v) \colon \, t \leqslant v \leqslant T)$, the $\sigma$-algebra generated by present-and-future values of $X$.
\end{corollary}

\begin{proof}[Proof of \texorpdfstring{\hyperref[Thm1]{Theorem \ref*{Thm1}}}] It suffices to show that each component process $B_{\nu}$ is a martingale of the backwards filtration $\widehat{\mathds{F}}$; because then, in view of the continuity of paths and the easily checked  property $\langle B_{\nu}, B_{\ell} \rangle(s) = s \, \delta_{\nu \ell}$, we can deduce that each $B_{\nu}$ is a Brownian motion in the backwards filtration $\widehat{\mathds{F}}$ (and of course also in its own filtration), and that $B_{\nu}, B_{\ell}$ are independent for $\ell \neq \nu$, by appealing to L{\'e}vy's theorem once again. 

\smallskip

Now we have to show $\mathds{E}\big[ \big(B_{\nu}(T-\theta) - B_{\nu}(T-t) \big) \cdot \mathcal{K} \big] = 0$ for $0 \leqslant \theta \leqslant t \leqslant T$ and every bounded, $\widehat{\mathcal{F}}(t)$-measurable $\mathcal{K}$; equivalently, 
\[
\mathds{E}\bigg[ \mathds{E}\big[ \mathcal{K} \, \vert \, \mathcal{F}(t) \big] \cdot \bigg( W_{\nu}(t) - W_{\nu}(\theta) + \int_{\theta}^{t} \frac{\operatorname{div}(p \, \overline{\mathrm{s}}_{\nu})}{p}\big(v,X(v)\big) \, \textnormal{d}v \bigg) \bigg] = 0,
\]
as the expression inside the curved braces is $\mathcal{F}(t)$-measurable. But recalling \hyperref[ps]{(\ref*{ps})} we have that $\mathds{E}[ \mathcal{K}  \, \vert \, \mathcal{F}(t)] = \mathds{E}[ \mathcal{K} \, \vert \, X(t)] = G(X(t))$ for some bounded, measurable $G \colon \mathds{R}^{n} \rightarrow \mathds{R}$, and the desired result follows from \hyperref[5.a]{(\ref*{5.a})}.
\end{proof}


\subsection{The diffusion property under time reversal} \label{tdputr}


Let us return now to the question, whether the time-reversed process $\widehat{X}$ of \hyperref[1]{(\ref*{1})}, \hyperref[2]{(\ref*{2})} is a diffusion. We start by expressing $X_{i}$ of \hyperref[2]{(\ref*{2})} in terms of a backwards It\^{o} integral (see \hyperref[substbii]{Subsection \ref*{substbii}}) as
\begingroup
\addtolength{\jot}{0.7em}
\begin{align*}
X_{i}(t) - \xi_{i} - \int_{0}^{t} b_{i}\big(\theta,X(\theta)\big) \, \textnormal{d}\theta
&= \sum_{\nu=1}^{d} \int_{0}^{t} \mathrm{s}_{i\nu} \big(\theta,X(\theta)\big) \, \textnormal{d} W_{\nu}(\theta) \\
&= \sum_{\nu=1}^{d} \bigg( \int_{0}^{t} \mathrm{s}_{i\nu}\big(\theta,X(\theta)\big) \bullet \textnormal{d} W_{\nu}(\theta) - \big\langle \mathrm{s}_{i\nu}( \, \cdot \, , X), W_{\nu} \big\rangle(t) \bigg).
\end{align*}
\endgroup
From \hyperref[2]{(\ref*{2})}, we have by It\^{o}'s formula that the process 
\[
\mathrm{s}_{i\nu}( \, \cdot \, , X) - \mathrm{s}_{i\nu}(0,\xi) - \sum_{j=1}^{n} \sum_{\kappa=1}^{d} \int_{0}^{\cdot} D_{j} \mathrm{s}_{i\nu}\big(\theta,X(\theta)\big) \cdot \mathrm{s}_{j\kappa}\big(\theta,X(\theta)\big) \, \textnormal{d} W_{\kappa}(\theta) 
\]
is of finite first variation, therefore $\displaystyle \big\langle \mathrm{s}_{i\nu}(\, \cdot \, , X) ,  W_{\nu} \big\rangle(t) = \sum_{j=1}^{n} \int_{0}^{t}  \mathrm{s}_{j\nu}\big(\theta,X(\theta)\big) \, D_j \mathrm{s}_{i\nu}\big(\theta,X(\theta)\big) \, \textnormal{d}\theta$. We conclude
\[
X_{i}(t) = \xi_{i} - \int_{0}^{t} \bigg( \sum_{j=1}^{n}  \sum_{\nu=1}^{d} \mathrm{s}_{j\nu} \, D_{j} \mathrm{s}_{i\nu} - b_{i} \bigg) \big(\theta,X(\theta)\big) \, \textnormal{d}\theta
+ \sum_{\nu=1}^{d} \int_{0}^{t} \mathrm{s}_{i\nu} \big(\theta,X(\theta)\big) \bullet \textnormal{d}  W_{\nu}(\theta).
\]
Evaluating also at $t=T$, then subtracting, we obtain
\[
X_{i}(t) = X_{i}(T) + \int_{t}^{T} \bigg( \sum_{j=1}^{n} \sum_{\nu=1}^{d} \mathrm{s}_{j\nu} \, D_{j} \mathrm{s}_{i\nu} - b_{i} \bigg) \big(\theta,X(\theta)\big) \, \textnormal{d}\theta
- \sum_{\nu=1}^{d} \int_{t}^{T} \mathrm{s}_{i\nu} \big(\theta,X(\theta)\big ) \bullet \textnormal{d}W_{\nu}(\theta),
\]
as well as
\[
\widehat{X}_{i}(s) = \widehat{X}_{i}(0) + \int_{0}^{s} \bigg( \sum_{j=1}^{n}  \sum_{\nu=1}^{d} \mathrm{s}_{j\nu} \, D_j \mathrm{s}_{i\nu} - b_{i} \bigg) \big(T-u,\widehat{X}(u)\big) \, \textnormal{d}u 
+  \sum_{\nu=1}^{d} \int_{0}^{s} \mathrm{s}_{i\nu} \big(T-u,\widehat{X}(u)\big) \, \textnormal{d} \widetilde{W}_{\nu}(u)
\]
by reversing time. It is important here to note that the backward It\^{o} integral for $W$ becomes a forward It\^{o} integral for the process $\widetilde{W}$, the time reversal of $W$ in the manner of \hyperref[5]{(\ref*{5})}.

But now let us recall \hyperref[6]{(\ref*{6})}, on the strength of which the above expression takes the form
\begingroup
\addtolength{\jot}{0.7em}
\begin{align*}
&\widehat{X}_{i}(s) =  \widehat{X}_{i}(0) + \sum_{\nu=1}^{d} \int_{0}^{s} \mathrm{s}_{i\nu} \big(T-u, \widehat{X}(u)\big) \, \textnormal{d} B_{\nu}(u) \\
& \qquad \qquad + \int_{0}^{s} \bigg( \sum_{j=1}^{n}  \sum_{\nu=1}^{d} \mathrm{s}_{j\nu} \, D_{j} \mathrm{s}_{i\nu} + \sum_{\nu=1}^{d} \mathrm{s}_{i\nu} \, \frac{\operatorname{div}(p \, \overline{\mathrm{s}}_{\nu})}{p} - b_{i}\bigg) \big(T-u, \widehat{X}(u) \big) \, \textnormal{d}u, \qquad 0 \leqslant s \leqslant T.
\end{align*}
\endgroup
But in conjunction with \hyperref[Thm1]{Theorem \ref*{Thm1}}, this means that the time-reversed process $\widehat{X}$ of \hyperref[1]{(\ref*{1})}, \hyperref[2]{(\ref*{2})} is a semimartingale of the backwards filtration $\widehat{\mathds{F}} = \big(\widehat{\mathcal{F}}(T-s)\big)_{0 \leqslant s \leqslant T}$, with decomposition
\begin{equation} \label{10}
\widehat{X}_{i}(s) = \widehat{X}_{i}(0) + \int_{0}^{s} \widehat{b}_{i} \big(T-u,\widehat{X}(u)\big) \, \textnormal{d}u +  \sum_{\nu=1}^{d} \int_{0}^{s} \mathrm{s}_{i\nu} \big(T-u,\widehat{X}(u)\big) \, \textnormal{d} B_{\nu}(u)
\end{equation}
for $0 \leqslant s \leqslant T$, where, for each $i = 1, \ldots, n$, the function $\widehat{b}_{i}( \, \cdot \, , \, \cdot \, )$ is specified by
\begingroup
\addtolength{\jot}{0.7em}
\begin{align*}
\widehat{b}_{i}(t,x) + b_{i}(t,x)
&= \sum_{j=1}^{n}  \sum_{\nu=1}^{d} \mathrm{s}_{j\nu}(t,x) \, D_{j} \mathrm{s}_{i\nu}(t,x) + \sum_{\nu=1}^{d} \mathrm{s}_{i\nu}(t,x) \, \frac{\operatorname{div}\big(p(t,x) \, \overline{\mathrm{s}}_{\nu}(t,x)\big)}{p(t,x)} \\
&= \sum_{j=1}^{n}  \sum_{\nu=1}^{d} \mathrm{s}_{j\nu}(t,x) \, D_{j} \mathrm{s}_{i\nu}(t,x) + \sum_{\nu=1}^{d} \frac{\mathrm{s}_{i\nu}(t,x)}{p(t,x)}  \bigg( \sum_{j=1}^{n} D_{j} \big(p(t,x) \, \mathrm{s}_{j\nu}(t,x)  \big)\bigg) \\
&= \sum_{j=1}^{n}  \big( D_{j} \mathrm{a}_{ij} (t,x) + \mathrm{a}_{ij}(t,x) \cdot  D_{j} \log p(t,x) \big).
\end{align*}
\endgroup

\begin{theorem} \label{Thm2} Under the assumptions of this section, the time-reversed process $\widehat{X}$ of \textnormal{\hyperref[1]{(\ref*{1})}}, \textnormal{\hyperref[2]{(\ref*{2})}} is a diffusion in the backwards filtration $\widehat{\mathds{F}} = \big(\widehat{\mathcal{F}}(T-s)\big)_{0 \leqslant s \leqslant T}$, with characteristics as in \textnormal{\hyperref[10]{(\ref*{10})}}, namely, dispersions $\mathrm{s}_{i\nu}(T-s,x)$ and drifts $\widehat{b}_{i}(T-s,x)$ given by the \textnormal{\textsf{generalized Nelson equation}}
\begin{equation} \label{11}
\widehat{b}_{i}(t,x) + b_{i}(t,x) = \sum_{j=1}^{n} \Big( D_{j} \mathrm{a}_{ij}(t,x) + \mathrm{a}_{ij}(t,x) \cdot D_{j} \log p(t,x) \Big), \qquad i = 1, \ldots, n.
\end{equation}
\end{theorem}

\noindent Equivalently, and with $\operatorname{div}\big(\mathrm{a}(t,x)\big) \vcentcolon = \big( \sum_{j=1}^{n} D_{j} \mathrm{a}_{ij}(t,x)\big)_{1 \leqslant i \leqslant n}$, we write
\begin{equation} \label{11a}
\widehat{b}(t,x) + b(t,x) = \mathrm{div} \big( \mathrm{a}(t,x) \big) + \mathrm{a}(t,x) \cdot  \nabla \log p(t,x).
\end{equation}

\begin{remark} This result can be extended to the case where  the sums of the distributional derivatives $\sum_{j=1}^{n} D_{j} \big( \mathrm{a}_{i j} (t,x) \, p(t,x) \big)$, $i = 1, \ldots, n$, are only assumed to be locally integrable functions of $x \in \mathds{R}^{n}$; see \cite{MNS89,RVW01}. 
\end{remark}

\begin{remark}[\textsf{Some filtration comparisons}] 
For an invertible dispersion matrix $\mathrm{s}( \, \cdot \, , \, \cdot \, )$, it follows from \hyperref[10]{(\ref*{10})} that the Brownian motion $B$ is adapted to the filtration generated by $\widehat{X}$; that is,  
\begin{equation} \label{101}
\mathcal{F}^{B}(s) \subseteq \mathcal{F}^{\widehat{X}}(s) \vcentcolon = \sigma\big(  \widehat{X}(u) \colon \, 0 \leqslant u \leqslant s \big) = \sigma\big( X(T-u) \colon \, 0 \leqslant u \leqslant s \big), \qquad 0 \leqslant s \leqslant T.
\end{equation}
Now recall \hyperref[6]{(\ref*{6})}; in its light, the filtration comparison in \hyperref[101]{(\ref*{101})} implies $\mathcal{F}^{\widetilde{W}}(s) \subseteq \mathcal{F}^{\widehat{X}}(s)$, thus $\mathcal{H}(T-s) \subseteq \mathcal{F}^{\widetilde{W}}(s) \subseteq \mathcal{F}^{\widehat{X}}(s)$ from \hyperref[5.aa]{(\ref*{5.aa})}, for $0 \leqslant s \leqslant T$, and from \hyperref[5.aaa]{(\ref*{5.aaa})} also 
\begin{equation} \label{1012}
\widehat{\mathcal{F}}(T-s) \subseteq \mathcal{F}^{\widehat{X}}(s), \qquad 0 \leqslant s \leqslant T.
\end{equation}
But we have also the reverse inclusion $\mathcal{F}^{\widehat{X}}(s) \subseteq \widehat{\mathcal{F}}(T-s)$ on account of \hyperref[bffhgna]{(\ref*{bffhgna})} and \hyperref[101]{(\ref*{101})}; therefore, $\mathcal{F}^{\widehat{X}}(s) = \widehat{\mathcal{F}}(T-s)$ holds for all $0 \leqslant s \leqslant T$ when $s( \, \cdot \, , \, \cdot \,)$ is invertible. These considerations inform our choice of backwards filtration $\mathcal{G}(T-s) \equiv \mathcal{F}^{\widehat{X}}(s)$, $0 \leqslant s \leqslant T$, in \hyperref[tbfmgtmt]{(\ref*{tbfmgtmt})}.
\end{remark}


\subsection{The backwards It\^{o} integral} \label{substbii}


For two continuous semimartingales $X = X(0) + M  + B$ and $Y = Y(0) + N + C$, with $B,C$ continuous adapted processes of finite variation and $M,N$ continuous local martingales, let us recall the definition of the Fisk-Stratonovich integral in \cite[Definition 3.3.13, p.\ 156]{KS98}, as well as its properties in \cite[Problem 3.3.14]{KS98} and \cite[Problem 3.3.15]{KS98}.

By analogy with this definition, we introduce the \textit{backwards It\^{o} integral}
\begin{equation}
\int_{0}^{\cdot} Y(t) \bullet \textnormal{d} X(t) \vcentcolon = \int_{0}^{\cdot} Y(t) \, \textnormal{d} M(t) + \int_{0}^{\cdot} Y(t) \, \textnormal{d} B(t) + \langle M, N \rangle,
\end{equation}
where the first (respectively, the second) integral on the right-hand side is to be interpreted in the It\^{o} (respectively, the Lebesgue-Stieltjes) sense.

\smallskip

If $\Pi = \{ t_{0}, t_{1}, \ldots, t_{m}\}$ is a partition of the interval $[0,T]$ with $0 = t_{0} < t_{1} < \ldots < t_{m} = T$, then the sums
\begin{equation}
\sum_{j=0}^{m-1}  Y(t_{j+1}) \big( X(t_{j+1}) - X(t_{j}) \big)
\end{equation}
converge in probability to $\int_{0}^{T} Y(t) \bullet \textnormal{d}X(t)$ as the mesh $\|\Pi \|$ of the partition tends to zero. Note that the increments of $X$ here ``stick backwards into the past'', as opposed to ``sticking forward into the future'' as in the It\^{o} integral.

\smallskip

For the backwards It\^{o} integral we have the change of variable formula
\begin{equation}
f(X) = f\big(X(0)\big) + \sum_{i=1}^{n} \int_{0}^{\cdot} D_{i} f\big(X(t)\big) \bullet \textnormal{d} X_{i}(t) - \frac{1}{2} \sum_{i,j=1}^{n} \int_{0}^{\cdot} D^{2}_{ij} f\big(X(t)\big) \, \textnormal{d} \langle M_{i}, M_{j} \rangle (t),
\end{equation}
where now $X = ( X_{1}, \ldots, X_{n})'$ is a vector of continuous semimartingales $X_{1}, \ldots, X_{n}$ of the form $X_{i} = X_{i}(0) + M_{i} + B_{i}$ as above, for $i = 1, \ldots, n$. Note the change of sign, from $(+)$ to $(-)$ in the last, stochastic correction term.
\end{appendices}


\newpage

\bibliographystyle{alpha}
{\footnotesize
\bibliography{references}}


\end{document}